\documentclass[a4paper,10pt]{article}

\usepackage{graphicx}
\usepackage[T1]{fontenc} 
\usepackage{enumerate}
\usepackage[utf8]{inputenc} 
\usepackage[english]{babel}
\usepackage[textwidth=30pc,textheight=54.5pc]{geometry} 
\usepackage[toc,page]{appendix}
\usepackage{amsmath,amsfonts,amssymb,amsthm,mathrsfs}
\usepackage[breaklinks,hidelinks]{hyperref}
\usepackage{paralist}
\usepackage{enumitem}
\usepackage{textcase} 
\usepackage{filemod} 
\usepackage[longnamesfirst]{natbib}
\usepackage{titlesec} 
\usepackage{etoolbox} 
\usepackage{mathtools} 
\usepackage{tikz-cd}
\usepackage{booktabs,multirow,array}
\usepackage{rotating}
\usepackage{multirow,tabularx}
\makeatletter

\usepackage{etoolbox,refcount}
\usepackage{multicol}

\newcounter{countitems}
\newcounter{nextitemizecount}
\newcommand{\setupcountitems}{%
  \stepcounter{nextitemizecount}%
  \setcounter{countitems}{0}%
  \preto\item{\stepcounter{countitems}}%
}
\makeatletter
\newcommand{\computecountitems}{%
  \edef\@currentlabel{\number\c@countitems}%
  \label{countitems@\number\numexpr\value{nextitemizecount}-1\relax}%
}
\newcommand{\nextitemizecount}{%
  \getrefnumber{countitems@\number\c@nextitemizecount}%
}
\newcommand{\previtemizecount}{%
  \getrefnumber{countitems@\number\numexpr\value{nextitemizecount}-1\relax}%
}

\newenvironment{AutoMultiColEnumerate}{%
\ifnumcomp{\nextitemizecount}{>}{3}{\begin{multicols}{2}}{}%
\setupcountitems\begin{enumerate}}%
{\end{enumerate}%
\unskip\computecountitems\ifnumcomp{\previtemizecount}{>}{3}{\end{multicols}}{}}

\def\cite{\citet}

\numberwithin{equation}{section}
\allowdisplaybreaks

\def\@noindentfalse{\global\let\if@noindent\iffalse}
\def\@noindenttrue {\global\let\if@noindent\iftrue}
\def\@aftertheorem{%
  \@noindenttrue
  \everypar{%
    \if@noindent%
      \@noindentfalse\clubpenalty\@M\setbox\z@\lastbox%
    \else%
      \clubpenalty \@clubpenalty\everypar{}%
    \fi}}

\theoremstyle{plain}
\newtheorem{theorem}{Theorem}[section]
\AfterEndEnvironment{theorem}{\@aftertheorem}

\AfterEndEnvironment{definition}{\@aftertheorem}
\newtheorem{lemma}[theorem]{Lemma}
\AfterEndEnvironment{lemma}{\@aftertheorem}

\AfterEndEnvironment{corollary}{\@aftertheorem}

\theoremstyle{definition}
\newtheorem{remark}[theorem]{Remark}
\AfterEndEnvironment{remark}{\@aftertheorem}
\newtheorem{assumption}[theorem]{Assumption}
\AfterEndEnvironment{assumption}{\@aftertheorem}

\AfterEndEnvironment{example}{\@aftertheorem}

\AfterEndEnvironment{proof}{\@aftertheorem}

\newtheorem{proposition}[theorem]{Proposition}

\titleformat{name=\section}
  {\center\small\sc}{\thesection\kern1em}{0pt}{\MakeTextUppercase}
\titleformat{name=\section, numberless}
  {\center\small\sc}{}{0pt}{\MakeTextUppercase}

\titleformat{name=\subsection}
  {\bf\mathversion{bold}}{\thesubsection\kern1em}{0pt}{}
\titleformat{name=\subsection, numberless}
  {\bf\mathversion{bold}}{}{0pt}{}

\titleformat{name=\subsubsection}
  {\normalsize\itshape}{\thesubsubsection\kern1em}{0pt}{}
\titleformat{name=\subsubsection, numberless}
  {\normalsize\itshape}{}{0pt}{}

\def\be#1{\begin{equation*}#1\end{equation*}}
\def\ben#1{\begin{equation}#1\end{equation}}
\def\bes#1{\begin{equation*}\begin{split}#1\end{split}\end{equation*}}
\def\besn#1{\begin{equation}\begin{split}#1\end{split}\end{equation}}

\def\ba#1{\begin{align*}#1\end{align*}}

\def\bg#1{\begin{gather*}#1\end{gather*}}
\def\bgn#1{\begin{gather}#1\end{gather}}
\usepackage[linewidth=1pt]{mdframed}
\usepackage[textwidth=1.65in]{todonotes}
\let\@@todo\todo
\def\todo#1{\@@todo[color=red,backgroundcolor=red!10,size=\tiny]{#1}}

\usepackage{delimset}
\def\given{\mskip 0.5mu plus 0.25mu\vert\mskip 0.5mu plus 0.15mu}
\newcounter{bracketlevel}%
\def\@bracketfactory#1#2#3#4#5#6{%
\expandafter\def\csname#1\endcsname##1{%
\global\advance\c@bracketlevel 1\relax%
\global\expandafter\let\csname @middummy\alph{bracketlevel}\endcsname\given%
\global\def\given{\mskip#5\csname#4\endcsname\vert\mskip#6}\csname#4l\endcsname#2##1\csname#4r\endcsname#3%
\global\expandafter\let\expandafter\given\csname @middummy\alph{bracketlevel}\endcsname%
\global\advance\c@bracketlevel -1\relax%
}%
}
\def\bracketfactory#1#2#3{%
\@bracketfactory{#1}{#2}{#3}{relax}{0.5mu plus 0.25mu}{0.5mu plus 0.15mu}
\@bracketfactory{b#1}{#2}{#3}{big}{1mu plus 0.25mu minus 0.25mu}{0.6mu plus 0.15mu minus 0.15mu}
\@bracketfactory{bb#1}{#2}{#3}{Big}{2.4mu plus 0.8mu minus 0.8mu}{1.8mu plus 0.6mu minus 0.6mu}
\@bracketfactory{bbb#1}{#2}{#3}{bigg}{3.2mu plus 1mu minus 1mu}{2.4mu plus 0.75mu minus 0.75mu}
\@bracketfactory{bbbb#1}{#2}{#3}{Bigg}{4mu plus 1mu minus 1mu}{3mu plus 0.75mu minus 0.75mu}
}
\bracketfactory{clc}{\lbrace}{\rbrace}
\bracketfactory{clr}{(}{)}
\bracketfactory{cls}{[}{]}
\bracketfactory{abs}{\lvert}{\rvert}
\bracketfactory{norm}{\Vert}{\Vert}
\bracketfactory{floor}{\lfloor}{\rfloor}
\bracketfactory{ceil}{\lceil}{\rceil}
\bracketfactory{angle}{\langle}{\rangle}

\newcounter{ctr}\loop\stepcounter{ctr}\edef\X{\@Alph\c@ctr}%
	\expandafter\edef\csname s\X\endcsname{\noexpand\mathscr{\X}}
	\expandafter\edef\csname c\X\endcsname{\noexpand\mathcal{\X}}
	\expandafter\edef\csname b\X\endcsname{\noexpand\boldsymbol{\X}}
	\expandafter\edef\csname I\X\endcsname{\noexpand\mathbb{\X}}
\ifnum\thectr<26\repeat

0\loop\stepcounter{ctr}\edef\X{\@alph\c@ctr}%
	\expandafter\edef\csname bs\X\endcsname{\noexpand\boldsymbol{\X}}
\ifnum\thectr<26\repeat

\def\bsl{l}
\def\bsi{i}
\def\bss{s}
\def\bsr{r}

\let\@IE\IE\let\IE\undefined
\newcommand{\IE}{\mathop{{}\@IE}\mathopen{}}
\let\@IP\IP\let\IP\undefined
\newcommand{\IP}{\mathop{{}\@IP}}

\newcommand{\Exp}{\mathop{\mathrm{Exp}}}
\newcommand{\Be}{\mathop{\mathrm{Be}}}
\newcommand{\Po}{\mathop{\mathrm{Po}}}
\newcommand{\Bi}{\mathop{\mathrm{Bi}}}
\newcommand{\Var}{\mathop{\mathrm{Var}}}

\newcommand{\bigo}{\mathop{{}\mathrm{O}}\mathopen{}}
\newcommand{\lito}{\mathop{{}\mathrm{o}}\mathopen{}}
\newcommand{\law}{\mathop{{}\sL}\mathopen{}}
\newcommand{\mel}{\MoveEqLeft}

\newcommand{\RomanNumeral}[1]{\lowercase\expandafter{\romannumeral#1}}

\def\^#1{\relax\ifmmode {\mathaccent"705E #1} \else {\accent94 #1}\fi}
\def\~#1{\relax\ifmmode {\mathaccent"707E #1} \else {\accent"7E #1}\fi}
\def\*#1{\relax#1^\ast}
\edef\-#1{\relax\noexpand\ifmmode {\noexpand\bar{#1}} \noexpand\else \-#1\noexpand\fi}
\def\>#1{\vec{#1}}
\def\.#1{\dot{#1}}

\def\atop{\@@atop}

\def\ER{Erd\H{o}s-R\'enyi}

\renewcommand{\leq}{\leqslant}
\renewcommand{\geq}{\geqslant}
\renewcommand{\phi}{\varphi}

\newcommand{\eq}{\eqref}
\newcommand{\I}{\mathop{{}\mathrm{I}}\mathopen{}}
\newcommand{\dtv}{\mathop{d_{\mathrm{TV}}}\mathopen{}}

\newcommand\indep{\protect\mathpalette{\protect\@indep}{\perp}}
\def\@indep#1#2{\mathrel{\rlap{$#1#2$}\mkern2mu{#1#2}}}

\newcommand{\trans}{{\top}}

\newcommand{\F}[1]{\mathfrak{#1}}
\newcommand{\C}[1]{\mathcal{#1}}

\def\parsetime#1#2#3#4#5#6{#1#2:#3#4}
\def\parsedate#1:20#2#3#4#5#6#7#8+#9\empty{20#2#3-#4#5-#6#7 \parsetime #8}
\def\moddate{\expandafter\parsedate\pdffilemoddate{\jobname.tex}\empty}

\makeatother

\begin{document}

\title{\sc\bf\large\MakeUppercase{The SIR epidemic on \\a dynamic Erd\H{o}s-R\'enyi random graph}}
\author{\sc Yuanfei Huang \and \sc Adrian R\"ollin}

\date{\itshape National University of Singapore}

\maketitle
\abstract{\noindent
We investigate the SIR epidemic on a dynamic inhomogeneous Erd\H{o}s-R\'enyi random graph, in which vertices are of one of $k$ types and in which edges appear and disappear independently of each other. We establish a functional law of large numbers for the susceptible, infected, and recovered ratio curves after a random time shift, and demonstrate that, under a variety of possible scaling limits of the model parameters, the epidemic curves are solutions to a system of ordinary differential equations. In most scaling regimes, these equations coincide with the classical SIR epidemic equations. In the regime where the average degree of the network remains constant and the edge-flipping dynamics remain on the same time scale as the infectious contact process, however, a novel set of differential equations emerges. This system contains  additional quantities related to the infectious edges, but somewhat surprisingly, contains no quantities related to higher-order local network configurations. To the best of our knowledge, this study represents the first thorough and rigorous analysis of large population epidemic processes on dynamic random graphs, although our findings are contingent upon conditioning on a (possibly strict) subset of the event of an epidemic outbreak. }

\medskip


\tableofcontents

\section{Introduction}

Rigorous mathematical modelling of epidemics began with \cite{Kermack1927}, who employed ordinary differential equations to capture the evolution of infectious diseases in large populations. While individual infection events are inherently probabilistic, the deterministic models proposed by \cite{Kermack1927} effectively characterize the disease's progression once an epidemic is well underway. In contrast, the exploration of stochastic effects, which are especially important during the onset of an epidemic, can be attributed to \cite{bartlett1949some}. As is well known, these two modelling approaches are intrinsically related. Indeed, as population sizes grow large, stochastic models tend to converge to their deterministic counterparts; hence, both provide complementary perspectives on the same epidemiological phenomena.

As interest in utilizing networks to shed light on social phenomena has grown, pioneered by works like that of \cite{Holland1976,Holland1977} (see also \cite{Leinhardt1977}), so too has the interest in exploring epidemiological processes within these network structures. The goal is to develop models that better reflect real-world scenarios and deepen our understanding of the progression of actual epidemics. Given the extensive body of literature in this area, we recommend \cite{Keeling2005} for an exhaustive review of seminal contributions up to the early 2000s.

Linking epidemiological models on networks with appropriate deterministic analogues has proven challenging, especially when seeking mathematically rigorous results. Typically, mean-field reasoning falls short when applied to network structures, unless the connections exhibit some degree of uniform mixing of vertex connectivity, as is the case in the \ER\ random graph model and the configuration model. For instance, \cite{Volz2007} proposed a system of differential equations for the SIR epidemic on a dynamic configuration model using heuristic methods, but without mathematical guarantees.

This article's objective is to derive rigorous large-population convergence results for the SIR epidemic process within a dynamic random graph model. While both the epidemic and the random graph dynamics come in rather elementary forms, the fact that they interact makes the analysis of the epidemic more difficult. Moreover, one has to keep in mind that the SIR epidemic may be overly simplistic, but it serves as a reference against which to compare more complicated models. Similarly, our model should be regarded as a benchmark against which to compare more intricate models, since it is challenging to appreciate the effects  more elaborate models exhibit without a clear understanding of the phenomena appearing in their simpler counterparts.

As our dynamic network model, we use an \ER\ random graph (or more specifically, its general version, the Stochastic Block Model), wherein edges  appear and disappear at constant rates, independently of each other. This model has the advantage of being Markovian and of having the \ER\ random graph itself as equilibrium distribution. Hence, in its simplest form, our model contains two epidemiological parameters --- an infection rate~$\beta$ and a recovery rate~$\gamma$ --- as well as two parameters governing the graph dynamics --- the rate~$\lambda$ at which an edge forms between two vertices, and the rate~$\mu$ at which an existing edge disappears. We can fix any of these parameters and change the others without loss of generality; hence, in this article, we treat $\gamma$ as constant. The other three parameters may depend on the population size $n$, which offers a broad spectrum of limiting regimes. Interestingly, for most of these regimes, the limiting ordinary differential equation (ODE) mirrors that of the Markovian SIR epidemic process. Yet, in the most interesting regime, where the average degree remains constant and edges form and vanish at a rate similar to infection transmission, a more complex limit emerges. The corresponding system of ordinary differential equations is new to the best of our knowledge.

\subsection{Markovian SIR model}

We first recall and discuss the basic Markovian SIR epidemic process, which is the most widely studied stochastic epidemic model and takes place in a homogenously mixing population. In this model, an infected individual makes infectious contacts with any other given individual at rate~$\beta/n$ within a population of $n$ individuals. If a susceptible individual is contacted by an infected individual, it immediately becomes infected and starts in turn to infect other individuals. The individual recovers at rate~$\gamma$, after which it no longer contributes to the spread of the disease. 

We denote by $S(t)$, $I(t)$, and $R(t)$ the respective number of susceptible, infected, and recovered individuals at time $t$. Under appropriate assumptions, the ratios $S/n$, $I/n$ and $R/n$ can be related to corresponding functions $s(t)$, $i(t)$ and $r(t)$ that satisfy the well-known system of ordinary differential equations (ODEs)
\besn{\label{deterministicSIR}
s'(t) &=-\gamma \bR_0 i(t)s(t),\\
i'(t) &=\gamma \bR_0 i(t)s(t)-\gamma i(t),\\
r'(t) &=\gamma i(t),
}
along with appropriate initial conditions, where $\bR_0=\beta/\gamma$ is the so-called \emph{reproduction ratio}. For instance, for any bounded time interval, the stochastic system converges to the deterministic system \eqref{deterministicSIR} as $n\rightarrow\infty$; see \cite{armbruster2017elementary} for a modern proof, and see, for example, \cite{britton2020stochastic} for a comprehensive overview of the current literature. 

The Markovian SIR epidemic process assumes Poissonian contacts and homogeneous mixing; the latter can be thought of as a contact network given by a static complete graph. That is, on this network, each infected individual contacts their neighbours according to a homogeneous Poisson point process with mean $\beta/n$ per time unit, and all individuals are connected with unchanged edges during an outbreak, with no edges dissolving. Real-world social interactions typically both exhibit non-Markovian properties and vary over time, with connections forming and dissolving, but the mathematical analysis of large population limits of more involved stochastic models by means of simple ODEs like \eqref{deterministicSIR} has not seen much progress. 

Now, in order to analyse and understand diseases with more complicated disease progressions and in non-homogenous populations, epidemiological modellers prefer to directly modify the system \eq{deterministicSIR}  without having to invoke any underlying stochastic model. Alternatively, while stochastic epidemiological processes on static and dynamic networks have been intensively studied, quite often such models are analysed with the help of the \emph{moment-closure method}, which is a mean-field argument and which allows to derive an \emph{approximating} large-population limit by means of a finite system of differential equations; see for example \cite{kiss2017mathematics}. This method, however, comes with no mathematical guarantees regarding validity and accuracy of the approximation. 

In some cases, such as for dynamic sexual partnership networks, the dynamics of the network are restrictive enough to allow the derivation of an appropriate compartmental model; see for example \cite{Kretzschmar1994} and \cite{Althaus2012}, or \cite{diekmann2000mathematical} in more general settings. In these cases, the differential equations can be justified heuristically to some degree, but again no mathematical guarantees and approximation errors in the large population limit are typically derived. In other instances, finite systems of ODEs can be derived, involving the PGF of the degree distributions; see \cite{Volz2007,Volz2008a,Volz2009}. Moreover, many mathematical aspects of the SIR epidemic on random, but non-dynamic random graph, have been derived; see \cite{Durrett2007} for a comprehensive discussion on such and related models. 

Some results to understand epidemics on dynamic random graphs were proved by \cite{Durrett2022}. However, the graph dynamics considered there (the `evoSIR' model) are involved enough to make large population limits difficult to derive, and the mathematical rigorous analysis is confined to quantities such as the final size of the epidemic, likelihood of an epidemic, and critical values for an epidemic. This is in parts due to the fact that there is a feedback between the network and contact process dynamics --- making the model interesting, but also difficult to analyse.

The theory of approximating pure jump Markov processes by limiting ODEs goes back to \cite{Kurtz1970}, but the theory relies on the Markovian property of the macroscopic quantities of interest, which SIR-type models on networks typically do not possess. While such models can  be made Markovian by incorporating macroscopic quantities that count the number of local configurations, such as S-I edges, S-I-S paths, etc., this leads to a non-closed system of infinitely many equations (hence the need for moment closures) with intractable dependencies between the equations.

In this article, we follow the approach of \cite{barbour2013approximating}, who are able to relate the epidemic curve to properties of certain branching processes. This approach allows the derivation of limiting integral equations of the quantities of interest along with mathematically rigorous large population limit results. The integral equations can sometimes be expressed as differential equations, and --- surprisingly --- this turns out the be case for the models under consideration in this article. 

\subsection{Brief outline of main results for single-type case}

As a preliminary result, we first analyse a general epidemic process on a homogeneously mixing population, or equivalently, on a complete graph that remains unchanged over time. The contact processes between infected and susceptible individuals need not be Poisson and the recovery times of infected individual need not be exponential. Recently, \cite{pang2022functional} derived functional laws of large numbers and functional central limit theorems for non-Markovian epidemic models. In these models, general distributions are allowed for the durations of the respective disease states, while the contact processes remain Poissonian. Moreover, the epidemic processes studied by \cite{pang2022functional} are single-type and begin with a fraction of the total population being infectious, without accounting for the randomness early in the epidemic. For general epidemic processes that start with a single infected individual and involve general point processes for the contact interactions between infected and susceptible individuals, \cite{barbour2013approximating} proved the functional law of large numbers for the susceptible ratio curves after a specific random time shift. The results of \cite{barbour2013approximating} are not directly applicable to our model due to a technical restriction in their main theorem, hence we need to first remove this technical restriction. With this, we are able to handle more general situations than \cite{pang2022functional}, albeit with different (most likely unnecessary) technical restrictions. Specifically, in the single type case, after a random time shift, the susceptible, infected, and recovery ratios converge (conditional on the event that there is an epidemic) to certain functions $s$, $i$, and $r$, respectively, as $n\to\infty$. The function $s$ has the form of the Laplace transform of a random variable $\widehat{W}$ (which depends on the contact processes), given by
\be{
  s(u)=\IE \exp\bclc{-\widehat{W}e^{\widehat{\C{M}}u}},
}
where $\widehat{\C{M}}$ is a constant that will be introduced in detail later. Denote by $Q$ a random variable modelling the time elapsed from an individual getting infected to recovery, with cumulative distribution function $F_Q$ and $F_Q^c=1-F_Q$. The limiting infected and recovery ratios can then be expressed as
\bes{
          i(t)&= 1-s(t)-\int_{-\infty}^{t}(1-s(v))dF^{c}_{Q}\clr{t-v},\\
          r(t)&=-\int_{-\infty}^{t}(1-s(v))dF_Q\clr{t-v}.
}
We then apply the results to a multi-type Markovian SIR process on a dynamic Erd\H{o}-R\'enyi random graph on $n$ vertices. In the single-type version of the model, the graph structure changes according to a 2-state Markov process: existing edges each dissolve at rate~$\mu_n$, and pairs of disconnected vertices each get connected at rate~$\lambda_n$, both independently of the vertices' infection statuses. Note that the equilibrium distribution of this network process is the Erd\H{o}s-R\'enyi graph on $n$ vertices  with connection probability $\lambda_n/(\lambda_n+\mu_n)$. Assuming the network process is in equilibrium state, a randomly selected individual becomes infectious. Any infectious vertex infects each of its connected neighbours at rate~$\beta_{n}$, and recovers (and becomes immune) at rate~$\gamma$. 

In order to analyse the various possible scaling limits, we assume that $\gamma$ remains fixed throughout, but allow $\mu_n$, $\lambda_n$ and $\beta_n$ to vary by letting 
\be{
  \beta_{n}=\beta n^{\kappa_{\beta}},
  \qquad
  \mu_n=\mu n^{\kappa_{\mu}}, 
  \qquad
  \lambda_n=\lambda n^{\kappa_{\lambda}}.
}
Here, $\kappa_{\beta}$, $\kappa_{\mu}$ and $\kappa_{\lambda}$ are fixed real numbers, and $\mu$, $\lambda$ and $\beta$ are fixed positive constants.\footnote{Note the slight abuse of notation here: In the expressions `$\kappa_{\beta}$', `$\kappa_{\mu}$' and `$\kappa_{\lambda}$' the variables $\beta$, $\mu$ and $\lambda$ serve as labels to differentiate the three constants, not as arguments of a function $\kappa_x$, $x\in\IR$. Hence, for instance, $\mu=\lambda$ does not imply $\kappa_\mu = \kappa_\lambda$.} While one may choose $\gamma=1$ without loss of generality, we will leave it in our equations explicit in order to keep the results as comprehensive as possible.

In this article we show that the ratios of infected and recovered individuals after a random time shift and conditional on a sequence of events of positive probabilities converge (in conditional probability) to $s$, $i$, and $r$ as $n\to\infty$; the different scaling limits and parameters arising in the limiting system are provided in Table~\ref{singletable}. Unfortunately, our proofs do not allow us to deduce that the epidemic curves satisfy these limits conditional on an epidemic outbreak in full generality, but only on a positive-probability subset of the event of an outbreak. We believe this is merely a restriction of our techniques and it seems unlikely that the epidemic curves behave differently outside this event, and so it remains an open problem to prove our results in full generality.

The limiting functions satisfy specific ordinary differential equations, which can be classified into two types. The first type applies to all cases in Table~\ref{singletable} except for case (6b), and corresponds to the equation \eqref{deterministicSIR} with a possibly different reproduction ratio depending on $\lambda$, $\mu$ and $\beta$. The second type, Case~(6b) in Table~\ref{singletable}, occurs when $\kappa_{\mu}=0$, $\kappa_{\lambda}=-1$, and $\kappa_{\beta}=0$. This is arguably the most interesting case, in which the average degree is of order 1 and in which the edges change at the same speed as the epidemic progresses. We  show that as $n$ grows large, there is sequence of events with probabilities converging to a positive probability such that the epidemic ratios satisfy the system of ODEs
\besn{\label{singleSLLR}
  \bss'(t) &=-\beta \bsl(t)s(t) ,\\
  \bsi'(t)&=\beta \bsl(t)s(t)-\gamma \bsi(t) ,\\
  \bsr'(t) &=\gamma \bsi(t),
}
subject to the constraint $s(t)+i(t)+r(t)=1$ and with the additional quantity $\bsl(t)$ satisfying the ODE
\ben{\label{singleSLLR2}
  \bsl'(t) = \frac{\beta\lambda}{\mu} \bsl(t)s(t) + \lambda i(t) - (\beta+\mu+\gamma)\bsl(t).
}

The system of ODEs given by \eq{singleSLLR} and \eq{singleSLLR2} have an interesting interpretation. While the quantities $\bss$, $\bsi$, and $\bsr$ follow the dynamics of a regular SIR epidemic, the \emph{force of infection} is moderated by the new quantity~$\bsl$. This quantity can be understood as the average number of active edges emanating from an infected individual, and each such edge has a probability of $\bss(t)$ of being connected to a susceptible individual through which the infection can spread at rate~$\beta$. This results in the appearance of the term $\beta \bsl(t)s(t)$ in the system \eqref{singleSLLR}. 

Now, note that if $n$ is large, the degree distribution of a typical vertex is~$\Po(\lambda/\mu)$. As is well-known, in models where edges ``mix well'', a newly infected vertex typically exhibits a size-biased degree distribution. Since the size-bias transform of a Poisson distribution is the same Poisson distribution shifted by 1 (due to the one edge through which the infection was contracted), upon infection, an individual has right away, on average, $\lambda/\mu$ active edges through which the infection can be further propagated. This yields the first term on the right hand side of \eq{singleSLLR2} since the overall rate of infection is~$\beta \bsl(t)s(t)$. However, as soon as a vertex is infected, additional edges start to appear due to the changing nature of the network, and this happens at rate~$\lambda$ for each infected vertex, which is captured by the second term on the right hand side of \eq{singleSLLR2}. Finally, each active edge has three ways of ceasing to participate in the spread of the infection: (1) an infection is transmitted through an edge (rate~$\beta$), (2) an edge disappears (rate~$\mu$), or (3) the corresponding vertex recovers (rate~$\gamma$). This yields the final term in \eq{singleSLLR2}.

It turns out that in order to handle the general case with multiple groups and parameter regimes, it is beneficial to split the active edges into two groups, $\bsl(t)=\bsl_{\mathrm{c}}(t)+\bsl_{\mathrm{d}}(t)$, where $\bsl_{\mathrm{c}}(t)$ denotes the number of those active edges already present at the time of infection, and where $\bsl_{\mathrm{d}}(t)$ denotes the number of active edges that only appeared after the time of infection. These two quantities are governed by the ODEs
\besn{\label{singleSLLR3}
  \bsl_{\mathrm{c}}'(t)
  &=\frac{\lambda}{\mu}\beta \clr{\bsl_{\mathrm{c}}(t)+ \bsl_{\mathrm{d}}(t)}s(t)  - (\mu +\beta +\gamma ) \bsl_{\mathrm{c}}(t),\\
  \bsl_{\mathrm{d}}'(t)&=\lambda\bsi(t)-(\mu +\beta +\gamma ) \bsl_{\mathrm{d}}(t),
}
satisfying the constraint
\ben{\label{singleSLLR4}
  \frac{\lambda}{\mu}\bsi(t) = \bsl_{\mathrm{c}}(t)+\bsl_{\mathrm{d}}(t)+\frac{\beta}{\mu}\bsl_{\mathrm{d}}(t).
}
Clearly, the sum of the two equations in \eq{singleSLLR3} yields \eq{singleSLLR2}. For the constraint~\eq{singleSLLR4}, we refer to Section~\ref{constraintonsimplecase}, but it says in essence that the edges emanating from an infected individual (apart from the one through which the individual was infected), comprises active edges present \emph{at time} of infection, active edges that appeared \emph{after} time of infection, and edges \emph{no longer} participating in the spread of the infection, the number of which happens to equal $\beta\bsl_{\mathrm{d}}(t)/\mu $.

\subsection{Basic notations}
Throughout this article, $\IN$ denotes the set of natural numbers, and $\IR$ ($\IR_+$, $\IR_-$) denotes the set of all (non-negative, negative) real numbers. For $k\in\IN$, $\IR^k$ is the Euclidean space with dimension $k$. For $\IR^k$, denoted by $\F{B}(\IR^k)$ the Borel $\sigma$-field induced by the Euclidean norm. For $k\in\IN$, we use the notation $[k]=\clc{1,2,\dots,k}$. For any $a\in \IR$, $\floor{a}=\max\{b\in\mathbb{N}\given b\leq a\}$ and  $\ceil{a}=\min\{b\in\mathbb{N}\given b\geq a\}$. For $a$, $b\in\IR$, let $a\wedge b=\min\{a, b\}$. Let $D[a,b]$ ($a,b\in\IR, a<b$) denote the space of $\IR$–valued c\`{a}dl\`{a}g functions defined on $[a,b]$, and convergence in $D[a,b]$ means convergence in the Skorohod $J_1$ topology, and we denote the associated metric by $d^\circ_{ab}:~D[a,b]\times D[a,b]\rightarrow \IR$, see \cite[Chapter 3]{billingsley1999convergence}. We use $\I\cls{\{\cdot\}}$ for indicator function and let $\delta_{x}(A)=\I\cls{A}(x)$ $(A\in\F{B}(\IR))$ be the Dirac measure with mass 1 at $x$. An Erd\H{o}s-R\'enyi random graph is denoted by $\C{G}(n,p)$, where $n\in\mathbb{N}$ is the number of the vertices/individuals, and $p$  the connection probability. For two families of positive random variables $\clc{X_n}$ and $\clc{Y_n}$, $X_n\sim Y_n$ means their quotient convergence to 1 in probability. If $X$ is a random variable or point process, we use $\law(X)$ to denote its law and we denote its cumulative distribution function by $F_X$; we let $F_X^c=1-F_X$. The total variation distance between $\law (X_1)$ and $\law (X_2)$ is denoted by $\dtv(\law (X_1),\law (X_2))$. We use $\|\cdot\|_{\mathrm{TV}}$ to denote the total variation distance of a signed measure, which is twice $\dtv$ when applied to the difference between two probability measures. To say that an event occurs with high probability means that its probability approaches 1 as $n\to\infty$.

\section{General epidemic model and its large population limits}\label{models}

In this section, we prove convergence results for a general stochastic epidemic process among $n$ individuals which can be separated into $k$ types and an individual is of exactly one type. Moreover, each individual is in exactly one of the three states \emph{susceptible}, \emph{infected} and \emph{removed}. While infected, an individual is also infectious and transmits the infection to still susceptible individuals in the population. Let $\cV_{n,i}$ denote the set of vertices of type~$i$, and let $n_i=|\cV_{n,i}|$ denote the number of vertices of type~$i$. Throughout this article, we make the following assumption.
\begin{assumption}\label{initial}
Assume that there exist positive constants $p_i$'s such that 
$n_i/n =p_i + \bigo(1/n)$
for all $i\in[k]$ and $n\in\IN$. 
\end{assumption}

We assume that the first infectious individual is of type~$\iota \in[k]$. Let $\cS^{\iota}_{n,i}(t)$ denote the susceptible individuals of type~$i$ at time $t$, and let 
\be{
  \cS^{\iota}_{n}(t)= \bclr{\cS^{\iota}_{n,1}(t),\dots,\cS^{\iota}_{n,k}(t)}.
}
Let $S_{n,i}^\iota(t) = \abs{\cS_{n,i}^\iota(t)}$, the number of susceptible individuals of type~$i$ at time~$t$. Let 
\be{
  S^{\iota}_{n}(t)= \bclr{S^{\iota}_{n,1}(t),\dots,S^{\iota}_{n,k}(t)},
}
the vector of numbers of all types of individuals. We write 
\be{
  \abs{S_n^{\iota}(t)}_1=\sum_{i\in[k]}S_{n,i}^\iota(t),
}
which is the total number of susceptible individuals. Similar notation is used for infected and recovery individuals: Let $\cI^{\iota}_{n,i}(t)$ and $\cR^{\iota}_{n,i}(t)$ denote the infected and removed individuals of type~$i$ at time $t$, respectively, and  
\be{
  \cI^{\iota}_{n}(t)= \bclr{\cI^{\iota}_{n,1}(t),\dots,\cI^{\iota}_{n,k}(t)},
  \qquad
  \cR^{\iota}_{n}(t)= \bclr{\cR^{\iota}_{n,1}(t),\dots,\cR^{\iota}_{n,k}(t)}.
}
Similarly,
\be{
  I^{\iota}_{n}(t)= \bclr{I^{\iota}_{n,1}(t),\dots,I^{\iota}_{n,k}(t)},
  \qquad
  R^{\iota}_{n}(t)= \bclr{R^{\iota}_{n,1}(t),\dots,R^{\iota}_{n,k}(t)} 
}
are the  vectors of numbers of all types of infected and recovery individuals, respectively. 
 
\subsection{A general epidemic process}\label{general}

The epidemic evolves according to the following scheme. For each type~$i$, let $\xi_{i}=\clr{\xi_{i,1},\dots,\xi_{i,k}}$ denote a collection of $k$ simple point processes on the positive real line being independent of population size $n$. These processes are called \emph{contact processes}, and we interpret the points in $\xi_{i,j}$ as the times when an individual of type~$i$ contacts individuals of type~$j$ after having been infected. If there are more than $n_j$ points in the process $\xi_{i,j}$ then only the first $n_j$ points represent contacts. Moreover, let $Q_{i}$ be a positive random variable, which we interpret as the time an individual of type~$i$ spends in the infected state. A susceptible individual contacted by an infected individual turns into an infectious individual. The processes $\xi_{i,1},\dots,\xi_{i,k}$ need not be independent of each other, and their joint distribution is allowed to depend on the type~$i$. We assume the cumulative distribution function $F_{Q_i}$ is differentiable on $\IR_+$.

Now, for each $i\in[k]$, each individual $v\in \cV_{n,i}$ is equipped with an independent copy $(Q^v_i,\xi^v_{i})$ of $(Q_i,\xi_{i})$. When $v$ is infected, it makes contacts according to $\xi^v_{i} = (\xi^v_{i,1},\dots,\xi^v_{i,k})$, where $0$ in the point process time line is to interpreted as the time when $v$ was infected in the absolute time line. It does so for an amount of time determined by $Q^v_i$, after which it changes to the removed state, and is no longer infectious. The individuals contacted during the infectious state are chosen independently at random without replacement from the respective sets $\cV_{n,j}$. For $j=i$, the individual is allowed to contact itself, which, however, results in no infection.

Since only contacts in the point processes during which an individual is infectious matter for the evolution of the epidemic, we introduce notation for these truncated point processes. For $i$, $j\in[k]$ and $x\geq 0$, let
\be{
  \xi_{x,i,j}(A)= \xi_{i,j}(A \cap [0,x]),\qquad A\in \F{B}(\IR_+),
} 
and let
\be{
  \bR_{0,i,j}=\IE \xi_{Q_i,i,j}(\IR_+),
}
and $M=(\bR_{0,i,j})_{i,j\in[k]}$. Denote the relative intensity measure of $\xi_{Q_i,i,j}$ by 
\be{
  G_{i,j}(A) = \IE\xi_{Q_i,i,j}(A)/\bR_{0,i,j},\qquad A\in \F{B}(\IR_+).
}

While it would mathematically suffice to only work with the $\xi_{Q_i,i,j}$ and without introducing the $Q_i$ and the $\xi_{i,j}$ separately, we believe it is conceptually clearer to separate the two components. We make the following assumptions on $\clc{\xi_{Q_i,i,j}}_{i,j\in[k]}$.

\begin{assumption}\label{dTVxi}\noindent
\begin{enumerate}[label={$(\roman*)$}]
 \item Assume $\IE \xi_{Q_i,i,j}(\IR_+)^2<\infty$ for $i$, $j\in[k]$.
\item The matrix $M$ is irreducible, and its largest eigenvalue is larger than 1.
\item For all $i,j\in[k]$, the relative intensity measure $G_{i,j}$ is absolutely continuous with respect to Lebesgue measure and the superposition $\xi_{i,1}+\dots+\xi_{i,k}$ is a simple point process.
\end{enumerate}
\end{assumption}
Define
\be{
  M^L_{i,j}(s)=\bR_{0,i,j}\int_{0}^{\infty}e^{-su}G_{i,j}(du),
}
and let $M^L(s)$ denote the $k\times k$ matrix with elements $M^L_{i,j}(s)$. Note that $M^L(0)=M$.

Denote by $\C{B}^{\iota}_{n, i}(t)$  the set of individuals of type~$i$ that, at time $t$, are either infected or removed. Let $\C{B}^{\iota}_n(t)=\bclr{\C{B}^{\iota}_{n, 1}(t),\dots,\C{B}^{\iota}_{n, k}(t)}$, let $B^{\iota}_{n,i}(t)=\abs{\C{B}^{\iota}_{n, i}(t)}$, and let $B_n^{\iota}(t)=\clr{B^{\iota}_{n,1}(t),\dots,B^{\iota}_{n,k}(t)}$. The total number of individuals in $\C{B}_n^{\iota}(t)$ is denoted by $\abs{B_n^{\iota}(t)}_1=\sum_{j\in[k]}B^{\iota}_{n, j}(t)$.

For the early stages of the epidemic, the process $\C{B}^{\iota}_n(t)$ can be well approximated by a multi-type Crump-Mode-Jagers (C-M-J) branching process. Let $\C{B}^{\iota\prime}$ denote a such branching process, starting with one individual of type~$\iota$, and offspring distribution given by $\clc{\xi_{Q_i,i,j}}_{i,j\in[k]}$. Let $B^{\iota\prime}(t)=\clr{B^{\iota\prime}_{1}(t),\dots,B^{\iota\prime}_{k}(t)}$ denote the vector of the number of individuals of each type that have been born in the branching process before time $t$, and denote by $\abs{B^{\iota\prime}(t)}_1=\sum_{j\in[k]}B^{\iota\prime}_{j}(t)$ the total number of individuals in the branching process. 

The Malthusian parameter of the branching process $\C{B}^{\iota\prime}$, denoted by $\C{M}$, is positive, and the largest eigenvalue of $M^L(\C{M})$ is 1. Let $\zeta=\clr{\zeta_{1},\dots,\zeta_{k}}$ and $\eta=\clr{\eta_{1},\dots,\eta_{k}}$ denote the positive left and right eigenvectors of $M^L(\C{M})$ associated with eigenvalue 1, normalised such that $\zeta^\trans{\bf 1}=\zeta^\trans\eta=1$. Part $(iii)$ of Assumption~\ref{dTVxi} implies the branching process $\C{B}^{\iota\prime}$ is supercritical.

It is well-known --- see for example \cite{iksanov2015rate} or \cite{jagers1989general} --- that 
\be{\label{forwardB}
    B^{\iota\prime}(t)e^{-\C{M} t}\rightarrow W^{\iota}\zeta \quad \text{in}\quad L_1 \quad \text{as}\quad t\rightarrow\infty.
}
Here, $W^{\iota}$ is a random variable.

For the analysis of the final size of susceptible individuals of type~$\nu\in[k]$, we use a backward branching process starting from an individual of type~$\nu$, denoted by $\widehat{\C{B}}^{\nu\prime}$, to approximate the susceptibility process (which will be introduced later); see for example \cite[Section 10.5.2]{diekmann2000mathematical} and \cite{britton2019infector}. Specifically, this multi-type branching process is defined by follows: An individual of type~$j$ gives births of type~$i$ according to a Poisson point process with intensity measure $p_i\bR_{0,i,j}G_{i,j}/p_j$ and all the Poisson point processes are independent.

Let
\be{
    \widehat\bR_{0,j,i}=\frac{p_i}{p_j}\IE \xi_{Q_i,i,j}(\IR_+)=\frac{p_i}{p_j}\bR_{0,i,j},\quad i,j\in[k],
}
and denoted by $\widehat{M}$ the matrix with elements $\widehat\bR_{0,i,j}$. 
Let
\be{
  \widehat{M}^L_{i,j}(s)=\widehat\bR_{0,i,j}\int_{0}^{\infty}e^{-su}G_{i,j}(du).
}
The branching process $\widehat{\C{B}}^{\nu\prime}$ has the Malthusian parameter, denoted by $\widehat{\C{M}}>0$, and $\widehat{M}^L(\widehat{\C{M}})$ has largest eigenvalue 1. Denote by $\hat{\zeta}=\clr{\hat{\zeta}_1,\dots,\hat{\zeta_k}}$ and $\hat{\eta}=\clr{\hat{\eta}_1,\dots,\hat{\eta}_k}$ the positive left and right eigenvectors of $\widehat{M}^L(\widehat{\C{M}})$ associated with eigenvalue 1, normalised such that $\hat{\zeta}^\trans{\bf 1}=\hat{\zeta}^\trans\hat{\eta}=1$. Similarly as before, we have  
\ben{\label{backwardB}
    \widehat{B}^{\nu\prime}(t)e^{-\widehat{\C{M}} t}\rightarrow \widehat{W}^{\nu}\hat{\zeta} \qquad \text{in $L_1$  as $t\rightarrow\infty$},
}
where $\widehat{W}^{\nu}$ is a random variable with a Laplace transform 
\be{
  \psi^{\nu}(s)=\IE e^{-s\widehat{W}^{\nu}}
} 
satisfying the implicit equations
\be{
        \psi^{\nu}(s)= \exp\bbbclr{-\sum_{i=1}^k\widehat\bR_{0,\nu,i}\int_0^{\infty}(1-\psi_{i}(se^{-\widehat{\C{M}}u}))G_{i,\nu}(du)} ,\qquad \nu\in[k].
}
Now let $\tau_n$ be the random time satisfying  $|B^{\iota\prime}(\tau_n)|=\lfloor n^{17/24} \rfloor$, which is well-defined due to $(iii)$ of Assumption~\ref{dTVxi}. Moreover, let 
\be{
  t_n(u)=\frac{17}{24\C{M}}\log n + u,  \qquad  u\in\IR.
} 
Define 
$A_n= \clc{|B^{\iota\prime}\clr{\tau_n}|e^{-\C{M}\tau_n} \geq n^{-1/48}}\cap\clc{|B^{\iota\prime}\clr{\tau_n}|-n^{11/24}\leq|B_n^{\iota}\clr{\tau_n}|\leq |B^{\iota\prime}\clr{\tau_n}|}$. Note that
$\clc{|B^{\iota\prime}\clr{\tau_n}|e^{-\C{M}\tau_n} \geq n^{-1/48}}=\clc{\tau_n\leq \C{M}^{-1}\log\clr{\floor{n^{17/24}}n^{1/48}}}$ which can be viewed as a subset of the occurrence of an epidemic, where an epidemic refers to the infection of a significant portion of the population, such as at least a fractional power of the population, and $\clc{|B^{\iota\prime}\clr{\tau_n}|-n^{11/24}\leq|B^{\iota}\clr{\tau_n}|\leq |B^{\iota\prime}\clr{\tau_n}|}$ is the event that the size difference between the forward branching process $\C{B}^{\iota\prime}$ and the epidemic process $\C{B}_n^{\iota}$ is bounded by $n^{11/24}$ at the random time $\tau_n$.

\begin{theorem}\label{maintheorem}
Under Assumptions~\ref{initial} and~\ref{dTVxi}, there exists a coupling between $\C{B}^{\iota}_n$ and $\C{B}^{\iota\prime}$ such that $\lim_{n\to\infty}\IP\cls{A_n}>0$ and such that, for any $\varepsilon>0$, $\nu\in[k]$ and $T\in\IR$,
\ben{\label{maintheoremP}
  \lim_{n\rightarrow\infty}\IP\bbbcls{\sup_{u\in(-\infty,T]}\left|n_{\nu}^{-1}S^{\iota}_{n,\nu}\clr{\tau_n+t_n(u)}-\bss^{\iota}_{\nu}(u)\right|>\varepsilon\given A_n}=0,
}
where $\bss^{\iota}_{\nu}(u)=\psi^{\nu}\bclr{e^{\widehat{\C{M}}u}m_*}$, 
$m_*=\sum_{l\in[k]}\zeta_l \hat{\zeta}_l p_l^{-1},$
and $\psi^{\nu}(v)=\IE\clc{e^{-v\widehat{W}^{\nu}}}$.
Moreover, if $\bss^{\iota}_{\nu}(\infty)=0$, the time interval $(-\infty,T]$ in \eqref{maintheoremP} can be replaced by $\mathbb R$.
\end{theorem}

\begin{remark} The open question is whether $\lim_{n\to\infty}\IP\cls{A_n}$ equals the survival probability of $\C{B}^{\iota\prime}$ and if not, whether a different coupling and a different sequence of events exists for which the probabilities do converge to the survival probability.   
\end{remark}

Theorem~\ref{maintheorem} demonstrates that, in the case of a large population size, the susceptible curve is determined asymptotically by a continuous curve after a random time shift. The next theorem indicates that the infected and recovery ratio curves also exhibit similar FLLN properties.

\begin{theorem}\label{IEmodels}
Under Assumptions~\ref{initial} and~\ref{dTVxi} and using the notation of Theorem~\ref{maintheorem}, there exists a coupling between $\C{B}^{\iota}_n$ and $\C{B}^{\iota\prime}$ such that $\lim_{n\to\infty}\IP\cls{A_n}>0$ and such that the following holds. For any $\varepsilon>0$, $\nu\in[k]$ and $T\in\IR$,
\ben{\label{FLLNir}
  \lim_{n\rightarrow\infty}\IP\bbbcls{\sup_{u\in(-\infty,T]}\left|n_{\nu}^{-1}I^{\iota}_{n,\nu}\clr{\tau_n+t_n(u)}-\bsi^{\iota}_{\nu}(u)\right|>\varepsilon\given A_n}=0,
}
and
\ben{\label{FLLNr}
  \lim_{n\rightarrow\infty}\IP\bbbcls{\sup_{u\in(-\infty,T]}\left|n_{\nu}^{-1}R^{\iota}_{n,\nu}\clr{\tau_n+t_n(u)}-\bsr^{\iota}_{\nu}(u)\right|>\varepsilon\given A_n}=0.
}
The curves  $(\bss^{\iota}_{\nu},\bsi^{\iota}_{\nu},\bsr^{\iota}_{\nu})$ are
the unique solutions to the deterministic integral equations
\besn{\label{IE}
     \bss^{\iota}_{\nu}(t)&= \exp\bbbclr{-\sum_{i=1}^k\widehat\bR_{0,\nu,i}\int_0^{\infty}(1-\bss^{\iota}_{i}(t-u))G_{i,\nu}(du)},\\
    \bsi^{\iota}_{\nu}(t)&= 1-\bss^{\iota}_{\nu}(t)-\int_{-\infty}^{t}(1-\bss^{\iota}_{\nu}(v))dF^{c}_{Q_{\nu}}\clr{t-v},\\
    \bsr^{\iota}_{\nu}(t)&=-\int_{-\infty}^{t}(1-\bss^{\iota}_{\nu}(v))dF_{Q_{\nu}}\clr{t-v},
}
where $dF_{Q_i}^c(u-v)$ is the differential of the map $v\rightarrow F_{Q_i}^c(u-v)$, and where the initial conditions are given by
\be{
  \bss^{\iota}_{\nu}(-\infty)=1,\qquad\bsi^{\iota}_{\nu}(-\infty)=0,\qquad\bsr^{\iota}_{\nu}(-\infty)=0,
}
and
\be{
  \lim_{t\rightarrow-\infty}e^{-\widehat{\C{M}}t}\frac{d}{dt}\bss^{\iota}_{\nu}(t)
  =-m_*\widehat{\C{M}}\IE\widehat{W}^{\nu}.
}
\end{theorem}

\subsection{Empirical point processes as contact processes}\label{contempirical}
In this subsection, consider the special case when the contact process is given by an empirical point process of independent points. As before, the first infected individual is of type~$\iota$ and $\C{B}^{\iota}_{n, i}(t)$ denotes the set of individuals of type~$i$ that, at time $t$, are either infected or removed.  In what follows, we will let $\xi_{n,i}=\clr{\xi_{n,i,1},\dots,\xi_{n,i,k}}$ be a collection of $k$ simple point processes on the positive real line. Each process $\xi_{n,i,j}$ will contain exactly $n_j$ points.  Specifically, we define $\xi_{n,i}$ as follows: Let $\clc{X_{n,i,j}}_{i,j\in[k]}$ be a set of positive independent random variables, and denote their distributions by $\clc{F_{X_{n,i,j}}}_{i,j\in[k]}$, and let $X_{n,i,j,1}, \ldots, X_{n,i,j,n_j}$ represent independent and identically distributed copies of $X_{n,i,j}$ for $i,j \in [k]$. Define
\bgn{\label{superposition1}
   \xi_{n,i,j}= \sum_{l=1}^{n_j}\delta_{X_{n,i,j,l}},
    \qquad
   \xi_{n,x,i,j}= \sum_{l=1}^{n_j}\delta_{X_{n,i,j,l}}(\cdot\cap[0,x]),\\ \label{superposition2}
   \xi_{n,x,i}= \bclr{\xi_{n,x,i,1},\dots,\xi_{n,x,i,k}},
}
where $x\geq 0$.

For $Q_{i}$ a random variable representing the time it takes for individuals of type~$i$ to recover, denote the relative intensity measure of $\xi_{n,Q_,i,j}$ by $G_{n,i,j}$; that is, for every Borel set $A\in \F{B}(\IR_+) $,
\be{
    G_{n,i,j}(A)=\IE \xi_{n,Q_i,i,j}(A)/\bR_{0,n,i,j},
}
where $\bR_{0,n,i,j}=\IE \xi_{n,Q_i,i,j}(\IR_{+})$.

Similarly as before, the  process $\C{B}^{\iota}_{n}$ can be approximated by a C-M-J branching process, which is denoted by $\C{B}^{\iota\prime}_{n}$; since now the birth giving processes $\clc{\xi_{n,i,j}}_{i,j\in[k]}$ depend on the population size $n$, we label it with subscript~$n$. 

Note that $\xi_{n,x,i,j}$ can be thought of as a binomial process, that is, 
\be{
\law(\xi_{n,x,i,j}) = \law\bbbbclr{\, \sum_{l=1}^{\alpha_{n,x}}\delta_{Y_{l,n,x,i,j}}}
}
where $Y_{1,n,x,i,j}, Y_{2,n,x,i,j}, \dots$ are independent and identically distributed random variables with cumulative distribution function 
\be{
  F_{Y_{1,n,x,i,j}}(t) = \frac{F_{X_{n,i,j}}(t\wedge x)}{F_{X_{n,i,j}}([0,x])}.
}
and $\alpha_{n,x} \sim \Bi(n,F_{X_{n,i,j}}([0,x]))$. The binomial process $\xi_{n,x,i,j}$ can be approximated by a Poisson point process.

\begin{assumption}\label{Lambda}
For every $x\in\IR_+$ and $i,j\in[k]$, there exists a measure $ \Lambda_{x,i,j}$ such that for any $A\in\F{B}(\IR_+)$
\be{
  \Lambda_{x,i,j}(A)=\lim_{n\rightarrow\infty}
      \IE\bclc{\xi_{n,x,i,j}(A\cap [0,x])}.
}
\end{assumption}

Let $\xi_{x,i,j}$, $j\in[k]$, be independent Poisson point processes with intensity measures $\Lambda_{x,i,j}$, $j\in[k]$, and let $\xi_{x,i}=\bclr{\xi_{x,i,1},\dots,\xi_{x,i,k}}$. If we replace the constant $x$ by the random variable $Q_{i}$, the process $\xi_{Q_{i},i}$ is a Cox process directed by the random measure $\Lambda_{Q_i,i,1}\otimes\cdots\otimes\Lambda_{Q_i,i,k}$. 

For $i$, $j\in[k]$, let
\be{
  \bR_{0,i,j}=\IE \xi_{Q_i,i,j}(\IR_+),
}
and let $M=(\bR_{0,i,j})_{i,j\in[k]}$. Denote the relative intensity measure of $\xi_{Q_i,i,j}$ by 
\be{
  G_{i,j}(A) = \xi_{Q_i,i,j}(A)/\bR_{0,i,j},\qquad A\in \F{B}(\IR_+).
}
We make the following assumption.
 
\begin{assumption}\label{empirical}\noindent 
\begin{enumerate}[label={$(\roman*)$}]
\item 
Assume $\clc{\xi_{Q_i,i,j}}_{i,j\in[k]}$ satisfy Assumption~\ref{dTVxi}.
\item If $k>1$, the total variation distance between the distributions of $\xi_{n,Q_i,i}$ and $\xi_{Q_i,i}$ is of order $\lito(n^{-17/24})$ for $i\in[k]$. 
\item The total variation distance between $\bR_{0,i,j}G_{i,j}$ and $\bR_{0,n,i,j}G_{n,i,j}$ is of order $\lito(n^{-1/3})$ for $i$, $j\in[k]$.
\end{enumerate}
\end{assumption}
We will show that, under this assumption, using $\clc{\xi_{Q_i,i}}_{i\in[k]}$ as birth giving processes, we can define a C-M-J branching process $\C{B}^{\iota\prime}$ which can be coupled to $\C{B}^{\iota\prime}_{ n}$ successfully with high probability until the birth of the $\floor{n^{17/24}}$-th individual. Moreover, the susceptibility process of an individual of type  $\nu$ can be coupled to a backward branching process, denoted by $\widehat{\C{B}}^{\nu\prime}$, successfully with high probability until the birth of the $\floor{n^{1/3}}$-th individual. The backward branching process $\widehat{\C{B}}^{\nu\prime}$ is constructed by starting with an individual of type~$\nu$ and individuals of type~$j$ give births of type~$i$ according to Poisson point processes with intensity measure $p_i\bR_{0,i,j}G_{i,j}/p_j$.

\begin{theorem}\label{contactempiricaltheorem}
Under Assumptions~\ref{initial},~\ref{Lambda} and~\ref{empirical}, the conclusions of Theorems~\ref{maintheorem} and~\ref{IEmodels} hold.
\end{theorem}

\subsection{Example: The multi-type Markovian SIR process}
Recall \eqref{IE}; for $\nu\in[k]$, we obtain that 
\ba{
          \bsi^{\iota}_{\nu}(t)&= 1-\bss^{\iota}_{\nu}(t)-\gamma_{\nu}\int_{-\infty}^{t}(1-\bss^{\iota}_{\nu}(v))e^{-\gamma_{\nu}(t-v)}dv,\\
          \bsr^{\iota}_{\nu}(t)&=\gamma_{\nu}\int_{-\infty}^{t}(1-\bss^{\iota}_{\nu}(v))e^{-\gamma_{\nu}(t-v)}dv.
}
Differenting both sides of the second equation, we obtain that
\bes{
          \frac{d}{dt}\bsr^{\iota}_{ \nu}(t)&=\gamma_{\nu}\bbclr{1-\bss^{\iota}_{\nu}(t)-\gamma_{\nu}\int_{-\infty}^{t}(1-\bss^{\iota}_{\nu}(v))e^{-\gamma_{\nu}(t-v)}dv}
          =\gamma_{\nu} \bsi^{\iota}_{\nu}(t).
}
Thus
\ben{\label{MarkovSIR1}
          \frac{d}{dt}\bsi^{\iota}_{ \nu}(t)= -\frac{d}{dt}\bss^{\iota}_{ \nu}(t)- \frac{d}{dt}\bsr^{\iota}_{ \nu}(t)=-\frac{d}{dt}\bss^{\iota}_{ \nu}(t)-\gamma_{\nu} \bsi^{\iota}_{\nu}(t),
}
which implies
\besn{\label{MarkovSIR2}
          \bsi^{\iota}_{\nu}(t)&= -\int_{0}^{\infty}\frac{d}{ds}\bss^{\iota}_{ \nu}(s)|_{s=t-u}e^{-\gamma_{\nu} u}du
          =-\int_{-\infty}^t\frac{d}{d\nu}\bss^{\iota}_{ \nu}(v)e^{-\gamma_{\nu}(t-v)}dv\\
          &=-\bss^{\iota}_{\nu}(t)+\gamma_{\nu}\int_{-\infty}^t\bss^{\iota}_{\nu}(v)e^{-\gamma_{\nu}(t-v)}dv.
}
Now by \eqref{IE} the susceptible curve satisfies
\bes{
 \frac{d}{dt}\bss^{\iota}_{ \nu}(t)&=\bss^{\iota}_{\nu}(t) \frac{d}{dt}\bbbclr{-\sum_{i=1}^k\frac{p_i}{p_{\nu}} \int_0^{\infty}(1-\bss^{\iota}_{i}(t-u))\beta_{i,\nu}e^{-\gamma^{(i)u}}(du)} \\
 &=\bss^{\iota}_{\nu}(t) \bbbclr{\sum_{i=1}^k\frac{p_i}{p_{\nu}} \int_0^{\infty}\dot{\bss}^{\iota}_{i}(t-u)\beta_{i,\nu}e^{-\gamma^{(i)u}}(du)}\\
 &=-\bss^{\iota}_{\nu}(t)\sum_{i\in[k]}\frac{p_i}{p_{\nu}}\beta_{i,\nu}\bsi^{\iota}_{i}(t).
}
Thus the equations \eqref{IE} turn to the system of ODEs
\besn{\label{SIRODE}
     \frac{d}{dt}\bss^{\iota}_{ \nu}(t)&=-\bss^{\iota}_{\nu}(t)\bbbclr{\sum_{i\in[k]}\frac{p_i}{p_{\nu}}\beta_{i,\nu}\bsi^{\iota}_{i}(t)},\\
     \frac{d}{dt}\bsi^{\iota}_{ \nu}(t)&=\bss^{\iota}_{\nu}(t)\bbbclr{\sum_{i\in[k]}\frac{p_i}{p_{\nu}}\beta_{i,\nu}\bsi^{\iota}_{i}(t)}-\gamma \bsi^{\iota}_{\nu}(t),\\
     \frac{d}{dt}\bsr^{\iota}_{ \nu}(t)&=\gamma_{\nu} \bsi^{\iota}_{\nu}(t),
}
with initial conditions $ \bss^{\iota}_{\nu}(-\infty)=1$, $\bsi^{\iota}_{\nu}(-\infty)=0$, $\bsr^{\iota}_{\nu}(-\infty)=0$ and the constrain \eqref{psiEW}. The case in which $k=1$ makes the ODEs turn to the well known single type Markovian SIR ODEs or Kermack–McKendrick equations \eqref{deterministicSIR}.

\subsection{Example: Non-Markovian model with general infectious period}
Assume now that the contact process between any two individuals is Poisson and that $Q_i$ has a general distribution. The inter-arrival time has an exponential distribution $\Exp(\beta_{i,j}/n_j)$ and the distribution of the equilibrium excess lifetime $X^{\mathrm{e}}_{n,i,j}$ is also exponential (see for example \cite{bar2010characterization}), hence
\be{
   f_{X^{\mathrm{e}}_{n,i,j}}(t) =\frac{\int_t^{\infty}f_{i,j}(\tau)d\tau}{\IE X^{\mathrm{e}}_{i,j}}
   = e^{-\beta_{i,j} t/n_j}.
}
For the superposition process $\xi_{n,i,j}$ defined in \eqref{superposition1}, we have
\bes{
    \IE \xi_{n,i,j}([0,t])
    &=\IE\sum_{l=0}^{n_j}\I\cls{ X_{i,j,l}\leq t}\\
    &=n_j\bclr{1-e^{-\beta_{i,j} t/n_j}}\rightarrow \beta_{i,j} t,\quad n\rightarrow\infty.
}
Now,
\bes{
    \mel\IE\clc{\xi_{n,Q_i,i,j}([0,t])}\\
    &=\sum_{l=0}^{n_j} l\cdot\IP[\xi_{n,i,j}([0,t\wedge Q_{i}])=l]\\
    &=\sum_{l=0}^{n_j}\bbbclr{l\int_0^t\IP\cls{\xi_{n,i,j}[0,s]=l}f_{Q_{i}}(s)ds+l\IP\cls{\xi_{n,i,j}[0,t]=l}\IP\cls{Q_{i}\geq t}}\\
    &=\int_0^t\IE\clc{\xi_{n,i,j}[0,s]}f_{Q_{j}}(s)ds+\IE\clc{\xi_{n,i,j}[0,t]}\IP\cls{Q_{i}\geq t}\\
     &=\int_0^tn_j(1-e^{-\beta_{i,j} s/n_j})f_{Q_{i}}(s)ds+n_j(1-e^{-\beta_{i,j} t/n_j})F^{(i)c}_{Q_{i}}(t).
}
and
\bes{
    \mel\IE\clc{\xi_{n,Q_i,i,j}(\IR_+)}\\
    &=\int_0^{\infty}n_j(1-e^{-\beta_{i,j} s/n_j})f_{Q_{i}}(s)ds+\lim_{t\rightarrow\infty}n_j(1-e^{-\beta_{i,j} t/n_j})F_{Q_i}^c(t).
}
The relative intensity measure $G_{n,i,j}$ of $\xi_{n,Q_i,i,j}$ satisfies
\bes{
   \mel
   \IE\clc{\xi_{n,Q_i,i,j}(\IR_+)}G_{n,i,j}(dt)\\
   &=n_j(1-e^{-\beta_{i,j} t/n})f_{Q_{i}}(t)dt\\
   &\quad+\beta_{i,j} e^{-\beta_{i,j} t/n_j}F_{Q_i}^c(t)dt-n_j(1-e^{-\beta_{i,j} t/n_j})f_{Q_{i}}(t)dt\\
   &=\beta_{i,j} e^{-\beta_{i,j} t/n_j}F_{Q_i}^c(t)dt\rightarrow\beta_{i,j} F_{Q_i}^c(t)dt,\quad n\rightarrow\infty.
}
Recall \eqref{IE} and note that
\bes{
           \bss^{\iota}_{\nu}(t)&=\exp\bbbclc{-\sum_{i\in[k]}\frac{p_i}{p_{\nu}}\beta_{i,j}\int_0^{\infty}(1-\bss^{\iota}_{i}(t-u))F_{Q_i}^c(u)du},
}
which implies
\bes{
           \frac{d}{dt}\bss^{\iota}_{ \nu}(t)&=\bss^{\iota}_{\nu}(t)\bbclr{\sum_{i\in[k]}\frac{p_i}{p_{\nu}}\beta_{i,\nu}\int_{0}^{\infty}\dot{\bss}^{\iota}_{i}(t-u)dF^{(i)c}_{Q_{i}}(u)}\\
           &=-\bss^{\iota}_{\nu}(t)\bbbclr{\sum_{i\in[k]}\frac{p_i}{p_{\nu}}\beta_{i,\nu} \bsi_{i}(t)}.
}
In this case, the integral equations \eqref{IE} reduce to deterministic Volterra integral equations of the form
\ba{
           \bss^{\iota}_{\nu}&=1-\int_{-\infty}^t\bss^{\iota}_{\nu}(v)\bbbclr{\sum_{i\in[k]}\frac{p_i}{p_{\nu}}\beta_{i,\nu}\bsi_{i}^{\iota}(v)}dv,\\
          \bsi^{\iota}_{\nu}&= \gamma_{\nu}\int_{-\infty}^{t}F^c_{Q_\nu}\clr{t-v}\bss^{\iota}_{\nu}(v)\bsi^{\iota}_{j}(v)dv,\\
          \bsr^{\iota}_{\nu}&=\gamma_{\nu}\int_{-\infty}^{t}F_{Q_{\nu}}\clr{t-v}\bss^{\iota}_{\nu}(v)\bsi^{\iota}_{\nu}(v)dv.
}
These equations are consistent with the result of \cite[Theorem 2.1]{pang2022functional}. The solutions to these equations are deterministic and provide a complete description of the dynamics of the system.

\section{SIR process on a dynamic Erd\H{o}s-R\'enyi random graph}\label{MarkovSIRonrandomG}

In this section we consider a multi-type Markovian SIR process on a dynamic random graph and apply the results of Section~\ref{contempirical} to it.

Consider a population of $n$ vertices consists of $k$ types of individuals (the number of individuals of type~$i\in[k]$ is $n_i$, and $n=\sum_{i\in[k]}n_i$). Every edge between vertices of type~$i$ and vertices of type~$j$ is connected independently with probability $p_n(i,j)$, where $p_n(i,j)=\lambda_{n,i,j}/(\lambda_{n,i,j}+\mu_{n,i,j})$, $i$, $j\in[k]$; this model is also known as the \emph{Stochastic Block Model}. The graph changes by edges (between individuals of type~$i$ and $j$) appearing at rate~$\lambda_{n,i,j}$ and disappearing at rate~$\mu_{n,i,j}$, independently between any two vertices. Particularly, $\lambda_{n,i,j}=\lambda_{n,j,i}$ and $\mu_{n,i,j}=\mu_{n,j,i}$ for all $i$, $j\in[k]$. We assume the model starts in equilibrium state, and we randomly select a vertex of type~$\iota $ to be infectious and set this time as the initial time 0. Any infectious vertex, say $x$, of type~$i$ infects any of its connected neighbours of type~$j$ according to a Poisson point process $\varsigma_{n,x,i,j}$ whose intensity is $\beta_{n,i,j}$, and once infected, recovers (and becomes immune) at rate~$\gamma_i $. All the Poisson processes $\clc{\varsigma_{n,x,i,j},x\in[n_j],i,j\in[k]}$ are independent.

This model is over-parametrised, so in what follows, for $i$, $j\in[k]$, we keep the recovery rate~$\gamma_i $ constant as $n$ grows, but we allow 
parameters $\lambda_{n,i,j}$, $\mu_{n,i,j}$ and $\beta_{n,i,j}$ to change with $n$. In particular, to simplify the discussion somewhat, we assume  throughout that 
\be{
    \lambda_{n,i,j}=\lambda_{i,j} n_j^{\kappa_{\lambda_{i,j}}},\qquad
    \mu_{n,i,j}=\mu_{i,j} n_j^{\kappa_{\mu_{i,j}}},\qquad
    \beta_{n,i,j}=\beta_{i,j} n_j^{\kappa_{\beta_{i,j}}},
}
for some constants $\lambda_{i,j},\mu_{i,j},\beta_{i,j}\geq 0$ and some exponents $\kappa_{\lambda_{i,j}},\kappa_{\mu_{i,j}}\in\IR$ and $\kappa_{\beta_{i,j}}\leq 0$. 

Consider two individuals $x$ and $y$, where $x$ is from group $i$ and $y$ is from group $j$, and consider an independent $\{0,1\}$-valued Markov process (0 representing absence of the edge between $x$ and $y$, and 1 representing presence of the edge), started in equilibrium, that is, $\Be(p_n(i,j))$. Now, let $X^{\mathrm{e}}_{n,i,j}$ denote the arrival of the first point of the Poisson process $\varsigma_{n,i,j}$ for which simultaneously the corresponding Markov chain is in state 1. Let $X^{\mathrm{e}}_{n,i,j,1}, \ldots, X^{\mathrm{e}}_{n,i,j,n_j}$ denote independent and identically distributed copies of $X^{\mathrm{e}}_{n,i,j}$, for $i,j \in [k]$, and construct the processes defined in \eqref{superposition1} and \eqref{superposition2} as following
\bgn{\label{vary1}
   \xi_{n,i,j}= \sum_{l=1}^{n_j}\delta_{X^{\mathrm{e}}_{n,i,j,l}},
    \qquad
   \xi_{n,Q_i,i,j}= \sum_{l=1}^{n_j}\delta_{X^{\mathrm{e}}_{n,i,j,l}}(\cdot\cap[0,Q_{i}]),\\
   \label{vary2}
   \xi_{n,Q_i,i}= \bclr{\xi_{n,Q_i,i,1},\dots,\xi_{n,Q_i,i,k}},
}
where $Q_{i}$ is a random variable whose distribution is $\Exp(\gamma_i )$.
Denote the relative intensity measure of $\xi_{n,Q_i,i,j}$ by $G_{n,i,j}$; that is, for every Borel set $A\in \F{B}(\IR_+) $,
\be{
    G_{n,i,j}(A)=\frac{\IE \xi_{n,Q_i,i,j}(A)}{\bR_{0,n,i,j}}.
}
For $i$, $j\in[k]$, let
\be{
  \bR_{0,i,j}=\IE \xi_{Q_i,i,j}(\IR_+).
}

Before we introduce the main result of the present section, we will show the following lemma.
\begin{lemma}\label{Gxi}
For $i$, $j\in[k]$, suppose the parameters $\lambda_{i,j}$, $\mu_{i,j}$, $\beta_{i,j}$, $\gamma_i$ as well as  $\kappa_{\lambda_{i,j}}$, $\kappa_{\mu_{i,j}}$ and $\kappa_{\beta_{i,j}}$ are given in Table~\ref{singletable} (if $k=1$) and Table~\ref{table:constraints} (if $k>1$), then the vector of randomly truncated point processes $\xi_{n,Q_i,i}$ convergences to a vector of point processes $\xi_{Q_i,i}=\bclr{\xi_{Q_i,i,1},\dots,\xi_{Q_i,i,k}}$ in total variation distance, and the intensity measure $G_{n,i,j}$ also convergences to the measure $G_{i,j}$ in total variation distance. 
\end{lemma}

The vectors $\clc{\xi_{Q_i,i}}_{i\in[k]}$ can be derived in the same way as the vectors of Cox processes, $\clc{\xi_{Q_i,i}}_{i\in[k]}$, in Section~\ref{contempirical}.

We will see in the later section that the measure $G_{i,j}$ is absolutely continuous with respect to the Lebesgue measure and the density has the form \be{
  g_{i,j}(t)=f_{\xi_{Q_i,i,j}}(t)e^{-\gamma_i t}.
} 
We call $G_{i,j}$ \emph{homogeneous} if $f_{\xi_{Q_i,i,j}}(t)$ is constant, and we call it \emph{non-homogeneous} otherwise.

In addition, we make the following assumption.
 
\begin{assumption}\label{Mi,j2} 
The vectors of processes $\clc{\xi_i}_{i\in[k]}$ satisfy Assumption~\ref{dTVxi}.
\end{assumption}

\subsection{Large population limit}

For this multi-type Markovian SIR process on a dynamic random graph, we will show that it can be analyzed using the results of Sections~\ref{general} and~\ref{contempirical}. Specially, in the forward branching process  $\C{B}^{\iota\prime}$, individuals of type~$i$ with an exponential distributed $\Exp(\gamma_i )$ life length give births of type~$j$ according to the point process $\xi_{Q_i,i,j}$. In the multi-type backward branching process (starting form an individual of type~$\nu$) $\widehat{\C{B}}^{\nu\prime}$, individuals of type~$j$ give births of type~$i$ according to independent Poisson point processes with intensity $p_i\bR_{0,i,j}G_{i,j}(da)/p_j$.

The time $t_n(u)$, the random time $\tau_n$ and the event $A_n$ are defined in the same way as those in Section~\ref{general}.

\begin{theorem}\label{dynamicmodel}
Under Assumptions~\ref{initial} and~\ref{Mi,j2}, and with the parameters given in Table~\ref{table:constraints}, there exists a coupling between $\C{B}^{\iota}_n$ and $\C{B}^{\iota\prime}$ such that the following holds. For any $\varepsilon>0$, $T\in\IR$, and $\nu\in[k]$, the epidemic ratios $\bss^{\iota}_{\nu}$, $\bsi^{\iota}_{\nu}$ and $\bsr^{\iota}_{\nu}$ satisfy \eq{maintheoremP}, \eq{FLLNir} and \eq{FLLNr}, as well as
the ordinary differential equations
\ba{
  \frac{d}{dt}\bss^{\iota}_{ \nu}(t)
  &=-\bss^{\iota}_{\nu}(t)\bbbbcls{\sum_{i\in K^{\nu}_{\mathrm{h}}}\frac{p_i}{p_{\nu}}\bbbclr{\bR_{0,i,\nu}\gamma_i \bsi^{\iota}_{i}}
  +\sum_{u\in K^{\nu}_{\mathrm{nh}}}\frac{p_u}{p_{\nu}}\beta_{u, \nu}\clr{\bsl^{\iota}_{\mathrm{c}, u, \nu}(t)+ \bsl^{\iota}_{\mathrm{d}, u, \nu}(t)}},\\
  \frac{d}{dt}\bsi^{\iota}_{ \nu}(t)
  &=\bss^{\iota}_{\nu}(t)\bbbbcls{\sum_{i\in K^{\nu}_{\mathrm{h}}}\frac{p_i}{p_{\nu}}\bbbclr{\bR_{0,i,\nu}\gamma_i \bsi_{i}}+\sum_{u\in K^{\nu}_{\mathrm{nh}}}\frac{p_u}{p_{\nu}}\beta_{u, \nu}\clr{\bsl^{\iota}_{\mathrm{c}, u, \nu}(t)+ \bsl^{\iota}_{\mathrm{d}, u, \nu}(t)}-\gamma_{\nu}\bsi^{\iota}_{\nu}(t)},\\
    \frac{d}{dt}\bsr^{\iota}_{ \nu}(t)
    &=\gamma_{\nu}\bsi^{\iota}_{\nu}(t),\\
  \frac{d}{dt}\bsl^{\iota}_{\mathrm{c}, u, \nu}(t)
  &=\bss^{\iota}_{u}(t)\bbbbcls{\sum_{i\in K^u_\mathrm{h}}\frac{p_i}{p_{\nu}}\bbbclr{\bR_{0,i,u}\gamma_i \bsi_{i}}+\sum_{h\in K^u_\mathrm{nh}}\frac{p_h}{p_{\nu}}\frac{\lambda_{hu}}{\mu_{hu}}\beta_{h\nu}\clr{\bsl^{\iota}_{\mathrm{c},h,\nu}(t)+ \bsl^{\iota}_{\mathrm{d},h,\nu}(t)}} \\
  &\quad- (\mu_{uj}+\beta_{u, \nu}+\gamma_u ) \bsl^{\iota}_{\mathrm{c}, u, \nu}(t),\\
  \frac{d}{dt}\bsl^{\iota}_{\mathrm{d}, u, \nu}(t)
  &=\mu_{u, \nu} \bsl^{\iota}_{\mathrm{c}, u,    \nu}(t)-\gamma_u  \bsl^{\iota}_{\mathrm{d}, u, \nu}(t),
}
where
\ba{
         K^{\nu}_{\mathrm{h}}&=\clc{i \in [k]:\text{$G_{i,\nu}$ is homogeneous}}, \\
         K^{\nu}_{\mathrm{nh}}&=\clc{u\in [k]:\text{$G_{u,\nu}$ is non-homogeneous}}.
}
The initial conditions for the above system of equations are
\be{
  \bss^{\iota}_{j}(-\infty)=1,\qquad\bsi^{\iota}_{\nu}(-\infty)=0,\qquad\bsr^{\iota}_{\nu}(-\infty)=0,
}
and
\be{
   \lim_{t\rightarrow-\infty}e^{-\widehat{\C{M}}t}\clr{\bss^{\iota}_{\nu}(t)}'=-m_*\widehat{\C{M}}\IE\widehat{W}^{\nu},
}
and in addition,
\be{
  \bsl^{\iota}_{\mathrm{c}, u, 
\nu}(-\infty)=0,
\qquad\bsl^{\iota}_{\mathrm{d}, u, 
\nu}(-\infty)=0.
}
\end{theorem}

\subsection{The basic reproduction number \texorpdfstring{$\bR_0$}{Lg}, final size, and  maximum ratio of infected population \texorpdfstring{$i_{\max}$}{Lg}}\label{analysis}

The limiting reproduction ratios are listed in Tables~\ref{singletable} and~\ref{table:constraints} and are a result of calculating $\bR_{0,i,j} = \lim_{n\to\infty} \bR_{0,n,i,j}$; the final size equations can be derived from \eqref{IE} by letting $t\rightarrow\infty$ and solving 
\ben{\label{finalsize}
    -\log \bss^{\iota}_{j}(\infty)=\sum_{i\in[k]}\bR_{0,i,j}(1-\bss^{\iota}_{i}(\infty)),\quad j\in[k],
}
The relation between the parameters and the final size of the epidemic can be extracted numerically from equations \eqref{finalsize}.

Let $t_{\max}$ be the time at which the infected curve $\bsi$ reaches its maximum value. We are now interested in the quantity $\bsi_{\max}=\bsi(t_{\max})$. In the homogenous cases, a classical result for equations \eqref{deterministicSIR} is that 
\be{
    \bsi(t)+\bss(t)-\frac{1}{\bR_{0}}\ln \bss(t)\equiv \text{constant},\qquad t\in \IR.
}
Letting $t\rightarrow-\infty$ we know that the constant is 1. Since we know $\bsi_{\max}$ occurs when $\frac{d\bsi}{dt}=0$ and when $s=\frac{1}{\bR_{0}}$, we can rearrange and substitute the above equation to find $\bsi_{\max}$, the maximum
ratio of individuals we expect to be infected at any point in the course of the epidemic. Thus we obtain that
\be{
          \bsi_{\max}+\frac{1}{\bR_{0}}-\frac{1}{\bR_{0}}\ln \frac{1}{\bR_{0}}=1;
 }
that is,
\be{
          \bsi_{\max}=1-\frac{1}{\bR_{0}}+\frac{1}{\bR_{0}}\ln \frac{1}{\bR_{0}}.
}
In the non-homogeneous case, we use \eqref{singleSLLR}. Since
\be{
  \bsl_{\mathrm{d}}(t)= \frac{\lambda}{\mu+\beta} \bsi(t)-\frac{\mu}{\mu+\beta}\bsl_{\mathrm{c}}(t),
}
and
\be{
  \bsl_{\mathrm{c}}(t)\leq \frac{\lambda}{\mu}\bsi(t),
}
we know that
\be{
        \frac{\lambda\beta}{\mu+\beta}\bsi(t) \bss(t)-\gamma \bsi(t)\leq \bsi'(t)\leq  \frac{\lambda\beta}{\mu}\bsi(t) \bss(t)-\gamma \bsi(t).
}
Let $t_1$ be the time point at which
\be{
  \bss(t_1)=\frac{(\mu+\beta)\gamma}{\lambda\beta}
}
and $t_2$ be the time point at which
\be{
  \bss(t_2)=\frac{\mu\gamma}{\lambda\beta},
}
then 
\be{
  t_1\leq t_{\max}\leq t_2.
}
The quantities $\frac{(\mu+\beta)\gamma}{\lambda\beta}$ and $\frac{\mu\gamma}{\lambda\beta}$ are thresholds points, the times for which represent specific times before and after the peak of the epidemic.
These thresholds are meaningful only when they are smaller than 1, although the smaller one is always valid in our framework. Indeed, the parameters are selected to make the expected number of infections  $\bR_{0}$ made by an individual to be greater than 1, and thus
\be{
   \frac{\mu\gamma}{\lambda\beta}=\frac{\mu+\gamma}{\bR_{0}(\mu+\beta+\gamma)}<1.
}

\section{Proof of Theorem~\ref{maintheorem}}\label{proof1}
Throughout we assume that all random variables and processes are defined on a suitable
rich enough probability space, which we do not specify.

\subsection{Epidemic process \texorpdfstring{$\C{B}_n^{\iota}$}{Lg} and forward branching process \texorpdfstring{$\C{B}^{\iota\prime}$}{Lg}} 
The epidemic process $\C{B}_n^{\iota}$ can be coupled to a branching process $\C{B}^{\iota\prime}$ successfully during early stage, see \cite[Theorem 2.1]{ball1995strong}. The key difference between the epidemic process and the forward branching process lies in the fact that the former does not increase the number of infected individuals if an infected individual is contacted again by another infected vertex. Conversely, ``repeat contacts'' in the branching process lead to new births, thereby increasing the number of offspring. To better distinguish between these two processes, we assign labels in chronological order to individuals of type~$i\in[k]$ born in the branching process, drawing them independently and at random from $\{(i,1), (i,2), \ldots, (i,n_i)\}$, and we mark individuals and all of their descendants as ``ghosts'' if their label has been used before. In this way, we can identify individuals in the epidemic process as those who are not ghosts in the branching process.

Let $L(t)$ denote the number of times that a label has
been used before, creating an initial ghost, and let $L_{+}(t)\geq L(t)$ denote the number of initial ghosts and their descendants by time $t$. By the Assumption~\ref{dTVxi}, we write $e_{W}=\max_{i\in[k]}\sup_t\IE\clc{|B^{i\prime}(t)|e^{-\C{M}t}}<\infty$ and $W^{\iota}(t)=|B^{\iota\prime}\clr{t}|e^{-\C{M}t}$. 
\begin{lemma}\label{L+17/24}
Under Assumption~\ref{initial} and~\ref{dTVxi}, 
\begin{enumerate}[label=(\roman*)]
\item $\IP\bcls{\clc{L_{+}\clr{\tau_n}\geq n^{11/24}}\cap \clc{W^{\iota}\clr{\tau_n} \geq n^{-1/48}}}= \bigo(n^{-1/48}\log n)$;
\item $\lim_{n\to\infty}\IP\cls{A_n}>0$.
\end{enumerate}
\end{lemma}
\begin{proof}
The first result is proved in much the same way as the proof of \cite[Lemma 2.7]{barbour2013approximating}. Denote $p^*=\min\clc{p_i,i\in[k]}$. For any of the first $\lfloor n^{17/24}\rfloor$ indices chosen, the probability that it is a repeat
of an index chosen earlier is at most $\lfloor n^{17/24}\rfloor/(\min\clc{n_i,i\in[k]})\sim (p^*)^{-1}n^{-7/24}$. For any given $\alpha$, let $t^{\alpha}_n=\frac{\alpha}{\C{M}}\log n$. Note that   at time $t^{\alpha}_n$, the expected number of descendants of an individual born at $t$ is at most $e_{W}e^{\C{M}(t^{\alpha}_n-t)}$. Hence, 
\bes{
    \mel\IE\clc{L_{+}\clr{\tau_n\wedge t^{\alpha}_n}}\\
    &\leq (p^*)^{-1}n^{-7/24}\IE\bbbclc{\int_0^{t^{\alpha}_n} e_{W} e^{\C{M}(t^{\alpha}_n-t)} |B^{\iota\prime}|(dt)}\\
    &\leq (p^*)^{-1}n^{-7/24} e_{W} e^{\C{M} t^{\alpha}_n}\IE\bbbclc{e^{-\C{M} t^{\alpha}_n}|B^{\iota\prime}(t^{\alpha}_n)|+\C{M}\int_0^{t^{\alpha}_n} e^{-\C{M}t}|B^{\iota\prime}(t)|dt}\\
    &\leq (p^*)^{-1}n^{-7/24} e^2_{W} e^{\C{M} t^{\alpha}_n} (1+\C{M} t^{\alpha}_n).
}
Thus, choosing $\alpha=(17+\varepsilon)/24$, we obtain
\bes{
     \mel n^{11/24}\IP\cls{L_{+}\clr{\tau_n\wedge t^{\alpha}_n}\geq n^{11/24}}\\
      &\leq
      \IE\clc{L_{+}\clr{\tau_n\wedge t^{\alpha}_n}}\\
      &\leq (p^*)^{-1}n^{(10+\varepsilon)/24} e^2_{W}  \bbclr{1+\frac{17+\varepsilon}{24}\log n},
}
and thus
\be{
     \IP\bcls{L_{+}\clr{\tau_n\wedge t^{\alpha}_n}
     \geq n^{11/24}}
     \leq(p^*)^{-1}n^{\frac{\varepsilon-1}{24}} e^2_{W}  \bbclr{1+\frac{17+\varepsilon}{24}\log n}.
}
When $W^{\iota}\clr{\tau_n} \geq n^{-\varepsilon/24}$, that is, when 
\be{
    |B^{\iota\prime}\clr{\tau_n}|e^{-\C{M}\tau_n}= \lfloor n^{17/24}\rfloor e^{-\C{M}\tau_n}\geq n^{-\varepsilon/24},
}
we have
\be{
    n^{17/24}  e^{-\C{M}\tau_n}\geq n^{-\varepsilon/24},
    \qquad
   \tau_n \leq\frac{\alpha}{\C{M}}\log n=t_n^{\alpha}.
}
The proof of $(i)$ is completed after taking $\varepsilon=1/2$.

In order to prove $(ii)$, using Markov's inequality, we easily obtain 
\be{
  \IP\bcls{|B^{\iota\prime}(\C{M}^{-1}\log n^{1/24})|\geq n^{1/12}}= \bigo(n^{-1/24} ),
} which implies 
\be{
  \IP\bcls{\tau_n\geq \C{M}^{-1}\log n^{1/24}}=1- \bigo(n^{-1/24} ).
} 
From \eqref{forwardB}, we then have \be{
  \lim_{t\to\infty}\IP[|W^{\iota}(t)-W^{\iota}|>\varepsilon_1]= 0
}
for any given $\varepsilon_1>0$. Furthermore, we observe that for the given $\varepsilon_1$ and large enough $n$, we have 
\bes{
  \mel\bbclc{\clc{\tau_n\geq \C{M}^{-1}\log n^{1/24}}\cap\clc{|W^{\iota}(\tau_n)-W^{\iota}|\leq\varepsilon_1}\cap\clc{W^{\iota} \geq n^{-1/48}+\varepsilon_1}}\\
   &\subset\bbclc{\clc{\tau_n\geq \C{M}^{-1}\log n^{1/24}}\cap \clc{|W^{\iota}(\tau_n)-W^{\iota}|\leq\varepsilon_1}\cap\clc{W^{\iota}\clr{\tau_n} \geq n^{-1/48}}}.
}
Moreover, it holds that $\clc{W^{\iota}>0}=\clc{\lim_{t\to\infty}|B^{\iota\prime}(t)|=\infty}~a.s.$ (see \cite[(3.10)]{nerman1981convergence}). Under Assumption~\ref{dTVxi}(iii), we have $\IP\cls{\clc{W^{\iota}>0}}>0$. Since $\varepsilon_1$ is arbitrary and $A_n= \clc{W^{\iota}\clr{\tau_n} \geq n^{-1/48}}\cap\clc{|B^{\iota\prime}\clr{\tau_n}|-n^{11/24}\leq|B^{\iota}\clr{\tau_n}|\leq |B^{\iota\prime}\clr{\tau_n}|}$, we conclude that $\lim_{n\to\infty}\IP\cls{A_n}>0$.
\end{proof}

\subsection{Susceptibility process and backward branching process}\label{sectionsusbranching}
The expected final size of a major outbreak, which refers to the proportion of individuals who eventually recover, can be represented by the survival probability of an approximate backward branching process. This branching process provides a good approximation of an individual's susceptibility set from a generation perspective, which is a crucial tool for determining the fraction of the population that will ultimately be infected during a major outbreak. The concept of susceptibility sets was first introduced by \cite{ball2001stochastic,ball2002general, ball2009threshold} to analyze epidemics in discrete time or from a generational perspective. \cite{britton2019infector} then introduced a dynamical extension of susceptibility sets called the \emph{susceptibility process}.

We can associate an epidemic process on a graph with a random graph $\C{G}(\cV_n,E_n)$, where $\cV_n$ represents all individuals in the graph. Let $\cV_{n,i}$ denote the set of all $n_j$ individuals of type~$i$ with $i\in[k]$, and $E_n$ denote the set of all weighted directed edges in $\C{G}$. The edge set $E_n$ consists of all directed pairs $(a,b)\in \cV_n\times \cV_n$, with $a\neq b$. For an edge $(a,b)\in \cV_n\times \cV_n$ from $a$ to $b$, we refer to $a$ as the tail and $b$ as the head. The length of edge $(a,b)$ is defined as the time from $a$ becoming infected until it infects or contacts $b$. The length of an edge could be infinite, but for our arguments, we restrict to the weighted edge set $E_n'\subset E_n$ of all edges $(a,b)\in E_n$ with finite length. We denote the corresponding directed random graph of $\C{G}(\cV_n,E_n)$ with edges of finite length as $\C{G}'(\cV_n,E'_n)$.

To construct the susceptibility process $\F{S}_{n,v}$ of a randomly chosen individual $v$ of type~$\nu$, we adopt the method developed by \cite[Section 3.1]{britton2019infector}. The process is denoted by ${\F{S}_{n,v}(t),t\geq0}$, where the time parameter counts backward. The construction is based on a weighted directed random graph representation $\C{G}'(V,E')$ of the population and the epidemic on it. Here, we provide a brief introduction to the construction. 

For a given time instant $t'$, the susceptibility process $\{\F{S}_{n,v}(t),t\in[0,t']\}$ is derived by  constructing part of the random directed graph $\C{G}'$ around vertex $v$ by means of
the process of growing graphs $\{\C{\hat{G}}(l)\} = \{\C{\hat{G}}(l), l\in\mathbb{N}_0=\mathbb{N}\cup\{0\}\}$ in which vertices in the susceptibility process are ``explored'' one at a time.

For $l\in\mathbb{N}_0$, $\C{\hat G}$ is the 4-tuple
\be{
    \C{\hat G}(l)=\{\hat{V}^a(l),\hat{V}^e(l),\hat{V}^p(l),\hat{E}(l)\}.
}
Here $\hat{E}(l)$ denotes the edge set of $ \C{\hat G}(l)$. Vertices in $ \C{\hat G}(l)$ can be ``active'', ``passive'', or ``explored''.
The sets of these vertices are denoted by $\hat{V}^a(l)$, $\hat{V}^e(l)$, $\hat{V}^p(l)$ respectively. For a given time $t'$ before the construction is ``flagged'' (the construction is flagged means the newly added vertex was chosen before; that is, the construction  is valid  before the flagged time, see \cite[Construction step 4]{britton2019infector}), there exists a constant $l'$ such that for every vertex, say $v'$, in $\hat{V}^e(l')$ of $\C{\hat G}(l')$, there exists a path with time length no larger than $t'$ starts from $v'$ and ends at $v$. That is, $\hat{V}^e(l')=\F{S}_{n,v}(t')$. 

The vertices in $\hat{V}^a(l')$ also have infection chains running forward to $v$ but with time length larger than $t'$. For each vertex in $\hat{V}^a(l')$, there exists an edge in $\hat{E}(l')$ connecting that vertex (as tail) and one of vertices of $\hat{V}^e(l')$ (as head), see \cite[Construction step 4  and note (c) in Section 3.1]{britton2019infector}. During the construction, the active vertex was added into $\C{G}(l')$ according to every explored vertex with a binomial random number with parameters $n_j$ and $\bR_{0,i,j}/n_j$ (where the explored one is of type~$j$ and the newly added one is of type~$i$), and some of the active vertices were converted to explored ones eventually; that is
\ben{\label{Va}
    |\hat{V}^a(l')| 
    \leq \sum_{j\in[k]}\sum_{u\in \hat{V}^e(l')\cap \cV_{n,j}}\sum_{i\in[k]}x_{u,i},
}
where $x_{u,i}\sim \Bi\clr{n_j,\bR_{0,i,j}/n_j}$ when $u\in \cV_{n,j}$.

For every vertex, say $a$ of type~$i$, in $\hat{V}^p(l')$, it is an individual that would be contacted by a vertex, say $b$ of type~$j$, in $\hat{V}^a(l')\cup\hat{V}^e(l')$ corresponding to one point in the point process $\xi_{b,Q_i,i,j}$, see \cite[Construction step 4-(4b)  and  note (d) in Section 3.1 ]{britton2019infector}. That is
\be{\label{Vp}
    |\hat{V}^p(l')|\leq \sum_{i\in[k]}\sum_{u\in\clr{\hat{V}^a(l')\cup\hat{V}^e(l')}\cap \cV_{n,i}}\sum_{j\in[k]}\xi_{u,Q_i,i,j}(\IR_+).
}
Define, for $a\geq0$,
\be{\F{S}_{n,v}(t;a,j)=\cV_{n,j}\cap\clc{\F{S}_{n,v}(t)/\F{S}_{n,v}(t-a)}.
}
That is, $\F{S}_{n,v}(t;a,j)$ consists of the vertices of type~$j$ in $\F{S}_{n,v}(t)$, that are not part of $\F{S}_{n,v}(t-a)$, or in
other words, the vertices of type~$j$ that join $v$'s susceptibility process between times $t-a$ and $t$. Let $\widehat B^{\nu\prime}_{j}(t;a,j)$ be the number of vertices of type~$j$ that have age less than $a$ in the branching process~$\widehat{\C{B}}^{\nu\prime}_{j}$ at time t.

We will see in the following that the susceptibility process construction is not ``flagged''  before time $\frac{1}{3\widehat{\C{M}}}\log n$ with high probability. 
\begin{lemma}\label{susbranching}
There exists a probability space on which we can define the branching process $\widehat{\C{B}}^{\nu\prime}$ and the susceptibility process $\F{S}_{n,v}$ such that for every $t^* \in\bbclr{0,\frac{1}{3\widehat{\C{M}}}\log n}$,
\be{
    \lim_{n\to\infty}\IP\bcls{
    \text{$|\F{S}_{n,v}(t;a,j)| =\widehat{B}^{\nu\prime}(t;a,j)$ for all $t\leq t^*$, $a\in[0,t^*]$, $j\in[k]$}} =1.
}
\end{lemma}
\begin{proof}
Before the time, say $t^*$, at which the construction is still not flagged, the susceptibility process $\{\F{S}_{n,v}(t')=\hat{V}^{e}(l')\}_{0\leq t'\leq t^*}$ is equivalent to a branching process $\widehat{\C{B}}^{\nu\prime}_{\mathrm{bin},v}$. In this branching process, an individual of type~$j$ gives birth to a binomial distributed number of individuals of type~$i$, where the parameters of the binomial-distributed random variable are $n_i$ and $\bR_{0,i,j}/n_j$, see \cite[Section 3.1]{britton2019infector}. Furthermore, the ages of the ``mother'' particles at birth of a child are independent and have density $\IE\clc{\xi_{Q_i,i,j}(da)}/\bR_{0,i,j}$, where $a$ is the age of the individual. The branching process $\widehat{\C{B}}_{\mathrm{bin},v}$ can be coupled to a branching process $\widehat{\C{B}}^{\nu\prime}_{v}$ (or denoted by $\widehat{\C{B}}^{\nu\prime}$) in such a way that individuals of type~$j$ give births to individuals of type~$i$ according to a Poisson point process with intensity $p_i\IE\clc{\xi_{Q_i,i,j}(da)}/p_j$ with failure probability $\bigo(n^{-1/2})$ as long as the number of individuals born in any of the two branching processes is $\lito(n^{1/2})$; see \cite[Section 3.3]{britton2019infector}.

We now consider the flagged time. We will show that the flagged time for the construction lies within the interval $(0, 1/3\widehat{\C{M}}\log n)$. This result is proven using a similar approach to the proof of \cite[Lemma 3.1]{britton2019infector}. For $t^*\in\bbclr{0,\frac{1}{3\widehat{\C{M}}}\log n}$, we know that
\be{
    \IE |\hat{V}^e(l^*)| \leq \IE|\widehat{B}^{\nu\prime}(t^*)|\leq  e_{\widehat{W}}n^{1/3}.
}
By Markov's inequality we obtain that
\be{
    \IP\bbcls{|\hat{V}^e(l^*)|> n^{9/24}}\leq e_{\widehat{W}}n^{-1/24}.
}
Conditional on the event $A_{n,1}= \clc{|\hat{V}^e(l^*)|\leq n^{9/24}}$ and according to the estimate \eqref{Va}, we obtain
\bes{
    \mel\IE\bclc{|\hat{V}^a(l^*)| \given A_{n,1}} \\
    &\leq\IE\bbbclc{\sum_{j\in[k]}\sum_{v\in \hat{V}^e(l^*)\cap V^{(j)}_n}\sum_{i\in[k]}x_{v,i} \given A_{n,1}}\leq \bbbclr{\sum_{i,j\in[k]}\bR_{0,i,j}}n^{9/24},
}
and thus
\be{
    \IP\bcls{|\hat{V}^a(l^*)|> n^{10/24} \given A_{n,1}}
    \leq \bbbclr{\sum_{i,j\in[k]}\bR_{0,i,j}} n^{-1/24}.
}
That is, there exists an event $A_{n,2}= A_{n,1}\cap \clc{|\hat{V}^a(l^*)|\leq n^{10/24}} $ on which
\be{
    |\hat{V}^a(l^*)\cup\hat{V}^e(l^*)|
    \leq 2n^{10/24}.
}
Conditioned on the event $A_{n,2}$ and by the estimate \eqref{Vp} we know that
\bes{
    \IE\clc{|\hat{V}^p(l^*)| | A_{n,2}}
    & \leq \IE\bbbclc{\sum_{i\in[k]}\sum_{v\in\clr{\hat{V}^a(l^*)\cup\hat{V}^e(l^*)}\cap \cV_{n,i}}\sum_{j\in[k]}\xi_{v,Q_i,i,j}(\IR_+) \given A_{n,2}}\\
    & \leq 2\bbbclr{\sum_{i,j\in[k]}\bR_{0,i,j}}n^{10/24},
}
and thus
\be{
    \IP\bcls{|\hat{V}^p(l^*)|> n^{11/24} \given A_{n,2}}
    \leq2\bbbclr{\sum_{i,j\in[k]}\bR_{0,i,j}}n^{-1/24}.
}
Thus there exists an event $A_{n,3}= A_{n,2}\cap\clc{|\hat{V}^p(l^*)|\leq n^{11/24}}$ on which
\be{
   |\hat{V}^a(l^*)|
   \leq n^{9/24},
   \qquad |\hat{V}^e(l^*)|\leq n^{10/24},
   \qquad|\hat{V}^p(l^*)|\leq n^{11/24},
}
and $\IP\cls{A_{n,3}^c}= \bigo(n^{-1/24})$.

We can use birthday-problem-like arguments (\cite[p.\,24]{grimmett2001probability}) to show that the probability that the first $|\hat{V}^a(l^*)\cup \hat{V}^e(l^*)|$ activated vertices in $\{\C{\hat{G}}(l), l\in\mathbb{N}_0\}$ are not all different is bounded from above by
\bes{
    n^{-1} (|\hat{V}^a(l^*)|+ |\hat{V}^e(l^*)|)^2
    = \bigo(n^{-1/6}),
}
while the probability that there are vertices in $\hat{V}^a(l^*)\cup\hat{V}^e(l^*)$ when the construction of susceptibility process ``tries to include'' in
$\{\hat{V}^p(l),l\in\mathbb{N}_0\}$ among the first $|\hat{V}^p(l^*)|$ vertices, is bounded from above by
\be{
    |\hat{V}^p(l^*)| \frac{|\hat{V}^a(l^*)|+ |\hat{V}^e(l^*)|}{n}
    = \bigo(n^{-1/8}),
}
which completes the proof.
\end{proof}

From the above argument, we know that $\widehat{\C{B}}^{\nu\prime}$ is constructed in backward time up to $t_n(u)$ and can be used to approximate the susceptibility process with high probability. We now proceed to discuss the epidemic process in forward time. The method of \cite{britton2019infector} enables us to construct the epidemic process $\C{B}_n^{\iota}$ in forward time. It is shown by \cite[Section 4.1, Proof of Theorem 1]{britton2019infector} how the epidemic process and the susceptibility process match to each other. While when these two processes disjoint to each other, the epidemic process is constructed in the original way (that is, infected individuals of type~$i$ infect their neighbors of type~$j$ via independent point processes identically distributed with $\xi_{n,Q_i,i,j}$) which can be coupled to a branching process at the initial stage as introduced before. We summarise the above argument as the following proposition.

\begin{proposition}
Using the method developed by \cite{britton2019infector}, we can construct the susceptibility process $\F{S}_{n,v}$ for a randomly chosen individual $v$ of type~$\nu$, as well as the epidemic process $\C{I}^{\iota}_n, $ in the same probability space. This construction is carried out before the point at which the two processes share any individual(s). We can approximate these processes $\F{S}_{n,v}$ and $\C{I}_n^{\iota}$ using the branching processes $\widehat{\C{B}}^{\nu}$ and $\C{B}^{\iota}$ respectively.
\end{proposition}

\subsection{Proof of Theorem~\ref{maintheorem}}\label{proof2.3}
We can now prove Theorem~\ref{maintheorem} using a method similar to that used by \cite[Proof of Theorem 2.10]{barbour2013approximating} and \cite[Proof of Theorem 3.2]{barbour2014couplings}.
 
According to \cite{britton2019infector}, the randomly chosen $v$ of type~$\nu$ is still susceptible at time $\tau_n+t_n(u)$ if and only if the intersection set of $\C{B}^{\iota}_n(\tau_n)$ and $\F{S}_{n,v}(t_n(u))$ is empty. Note that for independently and randomly chosen $v$
and $v'$ both of type~$j\in[k]$,
\be{
    \IE\clc{n_{\nu}^{-1}S^{\iota}_{n,\nu}(t)\given A_n}
    =\IP\cls{v\in \cS^{\iota}_{n,\nu}(t)\given A_n},
}
and similarly
\be{
   \Var\clr{n_{\nu}^{-1}S^{\iota}_{n,\nu}\clr{t}   \given A_n}
   =\IP\cls{\{v,v'\}\in \cS^{\iota}_{n,\nu}(t)\given A_n}-\IP\cls{v\in \cS^{\iota}_{n,\nu}(t)\given A_n}^2;
}
see for example \cite[Section 2.4]{barbour2013approximating}.

So the expectation of the proportion of individuals that have not been infected by time $\tau_n+t_n(u)$, which is
the same as the probability that the randomly chosen individual $v$ is still susceptible, is

\bes{
         \IP\cls{v\in \cS^{\iota}_{n,\nu}\clr{\tau_n+t_n(u)}   | A_n} 
         &=\IP\cls{\C{B}^{\iota}_n\clr{\tau_n}\cap\F{S}_{n,v}(t_n(u))=\emptyset\given A_n}\\
         &\leq \IP\cls{\C{B}^{{\iota\prime}}\clr{\tau_n}\cap\F{S}_{n,v}(t_n(u))=\emptyset\given A_n}.
}
We know that
\bes{
    \mel\IP\cls{\C{B}^{{\iota\prime}}\clr{\tau_n}\cap\F{S}_{n,v}(t_n(u))=\emptyset\given A_n}\\
    &=\IP\cls{\clc{\C{B}^{{\iota\prime}}\clr{\tau_n}\cap\F{S}_{n,v}(t_n(u))=\emptyset}\cap  A_{n,3}   \given A_n}\\
    &\quad+ \IP\cls{\clc{\C{B}^{{\iota\prime}}\clr{\tau_n}\cap\F{S}_{n,v}(t_n(u))=\emptyset} \cap  A_{n,3}^c  \given A_n}\\
    &\leq \IP\cls{\clc{\C{B}^{\iota\prime}\clr{\tau_n}\cap\C{\widehat{B}}^{\nu\prime}(t_n(u))=\emptyset}\cap A_{n,3} \given A_n} + \IP\cls{A_{n,3}^c | A_n},
}
where $A_{n,3}$ is the event defined in Section~\ref{sectionsusbranching}.
Note that the intersection set of the two vectors, $\C{B}^{\iota\prime}\clr{\tau_n}$ and $\widehat{\C{B}}^{\nu\prime}(t_n(u))$ is empty if and only if for every $i\in[k]$, $\C{B}^{\iota\prime}_{ i}\clr{\tau_n}\cap \widehat{\C{B}}^{\nu\prime}_i(t_n(u))=\emptyset$. The distribution of the common members of the 
sets $\C{B}^{\iota\prime}_{ i}\clr{\tau_n}$ and $\widehat{\C{B}}^{\nu\prime}_i(t_n(u))$, conditional on their sizes being known, is the hypergeometric distribution 
\be{
  H\clr{n_i,B^{\iota\prime}_{ i}\clr{\tau_n},\widehat{B}^{\nu\prime}_i(t_n(u))}
,} 
and the mean number of the intersection set is
\be{
   e_{\iota, \nu, i}=\frac{B^{\iota\prime}_{ i}\clr{\tau_n}\widehat{B}^{\nu\prime}_i(t_n(u))}{n_i}.
}
According to a standard Poisson approximation argument (see for example \cite[Theorem 6.A]{barbour1992poisson}), the total variation distance between the hypergeometric distribution $H\clr{n_i,B^{\iota\prime}_{ i}\clr{\tau_n},\widehat{B}^{\nu\prime}_i(t_n(u))}$ and Poisson distribution $\Po\clr{e_{\iota, \nu, i}}$ is at most
\bes{
    \mel\clr{1-e^{-e_{\iota, \nu, i}}}\frac{n_i}{n_i-1}\bbbclr{\frac{B^{\iota\prime}_{ i}\clr{\tau_n}}{n_i}+\frac{\widehat{B}^{\nu\prime}_i(t_n(u))}{n_i}
    -\frac{B^{\iota\prime}_{ i}\clr{\tau_n}\widehat{B}^{\nu\prime}_i(t_n(u))}{n_i^2} -\frac{1}{n_i}}\\
    &\leq  \frac{1}{n_i}\bclr{B^{\iota\prime}_{ i}\clr{\tau_n} + \widehat{B}^{\nu\prime}_i(t_n(u))}.
}
Let $P(n,m_1,m_2)$ denote the hypergeometric probability that two independently chosen uniform random subsets of $[n]$ of sizes $m_1$ and $m_2$ do not intersect. Let $\bf a$ and $\bf b$ be two $k$ dimensional vectors of non-negative integers. By using \eqref{forwardB}, we know that $W^{\iota}e^{\C{M}\tau_n}\sim \lfloor n^{17/24} \rfloor$ as $n\to\infty$, and for $j\in[k]$, $B^{\iota\prime}_{ j}(\tau_n)\sim \lfloor n^{17/24}\rfloor\zeta_j$ as $n\to\infty$. Similarly, using \eqref{backwardB}, for $j\in[k]$, $\widehat{B}^{\iota\prime}_{ j}(\tau_n)\sim \widehat{W}^{\nu}\hat{\zeta}_j$ as $n\to\infty$. Now, we have
\bes{
  \mel \IP\bcls{\clc{\C{B}^{\iota\prime}\clr{\tau_n}\cap\C{\widehat{B}}^{\nu\prime}(t_n(u))=\emptyset}\cap A_{n,3} \given A_n}\\
   &=\sum_{{\bf a,b}}\IP\bcls{\C{B}^{\iota\prime}\clr{\tau_n}\cap\C{\widehat{B}}^{\nu\prime}(t_n(u))=\emptyset \given A_n\cap A_{n,3} ; \widehat{B}^{\nu\prime}(t_n(u))={\bf a},B^{\iota\prime}(\tau_n)={\bf b}}\\
  &\kern4em\times\IP\bcls{\widehat{B}^{\nu\prime}(t_n(u))={\bf a}, B^{\iota\prime}(\tau_n)={\bf b} \given A_n\cap A_{n,3}}\IP\cls{A_{n,3}\given A_n}\\
      &=\sum_{{\bf a,b}} \prod_{i\in [k]} P \bclr{n_i , B^{\iota\prime}_{ i}(\tau_n) , \widehat{B}^{\nu\prime}_i(t_n(u))} \\
  &\kern4em\times \IP\cls{\widehat{B}^{\nu\prime}(t_n(u))={\bf a}, B^{\iota\prime}(\tau_n)={\bf b} \given A_n\cap A_{n,3}}\IP\cls{A_{n,3}\given A_n}\\
  &\leq \sum_{{\bf a,b}} \prod_{i\in[k]}\bbcls{\exp\bclc{- B^{\iota\prime}_{ i}\clr{\tau_n}\widehat{B}^{\nu\prime}_i(t_n(u))/n_i} + \frac{1}{n_i}\bclr{B^{\iota\prime}_{ i}(\tau_n)+  \widehat{B}^{\nu\prime}_i(t_n(u))}}\\
  &\kern4em\times \IP\bcls{\widehat{B}^{\nu\prime}(t_n(u))={\bf a}, B^{\iota\prime}(\tau_n)={\bf b} \given A_n\cap A_{n,3}}\IP\cls{A_{n,3}\given A_n}\\
   &\leq \IE\bbbclc{\exp\bbbclc{-\sum_{i\in[k]} B^{\iota\prime}_{ i}\clr{\tau_n}\widehat{B}^{\nu\prime}_i(t_n(u))/n_i} \given A_n} \\
   &\qquad+ \IE\bbclc{\Delta\bbclr{B^{\iota\prime}(\tau_n),  \widehat{B}^{\nu\prime}(t_n(u))} \given A_n\cap A_{n,3}}\\
     &\sim \IE\bbclc{\exp\bbclc{-\sum_{i\in[k]}\zeta_i\lfloor n^{17/24}\rfloor \widehat{W}^{\nu}\hat{\zeta}_ie^{\widehat{\C{M}}u}n^{7/24}/n_i} \given A_n}+\bigo(n^{-7/24})\\
    &\to\IE\bbbclc{\exp\bbclc{- \widehat{W}^{\nu}e^{\widehat{\C{M}}u} \sum_{i\in[k]}\zeta_i \hat{\zeta}_i p_i^{-1}}} \qquad \text{as $n\to\infty$,}
}
where 
\bes{
  \mel\Delta\bbclr{B^{\iota\prime}(\tau_n),  \widehat{B}^{\nu\prime}(t_n(u))}\\
  &=\prod_{i\in[k]}\bbcls{\exp\bclc{- B^{\iota\prime}_{ i}\clr{\tau_n}\widehat{B}^{\nu\prime}_i(t_n(u))/n_i}
  + \frac{1}{n_i}\bclr{B^{\iota\prime}_{ i}(\tau_n)+  \widehat{B}^{\nu\prime}_i(t_n(u))}} \\
  &\quad-\exp\bbclc{-\sum_{i\in[k]} B^{\iota\prime}_{ i}\clr{\tau_n}\widehat{B}^{\nu\prime}_i(t_n(u))/n_i}.
}
In the event $A_n\cap A_{n,3}$, we know that $|B^{\iota\prime}\clr{\tau_n}|=\floor{n^{17/24}}$ and $|\widehat{B}^{\nu\prime}(t_n(u))|\leq \bigo(n^{1/3})$. Thus the conditional expectation of $\Delta\bclr{B^{\iota\prime}(\tau_n),  \widehat{B}^{\nu\prime}(t_n(u))}$ is of order $\bigo(n^{-7/24})$. Similarly, for the lower bound side, we have
\bes{
    \mel\IP\bcls{v\in \cS^{\iota}_{n,\nu}\clr{\tau_n+t_n(u)}   \given A_n}\\
    &\geq \IP\bcls{\clc{\C{B}^{\iota}_n\clr{\tau_n}\cap\F{S}_{n,v}(t_n(u))=\emptyset}\cap  A_{n,3} \given A_n}\\
    &=\IP\bcls{\clc{\C{B}^{\iota}_n\clr{\tau_n}\cap\C{\widehat{B}}^{\nu\prime}(t_n(u))=\emptyset}\cap A_{n,3} \given A_n}\\
     &\geq \sum_{{\bf a,b}} \prod_{i\in[k]}\bbcls{\exp\bclc{- (B^{\iota\prime}_{ i}\clr{\tau_n}-n^{11/24})\widehat{B}^{\nu\prime}_i(t_n(u))/n_i} \\
    &\kern7em-\frac{1}{n_i}\bbclr{(B^{\iota\prime}_{ i}\clr{\tau_n}-n^{11/24})+  \widehat{B}^{\nu\prime}_i(t_n(u))}}\\
   &\kern4em\times \IP\cls{\widehat{B}^{\nu\prime}(t_n(u))={\bf a}, B^{\iota\prime}(\tau_n)={\bf b} \given A_n\cap A_{n,3}}\IP\cls{A_{n,3}\given A_n}\\
    &\geq \IE\bbclc{\exp\bbclc{-\sum_{i\in[k]} (B^{\iota\prime}_{ i}\clr{\tau_n}-n^{11/24})\widehat{B}^{\nu\prime}_i(t_n(u))/n_i} \given A_n} \\
    &\quad-\IP(\cls{A_{n,3}}^c\given A_n) - \IE\bbclc{\Delta\bbclr{(B^{\iota\prime}_{ i}\clr{\tau_n}-n^{11/24}),  \widehat{B}^{\nu}(t_n(u))} \given A_n\cap A_{n,3}}\\
   &\sim\IE\bbclc{\exp\bbclc{-\sum_{i\in[k]}\frac{(\zeta_i\lfloor n^{17/24}\rfloor -n^{11/24}) n^{7/24}}{n_i}\widehat{W}^{\nu}\hat{\zeta}_ie^{\widehat{\C{M}}u}}\given A_n}\\
   &\quad-\IP(\cls{A_{n,3}}^c\given A_n) -\bigo(n^{-7/24})\\
   & \to \IE\bbclc{\exp\bbclc{- \widehat{W}^{\nu}e^{\widehat{\C{M}}u} \sum_{i\in[k]}\zeta_i \hat{\zeta}_i p_i^{-1}}} \quad \text{as $n\to\infty$,}
}
where $\Delta\bclr{(B^{\iota\prime}_{ i}\clr{\tau_n}-n^{11/24}),  \widehat{B}^{\nu\prime}(t_n(u))} $ is defined in the same way as above. Then, when $n$ is large enough, we obtain that, for a randomly chosen individual $v$ of type~$\nu$,
\be{
    \IP\bcls{v\in \cS^{\iota}_{n,\nu}\clr{\tau_n+t_n(u)}   \given A_n}
    \sim \IE\bbclc{\exp\bbclr{-\widehat{W}^{\nu}e^{\widehat{\C{M}}u} \sum_{i\in[k]}\zeta_i \hat{\zeta}_i p_i^{-1}}}.
}
The argument for approximating the probability that both $v$ and $v'$ of type~$\nu$ belong to $\cS^{\iota}_{n,\nu}\clr{\tau_n+t_n(u)}$ runs in much the same way. The limiting random variable for the backward branching process $\widehat{B}_{v,v'}$ starting with two individuals can be expressed as $\widehat{W}^{\nu}_{v}+\widehat{W}^{\nu}_{v'}$, where the two are independent copies of $\widehat{W}^{\nu}$. The sizes $|\widehat{B}^{\nu}_{v}(t_n(u))|$ and $|\widehat{B}^{\nu}_{v'}(t_n(u))|$ are asymptotically $n^{7/24}\widehat{W}^{\nu}e^{\widehat{\C{M}}u}$. Thus the expectation number of individuals in the intersection set of $\C{\widehat{B}}^{\nu}_{v}(t_n(u))$ and $\C{\widehat{B}}^{\nu}_{v'}(t_n(u))$ is
\bes{
    \mel\IE\bbbclc{\sum_{i\in[k]}\bbabs{\C{\widehat{B}}^{\nu}_{v,i}(t_n(u))\cap\C{\widehat{B}}^{\nu}_{v',i}(t_n(u))}}\\
    &=\IE\bbbclc{\sum_{i\in[k]}n_i^{-1}\bbabs{\widehat{B}^{\nu}_{v,i}(t_n(u))}\bbabs{\widehat{B}^{\nu}_{v',i}(t_n(u))}}\\
    &\sim\sum_{i\in[k]}p_in^{-10/24}e^{2\widehat{\C{M}}u}\clr{\IE\clc{\hat\zeta_i\widehat{W}^{\nu}}}^2\rightarrow0
}
as $n\rightarrow\infty$, which implies that
\be{
    \IP\bbcls{\C{\widehat{B}}^{\nu}_{v}(t_n(u))\cap\C{\widehat{B}}^{\nu}_{v'}(t_n(u))=\emptyset}\rightarrow 1.
}
Conditioned on this event, the conditional probability that two randomly chosen $v$ and $v'$ are susceptible at time $\tau_n+t_n(u)$, is close to
\be{
  \IE\bclc{\exp\clc{-(\widehat{W}^{\nu}_{v}+\widehat{W}^{\nu}_{v'})e^{\widehat{\C{M}}u}}}.
}
Thus. we obtain that when $n$ is large enough,
\be{
    \IP\bcls{\{v,v'\}\subset \cS^{\iota}_{n,\nu}\clr{\tau_n+t_n(u)}   \given A_n}\sim\bbclr{\IE\bclc{\exp\clc{-\widehat{W}^{\nu}e^{\widehat{\C{M}}u}}}}^2,
}
which implies that $\Var\clr{n_{\nu}^{-1}S^{\iota}_{n,\nu}\clr{\tau_n+t_n(u)}   \given A_n}\to0$ as $n\to\infty$. 

Now we know that for any given positive constant $\varepsilon$ and any $u\in\IR$,
\be{
  \lim_{n\rightarrow\infty}\IP\bcls{\babs{n_{\nu}^{-1}S^{\iota}_{n,\nu}\clr{\tau_n+t_n(u)}-\bss^{\iota}_{\nu}(u)}>\varepsilon\given A_n}=0,
}
where $\bss^{\iota}_{\nu}(u)=\IE\clc{\exp\clc{-\widehat{W}^{\nu}e^{\widehat{\C{M}}u}m_*}}$ and 
\be{\label{m*}
  m_*=\sum_{i\in[k]}\zeta_i \hat{\zeta}_i p_i^{-1}.
}
Thus for any finite set $U\subset\IR$ we have
\bes{\label{finitedimension}
        \mel\lim_{n\rightarrow\infty}\IP\bbbcls{\sup_{u\in U}\bbabs{n_{\nu}^{-1}S^{\iota}_{n,\nu}\clr{\tau_n+t_n(u)}-\bss^{\iota}_{\nu}(u)}>\varepsilon\given A_n}\\
        & \leq \lim_{n\rightarrow\infty}\sum_{u\in U}\IP\bbbcls{\bbabs{n_{\nu}^{-1}S^{\iota}_{n,\nu}\clr{\tau_n+t_n(u)}-\bss^{\iota}_{\nu}(u)}>\varepsilon\given A_n}\\
        &=\sum_{u\in U}\lim_{n\rightarrow\infty}\IP\bbbcls{\bbabs{n_{\nu}^{-1}S^{\iota}_{n,\nu}\clr{\tau_n+t_n(u)}-\bss^{\iota}_{\nu}(u)}>\varepsilon\given A_n}
        =0.
}
Note that the function $\bss^{\iota}_{\nu}(\cdot)$ decreases smoothly from 1 to
the extinction probability for individuals of type~$\nu$, $\bss^{\iota}_{\nu} (\infty)$. So for any given $\varepsilon>0$ there is value $a_{\varepsilon}\in\IR$ such that $\bss^{\iota}_{\nu}(a_{\varepsilon})=1-\varepsilon/2$. Noting that $n_{\nu}^{-1}S^{\iota}_{n,\nu}(\cdot)$ is non-increasing and with the convention that  $S^{\iota}_{n,\nu}(t)=S^{\iota}_{n,\nu}(0)$ for each $n\in\IN$ and each $t\in(-\infty,0)$, we obtain
\bes{
       \mel\lim_{n\rightarrow\infty}\IP\bbbcls{\sup_{u\in (-\infty,a_{\varepsilon}]}  \bbabs{n_{\nu}^{-1}S^{\iota}_{n,\nu}\clr{\tau_n+t_n(u)} -\bss^{\iota}_{\nu}(u)}\leq\varepsilon\given A_n}\\
       &\geq  \lim_{n\rightarrow\infty}\IP\bbbcls{\bbabs{n_{\nu}^{-1}S^{\iota}_{n,\nu}\clr{\tau_n+t_n(a_\varepsilon)} -\bss^{\iota}_{\nu}(a_\varepsilon)}\leq\varepsilon/2\given A_n}=1.
}
If $\bss^{\iota}_{\nu}(\infty)=0$, then for any given $\varepsilon>0$, there exists $b_{\varepsilon}\in\IR$ such that $\bss^{\iota}_{\nu}(b_{\varepsilon})=\varepsilon/2$, and the discussion on $[b_{\varepsilon},\infty)$ will be similar due to the monotonicity of $\bss^{\iota}_{\nu}(\cdot)$ and $n_{\nu}^{-1}S^{\iota}_{n,\nu}(\cdot)$.

Now we turn to the convergence on $[a_{\varepsilon},T]$, where $T\in\mathbb R$ and $T>a_{\varepsilon}$. By \eqref{finitedimension} and \cite[Theorem 12.6]{billingsley1999convergence}, we know that the process $n_{\nu}^{-1}S^{\iota}_{n,\nu}(\tau_n+t_n(\cdot))$ convergences to $\bss^{\iota}_{\nu}(\cdot)$ in $D[a_{\varepsilon},T]$ in conditional probability, that is,
\be{
  \lim_{n\rightarrow\infty}\IP\bbcls{d^\circ_{a_{\varepsilon}T} \clr{n^{-1}S^{\iota}_{n,\nu}\clr{\tau_n+t_n(\cdot)},\bss^{\iota}_{\nu}(\cdot)}>\varepsilon\given A_n}=0.
}
Since $\bss^{\iota}_{\nu}(\cdot)$ is uniformly continuous on $[a_{\varepsilon},T]$, the Skorohod convergence of $n_{\nu}^{-1}S^{\iota}_{n,\nu}(\tau_n+t_n(\cdot))$ to $\bss^{\iota}_{\nu}(\cdot)$ on $[a_{\varepsilon},T]$ implies the uniform convergence (see \cite[Section 14 and 15]{billingsley1999convergence}), and the result of Theorem~\ref{maintheorem} follows.

\begin{remark} The course of the epidemic process mainly dependents on the backward branching process. Assume we start the process letting $\cV^0=\{\cV^0_1,\dots,\cV^0_k\}$ be the initial infected individuals, where $\cV^0_i$ are the individuals randomly chosen from population of type~$i$, letting the total number of initial infected individuals satisfy $|\cV^0|=\floor{n^{\alpha}}$, where $\alpha\in(\frac{2}{3},1)$ is fixed, and assuming $\lim_{n\rightarrow\infty}|\cV^0_i|/|\cV^0|=\zeta^*_i$ exists for all $i\in[k]$. Then it suffices to only construct the backward branching process until time $\frac{1-\alpha}{\C{\widehat M}}\log n$. A randomly chosen individual $v$ of type~$j$ is still susceptible at time $\frac{1-\alpha}{\C{\widehat M}}\log n$ if and only if the intersection set of the susceptibility set of $v$, $\F{S}_{n,v}$, and the initial infectives, $\cV^0$, is empty. This leads to the similar susceptible epidemic curve derived before through replacing $\zeta_i$ in \eqref{m*} by $\zeta^*_i$.
\end{remark}

\section{Proof of Theorem~\ref{IEmodels}}\label{proof3}
In Theorem~\ref{maintheorem}, we have characterised the susceptible curves for general epidemic models. Now we are going to prove  similar results for infected and recovery curves. The idea of the proof is inspired by \cite[Section 5]{pang2022functional}.

Recall that in a closed population of $n$ individuals, an individual $x$ of type~$i$ going through the susceptible-infectious-recovered (SIR) process has the time epochs $\pi_{n,x}$ and $\pi_{n,x}+Q_{i,x}$ associated with it, representing the time of infection and time of recovery, respectively.

The dynamics of the number of infected individuals of type~$\nu$, denoted by $I^{\iota}_{n, \nu}$, is given by
\bes{
        I^{\iota}_{n, \nu}(t)=\I\cls{\{Q^{\iota}_0>t\}}\delta_{\iota \nu}+\sum_{a=1}^{\substack{n_{\nu}-\delta_{\iota \nu}\\\quad-S^{\iota}_{n,\nu}(t)}}\I\cls{\pi_{n,a}+Q_{\nu,a}>t},\qquad t\geq 0,
}
where $\delta_{ij}$ is the usual Kronecker delta. The first term determines whether the initial infected individual remains infected at time $t$ if its type is $\nu$, and the second term counts the number of individuals of type~$\nu$ that were infected between time 0 and time $t$ and remain infectious at time $t$. Recall that $R^{\iota}_{n, \nu}(t)$ counts the number
of recovered individuals of type~$\nu$ at time $t$, and can be represented as
\be{
        R^{\iota}_{n, \nu}(t)=\I\cls{\{Q^{\iota}_0\leq t\}}\delta_{\iota \nu}+\sum_{a=1}^{\substack{n_{\nu}-\delta_{\iota \nu}\\\quad-S^{\iota}_{n,\nu}(t)}}\I\cls{\pi_{n,a}+Q_{\nu,a}\leq t},\quad  t\geq 0.
}
For any given $T\in\IR$ and $u\leq T$, define the $\sigma$-algebra
\be{
  \C{F}_{n,u}^{*}=A_n\bigvee\sigma(S^{\iota}_{n}(t\wedge(\tau_n+t_n(u))):t\geq0).
}
and
\ba{
        \breve{I}^{\iota}_{n, \nu}(t)&=\sum_{a=1}^{\substack{n_{\nu}-\delta_{\iota \nu}\\\quad-S^{\iota}_{n,\nu}(t)}}\I\cls{\{\pi_{n,a}+Q_{\nu,a}>t\}},\quad t\geq 0,\\
        \bar{\breve{I}}^{\iota}_{n, \nu}(u)&=\IE\bclc{n_j^{-1}\breve{I}^{\iota}_{n, \nu}\clr{\tau_n +t_n(u)}\given\C{F}_{n,u}^{*}}.
}
We know that
\bes{
        \bar{\breve{I}}^{\iota}_{n, \nu}(u)
        &=  \frac{1}{n_{\nu}}  \sum_{a=1}^{\substack{n_{\nu}-\delta_{\iota \nu}\\-S^{\iota}_{n,\nu}(\tau_n+t_n(u))}}F_{Q_{\nu}} (\tau_n +t_n(u)-\pi_{n,a})\\
        &=\int_{0}^{\tau_n +t_n(u)} F^{c}_{Q_{\nu}}(\tau_n +t_n(u)-\tau)d(1-n_{\nu}^{-1}\delta_{\iota \nu}-n_{\nu}^{-1}S^{\iota}_{n,\nu}(\tau)).
}
With the convention that $I^{\iota}_{n, l}(t)=I^{\iota}_{n, l}(0)$, $S^{\iota}_{n,l}(t)=S^{\iota}_{n,l}(0)$, $F_{Q_i}(t)=F_{Q_i}(0)$, $F^c_{Q_{i}}(t)=F^c_{Q_{i}}(0)$, for all $t\in(-\infty,0)$, all $i,l\in[k]$ and all $n\in\mathbb{N}$, we obtain by integration by parts that
\bes{
        \mel\bar{\breve{I}}^{\iota}_{n, \nu}(u)\\
        &=F^c_{Q_{\nu}}(0)(1-n_{\nu}^{-1}\delta_{\iota \nu}-n_{\nu}^{-1}S^{\iota}_{n,\nu}(\tau_n+t_n(u))\\
        &\quad-\int_{0}^{\tau_n +t_n(u)}(1-n_j^{-1}\delta_{\iota \nu}-n_j^{-1}S_n(v))dF^c_{Q_{\nu}}(\tau_n +t_n(u)-v)\\
        &=(1-n_{\nu}^{-1}-n_{\nu}^{-1}S^{\iota}_{n,\nu}(\tau_n+t_n(u))\\
        &\quad-\int_{-\tau_n-\frac{7}{24\C{M}_n}\log n}^{u}\clr{1-n_{\nu}^{-1}\delta_{\iota \nu}-n_{\nu}^{-1}S^{\iota}_{n,\nu}\clr{\tau_n+t_n(w)}}dF^c_{Q_{\nu}}(u-w)\\
         &=(1-n_j^{-1}-n_{\nu}^{-1}S^{\iota}_{n,\nu}(\tau_n+t_n(u))\\
        &\quad-\int_{-\infty}^{u}\clr{1-n_{\nu}^{-1}\delta_{\iota \nu}-n_{\nu}^{-1}S^{\iota}_{n,\nu}\clr{\tau_n+t_n(w)}}dF^c_{Q_{\nu}}(u-w).
}
Here, $dF^c_{Q_{\nu}}(\tau_n +t_n(u)-v)$ is the differential of the map $v\rightarrow F^c_{Q_{\nu}}(\tau_n +t_n(u)-v)$. Recall the result of Theorem~\ref{maintheorem}; that is, in $D(-\infty,T]$,
\be{
  n_{\nu}^{-1}S^{\iota}_{n,\nu}\clr{\tau_n+t_n(\cdot)}\to_{uc}\bss^{\iota}_{\nu}(\cdot)
}
in (conditional) probability as $n\rightarrow\infty$, where $\to_{uc}$ means uniformly convergence under metric $\mathrm{d}^\circ_{-\infty T}$. By the continuous mapping theorem applied to the map $x\in D(-\infty,T]\rightarrow \int_{-\infty}^{\cdot}x(\tau)dF^c_{Q_{\nu}}(\cdot-\tau)\in D(-\infty,T]$, we obtain
\be{
    \bar{\breve{I}}^{\iota}_{n, \nu}(\cdot)
    \to_{uc} 1-\bss^{\iota}_{\nu}(\cdot)-\int_{-\infty}^{\cdot}(1-\bss^{\iota}_{\nu}(v))dF^c_{Q_{\nu}}\clr{\cdot-v},
}
in (conditional) probability as $n\rightarrow\infty$. Now, for a given $u<T$, we turn to the (conditional) variance of $n_{\nu}^{-1}\breve{I}^{\iota}_{n, \nu}(\tau_n+t_n(u))$. 
Let
\bes{
    \Delta I^{\iota}_{n, \nu}(u)&=n_{\nu}^{-1}\breve{I}^{\iota}_{n, \nu}(\tau_n+t_n(u))-\bar{\breve{I}}^{\iota}_{n, \nu}(u)
    =\frac{1}{n_{\nu}}\sum_{a=1}^{\substack{n_{\nu}-\delta_{\iota \nu}\\-S^{\iota}_{n,\nu}(\tau_n+t_n(u))}}\chi_{n,a}(u),
}
where
\be{
  \chi_{n,a}(u)=\I\cls{\{\pi_{n,a}+Q_{\nu,a}>\tau_n+t_n(u)\}}-F^c_{Q_{\nu}}(\tau_n+t_n(u)-\pi_{n,a}).
}
It is easy to check that, for all $a$, 
\bes{
        \IE\clc{\chi_{n,a}(u)|\C{F}^*_{n,u}}=0, 
}
and, for all $a\neq b$,
\be{
        \IE\clc{\chi_{n,a}(u)\chi_{n,b}(u}|\C{F}^*_{n,u})=0.
}
Now,
\bes{
        \mel\IE\clc{(\Delta I^{\iota}_{n, \nu}(u))^2|\C{F}^*_{n,u}}\\
        &=\IE\bbbbclc{\bbbbclr{n_j^{-1}\sum_{a=1}^{\substack{n_{\nu}-\delta_{\iota \nu}\\-S^{\iota}_{n,\nu}(\tau_n+t_n(u))}}\chi_{n,a}}^2\given \C{F}^*_{n,u}}\\
        &=n_{\nu}^{-2}\IE\bbbbclc{\sum_{a=1}^{\substack{n_{\nu}-\delta_{\iota \nu}\\-S^{\iota}_{n,\nu}(\tau_n+t_n(u))}}\chi_{n,a}^2\given \C{F}^*_{n,u}}\\
        &=n_{\nu}^{-2}\sum_{a=1}^{\substack{n_{\nu}-\delta_{\iota \nu}\\-S^{\iota}_{n,\nu}(\tau_n+t_n(u))}}F_{Q_{\nu}}(\tau_n+t_n(u)-\pi_{n,a}) 
        F^c_{Q_{\nu}}(\tau_n+t_n(u)-\pi_{n,a})\\
        &=\frac{1}{n_{\nu}}\int_{0}^{\tau_n+t_n(u)}F_{Q_{\nu}}(\tau_n+t_n(u)-\tau)\\
        &\kern9em\times F^c_{Q_{\nu}}(\tau_n+t_n(u)-\tau)d(1-n_{\nu}^{-1}\delta_{\iota \nu}-n_{\nu}^{-1}S^{\iota}_{n,\nu}(\tau))\\
        & \leq \frac{1}{n_{\nu}}(1-n_{\nu}^{-1}\delta_{\iota \nu}-n_{\nu}^{-1}S^{\iota}_{n,\nu}(\tau_n+t_n(u)))= \bigo(n^{-1}).
}
Thus,
\be{
  \IE\bclc{(n_{\nu}^{-1}\breve{I}^{\iota}_{n, \nu}(\tau_n+t_n(u))-\bar{\breve{I}}^{\iota}_{n, \nu}(u))^2\given A_n}
      = \bigo(n^{-1})
}
almost surely, and for any $\varepsilon>0$, $\nu\in[k]$ and any $u\leq T$,
\be{
  \IP\bcls{|n_{\nu}^{-1}\breve{I}^{\iota}_{n, \nu}(\tau_n+t_n(u))-\bar{\breve{I}}^{\iota}_{n, \nu}(u)|>\varepsilon\given A_n}
       = \bigo(n^{-1}\varepsilon^{-2})\rightarrow0
}
as $n\rightarrow\infty$. Thus, we obtain that for any $\varepsilon>0$ and $u\leq T$,
\be{
  \lim_{n\rightarrow\infty}
  \IP\bbbbcls{\bbbabs{n_{\nu}^{-1}I^{\iota}_{n, \nu}(\tau_n+t_n(u))-\bbbclr{1-\bss^{\iota}_{\nu}(u)
  -\int_{-\infty}^u(1-\bss^{\iota}_{\nu}(v))dF^c_{Q_{\nu}}(u-v)}} >\varepsilon\given A_n}
   =0.
}
Similarly, we can derive that for any $\varepsilon>0$ and $u\leq T$,
\be{
  \lim_{n\rightarrow\infty}
  \IP\bbbbcls{\bbbabs{n_{\nu}^{-1}R^{\iota}_{n, \nu}(\tau_n+t_n(u))
    -\bbbclr{-\int_{-\infty}^u(1-\bss^{\iota}_{\nu}(v))dF_{Q_{\nu}}(u-v)}}>\varepsilon\given A_n}
  =0.
}
The extension to the supremum over $u$ for $n_{\nu}^{-1}R^{\iota}_{n, \nu}$ is much the same as the end of the proof in Section~\ref{proof2.3} by noting that $n_j^{-1}R^{\iota}_{n, \nu}(\cdot)$ is non-decreasing and $-\int_{-\infty}^{\cdot}(1-\bss^{\iota}_{\nu}(v))dF_{Q_{\nu}}(\cdot-v)$ increases smoothly from 0 to
$1-\bss^{\iota}_{\nu}(\infty)$. The analogous result for $n_{\nu}^{-1}I^{\iota}_{n, \nu}$ follows due to the conservation law
\be{
  n_{\nu}^{-1}S^{\iota}_{n,\nu}+n_{\nu}^{-1}I^{\iota}_{n, \nu}+n_{\nu}^{-1}R^{\iota}_{n, \nu}=1.
}
Now we turn to the limiting susceptible population ratio curve 
\bes{
    \bss^{\iota}_{\nu}(u)=\IE\bclc{\exp\bclr{-\widehat{W}^{\nu}e^{\widehat{\C{M}}u}m_*}}.
}
Recall that $ \psi^{\nu}(\tau)=\IE\clc{e^{-\tau\widehat{W}^{\nu}}}$ and that
\ben{\label{psi}
    \psi^{\nu}(s)= \exp\bbbclr{-\sum_{i=1}^k\widehat\bR_{0,\nu,i}\int_0^{\infty}(1- \psi_{i}(se^{-\widehat{\C{M}}u}))G_{i,\nu}(du)}.
}
Now,
\besn{\label{ratioS}
    \bss^{\iota}_{\nu}(t)&= \psi^{\nu}(e^{\widehat{\C{M}}t}m_*)\\
     &= \exp\bbbclr{-\sum_{i=1}^k\widehat\bR_{0,\nu,j}\int_0^{\infty}(1- \psi_{i}(m_*e^{\widehat{\C{M}}(t-u)}))G_{i,\nu}(du)} \\
     &= \exp\bbbclr{-\sum_{i=1}^k\widehat\bR_{0,\nu,j}\int_0^{\infty}(1-\bss^{\iota}_{i}(t-u))G_{i,\nu}(du)}.
}
Note that the equation \eqref{psi} has many solutions, since, if $ \psi(t)$ is a solution, so is  $ \psi_h(t)=:\psi(ht)$, for any fixed $h>0$; this fact has already been observed by \cite[p.\,20]{barbour2013approximating}. The condition  $ \psi^{\nu}(0)=1$, or equivalently, $\bss^{\iota}_{\nu}(-\infty)=1$, is satisfied by all  $ \psi_h$. The relevant choice of unique solution to \eqref{psi} is determined by matching $\IE\widehat{W}^{\nu}$ with $-\frac{d}{dt}\psi^{\nu}(t)|_{t=0}$ by noting that $\bss^{\iota}_{\nu}(t)= \psi^{\nu}(e^{\widehat{\C{M}}t}m_*)$ and
\ben{\label{psiEW}
     \lim_{t\rightarrow-\infty}e^{-\widehat{\C{M}}t}\frac{d}{dt}\bss^{\iota}_{\nu}(t)=m_*\widehat{\C{M}}\lim_{t\rightarrow-\infty}\frac{d}{dt} \psi^{\nu}(t)|_{t=e^{\widehat{\C{M}}t}m_*}=-m_*\widehat{\C{M}}\IE\widehat{W}^{\nu}.
}
This concludes the proof of Theorem~\ref{IEmodels}.

\section{Proof of Theorem~\ref{contactempiricaltheorem}}\label{proof2.7}
The distinction between the epidemic processes in Section~\ref{general} and~\ref{contempirical} whether the contact process does or does not depend on the population size $n$.

Recall the definition of  $\xi_{n,Q_i,i,j}$ in \eqref{superposition1}. When $Q_{i}$ given, $\xi_{n,Q_i,i,j}$ becomes a vector of deterministically truncated point processes, which can be approximated by Poisson point processes in terms of total variation distance; see for example \cite[Chapter 1]{reiss2012course}. Specifically, given $Q_{i}=T$,
\be{ 
    \xi_{n,Q_i,i}(\cdot)|_{Q_{i}=T} =\bbbclr{\sum_{l=1}^{n_j}\delta_{X_{l,n,i,1}}(\cdot\cap[0,T]),\dots, \sum_{l=1}^{n_j}\delta_{X_{l,n,i,k}}(\cdot\cap[0,T])}.
}
Here, $\sum_{l=1}^{n_j}\delta_{X_{n,i,j,l}}(\cdot\cap[0,T])$ is a truncated empirical process, which has the same distribution as the binomial point process
\be{ 
  N_{n,T,i,j}=\sum_{l=1}^{\alpha_{n,T}}\delta_{Y_{l,n,T,i,j}},
}
where $\alpha_{n,T}$, $Y_{1,n,T,i,j}$, and $Y_{2,n,T,i,j}\dots$ are independent and with distributions $\law(Y_{l,n,T,i,j})=F_{X_{n,i,j}}|_T$, $l\in\IN$, and $\law (\alpha_{n.T})=\Bi\bbclr{n_j,F_{X_{n,i,j}}([0,T])}$, where
\be{
  F_{X_{n,i,j}}\big|_T=\frac{F_{X_{n,i,j}}(\cdot\cap[0,T])}{F_{X_{n,i,j}}([0,T])};
}
see \cite[Theorem 1.4.1]{reiss2012course}.
The binomial process $N_{n,T,i,j}$ can be approximated by a Poisson point process, denoted by $N^{*}_{n,T,i,j}$, with the same intensity measure, namely 
\be{
    \Lambda_{n,T,i,j}(A)=\IE\bbbclc{\sum_{l=1}^{n_j}\delta_{X_{n,i,j,l}}(A\cap[0,T])}
    =n_j\IP[X_{n,i,j}\in (A\cap[0,T])]
}
for all $A\in \F{B}(\IR_+)$. By Assumption~\ref{Lambda}, let 
\be{
  \Lambda_{T,i,j}(dt)=\lim_{n\rightarrow\infty}\Lambda_{n,T,i,j}(dt),
}
and let $N^{*}_{T,i,j}$ be a Poisson point process with intensity measure $\Lambda_{T,i,j}$ and $N^{*}_{T,i}=\bclr{N^{*}_{T,i,1},\dots,N^{*}_{T,i,k}}$. Now let $\C{B}^{\iota\prime}$ be a C-M-J branching process starting with an individual of type~$\iota$ and individuals of type~$i$ give births according to the vector of point processes $\xi_{Q_i,i}$. Under Assumption~\ref{empirical}(ii), the epidemic process $\widehat{\C{B}}^{\nu}_n$ can be coupled to the branching process $\widehat{\C{B}}^{\nu\prime}$ successfully before the random time $\tau_n$ with high probability. Denote this event by $A_{n,4}$; hence, $\IP(A_{n,4}^c)\to0$ as $n\to\infty$. For the susceptibility process, recall that the total variation distance between two Poisson processes is bounded by the total variation distance between their mean measures; see for example \cite[Theorem 2.4]{barbour1992stein}. Note that all the Poisson processes in the backward branching process are independent. Thus the branching process $\widehat{\C{B}}^{\nu\prime}_n$ can be coupled to the branching process $\widehat{\C{B}}^{\nu\prime}$ due to Assumption~\ref{empirical}, as long as the number of offspring in these two branching processes is at most of order $\bigo(n^{1/3})$. Specifically, using Markov's inequality, we can estimate the probability that the total number of individuals born in the backward branching process $\C{\widehat{B}}^{\nu'}$ up to time $t_n(u_n)=7/(24\C{\widehat{M}})\log n+u_n$ (where $u_n=1/(48\C{\widehat{M}})\log n$) exceeds $n^{1/3}$ as
\be{
    \IP\bbcls{|\widehat{B}^{\nu\prime}(t_n(u_n))| >n^{1/3}}\leq e_{\widehat{W}} n^{-1/48},
}
where $e_{\widehat{W}}=\sup_{t\geq0}\IE\bclc{|\widehat{B}^{\nu\prime}(t)|e^{-\widehat{\C{M}}t_n(u_n)}}<\infty$ due to Assumption~\ref{empirical}. Thus, the (backward) branching process $\widehat{\C{B}}^{\nu'}_n$ can be replaced by the branching process $\widehat{\C{B}}^{\nu'}$ until time $t_n(u)$, except on an event, denoted by $A_{n,5}$, of probability at most
\be{
  \bigo(n^{-1/48})+\bigo(n^{-1/2})+\bigo\bbbclr{n^{1/3}\max_{i\in[k]}\sum_{j\in[k]}\dtv\bbclr{G_{n,i,j}, G_{i,j}}},
}
where $\max_{i\in[k]}\sum_{j\in[k]}\dtv\bclr{G_{n,i,j},  G_{i,j}}=\lito(n^{-1/3})$ due to Assumption~\ref{empirical}(iii).  We consider the population size to be large enough such that
\be{
  u\in\bcls{-7/(24\widehat{\C{M}})\log n,1/(48\widehat{\C{M}})\log n}.
} 

Let $A_{n,6}=A_{n,3}\cap A_{n,4}\cap A_{n,3}$, where $A_{n,3}$ is defined in the same way as the one in Section~\ref{sectionsusbranching}, and let $t_n(u)=7/(24\widehat{\C{M}})\log n+u$ for any given $u\in\IR$. The statement for the susceptible ratio of Theorem~\ref{contactempiricaltheorem} can be verified through replacing $A_{n,3}$ by $A_{n,6}$ in Section~\ref{proof2.3} and the proof for the infected and recovery ratios are similar as in Section~\ref{proof3}.

\section{Proof of Theorem~\ref{dynamicmodel}}\label{proof2}

Recall the multi-type Markovian SIR process on a dynamic random graph, which we shall call ``Model 1''. In Model 1, the contact process between two individuals of type~$i$ and $j$ is realised through a homogeneous Poisson point process $\varsigma_{n,i,j}$, modulated by a 2-state Markov chain (states `0' and `1'). The intensity parameter of $\varsigma_{n,i,j}$ is denoted by $\beta_{n,i,j}$. The 2-state Markov chain transitions from state 0 to 1 at rate  $\lambda_{n,i,j}=\lambda_{n,j,i}$, and transitions from state~1 to~0 at rate $\mu_{n,i,j}=\mu_{n,j,i}$. State 1 indicates that the edge is on, meaning that the disease can spread through that edge, while state 0 indicates that it is off, and the disease cannot spread through that edge. The stationary distribution of the edge state is the Bernoulli distribution $\Be\bclr{\lambda_{n,i,j}/\clr{\lambda_{n,i,j}+\mu_{n,i,j}}}$. We assume that all Markov chains start in equilibrium and are independent of each other. We also assume that the Poisson point processes are independent and identically distributed copies of $\varsigma_n$ and are independent of the 2-state Markov chains.

If $x$ is infected and $y$ is susceptible, then an infection occurs successfully at the first time a point in the Poisson process $\varsigma_{n,x,i,j}$ occurs while the Markov chain is in state 1 before $x$ recovers. An infected individual of type~$i$ recovers at a rate of $\gamma_i $. The infectious period of an infected individual obeys an exponential distribution $\Exp(\gamma_i )$. Contacts between any two infected individuals do not affect the epidemic process.

A Poisson point process modulated by the above Markov chain is called an ``interrupted Poisson process'', which is also referred to as a ``doubly stochastic Poisson process'' (\cite[pp. 328]{resnick1992adventures}), or a special ``Markov-modulated Poisson process'' (\cite{fischer1993markov}).

Finally, we consider the following two modifications of Model 1, which will turn out to be equivalent when only studying the SIR epidemic in Model 1. 

\medskip

\noindent\textbf{Model 2.} 
We use the same setup as in Model 1, but with the key difference that when an infection occurs between two individuals of type~$i$ and $j$, the edge process between them is reset to its equilibrium state given by a Bernoulli distribution with parameter $\lambda_{n,i,j}/(\lambda_{n,i,j}+\mu_{n,i,j})$.

\medskip

\noindent\textbf{Model 3.} 
Start with one infectious individual of type~$\iota $, call it $x_0$, in a population of $n$ individuals separated into $k$ communities. Equip $x_0$ with the point process $\xi^{\iota}_{n,x_0}=\clr{\xi^{\iota}_{n,x_0, 1},\dots,\xi^{\iota}_{n,x_0,k}}$, which has the same distribution with $\xi^{\iota}_n=\clr{\xi^{\iota}_{n, 1},\dots,\xi^{\iota}_{n,k}}$ defined in \eqref{vary1}. Whenever there is a point on $\xi^{\iota}_{n,\iota,i}$, pick a random individual without replacement from the population of type 
$i$ and infect it (unless it already is infected, in which case do nothing). Continue this process until the initially infected individual of type~$\iota $ recovers, and then repeat the process with any newly infected individuals until there are no more infected individuals.

\medskip

\begin{lemma}
The numbers of infected individuals over time in Model 1, 2 and 3 follow the same distribution.
\end{lemma}
\begin{proof}
The law of the random variable $X^{\mathrm{e}}_{n,i,j}$ is the distribution of the time of infection when considering one infectious individual of type~$i$ (call it $x$) and another susceptible individual of type~$j$ (call it $y$): the infection only occurs when the edge between $x$ and $y$ is ``on'' (the Markov chain is in state 1). At the point when $x$ gets infected, the status of the edge between $x$ and $y$ is unknown, hence in equilibrium $\Be\bclr{ \lambda_{n,i,j}/(\lambda_{n,i,j}+\mu_{n,i,j})}$. Once $y$ is infected, the edge between $x$ and $y$ is known to be ``on'', but for the infectious process of $y$ we can assume it restarts in equilibrium because $y$ already is infected, so the edge between $x$ and $y$ is irrelevant for the rest of the epidemic.
\end{proof}

Thanks to this lemma, we can analyze the epidemic curves of Model 1 by examining those of Model 3 and applying the results of Theorem~\ref{contactempiricaltheorem} to it. In Model 3, the dynamic property of the random graph is encapsulated in the contact point process $\xi_{n,\cdot}$. The description of Model 3 is very similar to that of the general (Crump-Mode-Jagers) branching process, in which contacts are treated as births and an infected individual's recovery represents its death. Therefore, we can interpret Model 3 as the general epidemic model introduced in Section~\ref{general} and~\ref{contempirical}. The vector of point processes $\xi_{\cdot}$ of Model 3 corresponds to the vector of point processes $\xi_{\cdot}$ of the general epidemic model. Model 3 will serve as the foundation for our subsequent analyses.

In the following section, we provide a detailed analysis of Model 3 and establish the validity of Theorem~\ref{dynamicmodel}. Notably, the result of Theorem~\ref{dynamicmodel} for Model 3 can be directly derived from Theorem~\ref{contactempiricaltheorem}. Consequently, our focus will be on verifying that Model 3 satisfies the assumptions stipulated in Theorems~\ref{contactempiricaltheorem}, which will suffice to establish the validity of Theorem~\ref{dynamicmodel}.

\subsection{Preliminary results}\label{preliminaryresult}
Recall that, the Poisson point process $\varsigma_{n,i,j}$ (between $i$-individual to $j$-individual) modulated by a 2-state Markov chain is an interrupted Poisson process. The density function of the interarrival time $X^{\mathrm{i}}_{i,j}$ has  been calculated by \cite{kuczura1973interrupted} as
\be{
  f^{\mathrm{i}}_{i,j}(t)=p_{n,i,j}r_{1,n,i,j}e^{-r_{1,n,i,j}t}+(1-p_{n,i,j})r_{2,n,i,j}e^{-r_{2,n,i,j}t},
}
where
\ba{
    r_{1,n,i,j} &= \frac{1}{2}\bbbclc{\beta_{n,i,j}+\lambda_{n,i,j}+\mu_{n,i,j}+\sqrt{(\beta_{n,i,j}+\lambda_{n,i,j}+\mu_{n,i,j})^2-4\beta_{n,i,j}\lambda_{n,i,j}}},\\
    r_{2,n,i,j} &= \frac{1}{2}\bbbclc{\beta_{n,i,j}+\lambda_{n,i,j}+\mu_{n,i,j}-\sqrt{(\beta_{n,i,j}+\lambda_{n,i,j}+\mu_{n,i,j})^2-4\beta_{n,i,j}\lambda_{n,i,j}}},\\
    p_{n,i,j} &= \frac{\beta_{n,i,j}-r_{2,n,i,j}}{r_{1,n,i,j}-r_{2,n,i,j}}.
}
Moreover, 
\be{
   \IE X^{\mathrm{i}}_{i,j}=\frac{1}{\beta_{n,i,j}}\frac{\lambda_{n,i,j}+\mu_{n,i,j}}{\lambda_{n,i,j}}=\frac{p_{n,i,j}}{r_{1,n,i,j}}+\frac{1-p_{n,i,j}}{r_{2,n,i,j}}.
}
It is easy to check that
\bg{
    r_{1,n,i,j}r_{2,n,i,j}=\beta_{n,i,j}\lambda_{n,i,j},\\
    p_{n,i,j}r_{2,n,i,j}+(1-p_{n,i,j})r_{1,n,i,j}=\lambda_{n,i,j}+\mu_{n,i,j}.
}
The equilibrium excess lifetime distribution has probability density function
\besn{\label{Xe}
   f_{X^{\mathrm{e}}_{n,i,j}}(t) 
   &=\frac{\int_t^{\infty}f^{\mathrm{i}}_{i,j}(\tau)d\tau}{\IE X^{\mathrm{i}}_{i,j}}= \frac{p_{n,i,j}e^{-r_{1,n,i,j}t}+(1-p_{n,i,j})e^{-r_{2,n,i,j}t}}{\IE X^{\mathrm{i}}_{i,j}}\\
   &= \frac{p_{n,i,j}r_{1,n,i,j}r_{2,n,i,j}e^{-r_{1,n,i,j}t}+(1-p_{n,i,j})r_{1,n,i,j}r_{2,n,i,j}e^{-r_{2,n,i,j}t}}{p_{n,i,j}r_{2,n,i,j}+(1-p_{n,i,j})r_{1,n,i,j}}.
}
For any bounded interval $[s,t]\subset [0,+\infty)$, we have
\bes{
    \mel\IP\cls{X^{\mathrm{e}}_{n,i,j}\in[s,t]}
    =\int_s^tf_{X^{\mathrm{e}}_{n,i,j}}(\tau)d\tau\\
    &=\frac{p_{n,i,j}r_{2,n,i,j}(e^{-r_{1,n,i,j}s}-e^{-r_{1,n,i,j}t})+(1-p_{n,i,j})r_{1,n,i,j}(e^{-r_{2,n,i,j}s}-e^{-r_{2,n,i,j}t})}{p_{n,i,j}r_{2,n,i,j}+(1-p_{n,i,j})r_{1,n,i,j}}.
}
Recall from \eqref{superposition1} that
\bg{
    \xi_{n,i,j}=\sum_{l=1}^{n_j}\delta_{X^{\mathrm{e}}_{n,i,j,l}},\qquad
    \xi_{n,Q_i,i,j}([0,t]) =\xi_{n,i,j}([0,t\wedge Q_{i}]),\\
    \xi_{n,Q_i,i}([0,t]) =\bbclr{\xi_{n,Q_i,i,1}([0,t]),\dots,\xi_{n,Q_i,i,k}([0,t])}.
}
Hence,
\bes{
    \mel\IE\clc{\xi_{n,Q_i,i,j}([0,t])}=\IE\clc{\xi_{n,i,j}([0,t\wedge Q_{i}])}\\
    &=\sum_{k=0}^{n_j}\bbbclc{k\int_0^t\IP\cls{\xi_{n,i,j}[0,s]=k}f_{Q_{i}}(s)ds+k\IP\cls{\xi_{n,i,j}[0,t]=k}\IP\cls{Q_{i}\geq t}}\\
    &=\int_0^t\IE\clc{\xi_{n,i,j}[0,s]}f_{Q_{i}}(s)ds+\IE\clc{\xi_{n,i,j}[0,t]}\IP\cls{Q_{i}\geq t}\\
    &=\int_0^tn_j\cdot\IP\cls{X^{\mathrm{e}}_{n,i,j}\leq s}\gamma_i  e^{-\gamma_i  s}ds+n_j\cdot\IP\cls{X^{\mathrm{e}}_{n,i,j}\leq t}e^{-\gamma_i  t},
}
and
\besn{\label{expxi}
    \mel\IE\clc{\xi_{n,Q_i,i,j}(\IR_+)}=\int_0^{\infty}n_j\cdot\IP\cls{X^{\mathrm{e}}_{n,i,j}\leq s}\gamma_i  e^{-\gamma_i  s}ds\\
    &=n_j\bbbclr{1- \frac{p_{n,i,j}r_{2,n,i,j}\gamma_i (r_{2,n,i,j}+\gamma_i )+(1-p_{n,i,j})r_{1,n,i,j}\gamma_{i,j}(r_{1,n,i,j}+\gamma_i )}{(p_{n,i,j}r_{2,n,i,j}+(1-p_{n,i,j})r_{1,n,i,j})(r_{1,n,i,j}+\gamma_i )(r_{2,n,i,j}+\gamma_i )}}\\
    &=n_j\frac{\beta_{n,i,j}\lambda_{n,i,j}(\lambda_{n,i,j}+\mu_{n,i,j}+\gamma_i )}{(\lambda_{n,i,j}+\mu_{n,i,j})((\gamma_{i,j})^2+\gamma_i (\beta_{n,i,j}+\lambda_{n,i,j}+\mu_{n,i,j})+\beta_{n,i,j}\lambda_{n,i,j})}.
}
The relative intensity measure $G_{n,i,j}$ of $\xi_{n,Q_i,i,j}$ is
\bes{
    \mel G_{n,i,j}(dt) 
    = \IE\clc{\xi_{n,Q_i,i,j}(dt)}/\IE\clc{\xi_{n,Q_i,i,j}(\IR_+)}\\
    &=n_j\cdot f_{X^{\mathrm{e}}_{n,i,j}}(t) e^{-\gamma_i  t}dt\bigg/\bbbclr{\int_0^{\infty}n_j\cdot\IP\cls{X^{\mathrm{e}}_{n,i,j}\leq s}\gamma_i  e^{-\gamma_i  s}ds}\\
    &= \frac{n_j}{\bR_{0,n,i,j}}\cdot \frac{p_{n,i,j}r_{2,n,i,j}}{p_{n,i,j}r_{2,n,i,j}+(1-p_{n,i,j})r_{1,n,i,j}}r_{1,n,i,j}e^{-(r_{1,n,i,j}+\gamma_i ) t}dt\\
    &\quad+\frac{n_j}{\bR_{0,n,i,j}}\cdot \frac{(1-p_{n,i,j})r_{1,n,i,j}}{p_{n,i,j}r_{2,n,i,j}+(1-p_{n,i,j})r_{1,n,i,j}}r_{2,n,i,j}e^{-(r_{2,n,i,j}+\gamma_i ) t}dt.
}
Let $\Xi_{n,i,j}$ denote a Poisson point process with intensity measure $ p_i\bR_{0,n,i,j}G_{n,i,j}/p_j $ and $\Xi_{i,j}$ a Poisson point process with intensity measure $ p_i\bR_{0,i,j}G_{i,j}/p_j $. Then the total variation distance between $\law (\Xi_{n,i,j})$ and $\law (\Xi_{i,j})$, is bounded by the total variation distance between their mean measures; that is
\bes{
        \mel\dtv(\law (\Xi_{n,i,j}),\law (\Xi_{i,j}))\\
        &\leq \frac{p_i}{p_j}\| \bR_{0,n,i,j}G_{n,i,j},\bR_{0,i,j}G_{i,j} \|_{\mathrm{TV}}\\
        &\leq  \frac{p_i}{p_j}\sup_{f }\int_{0}^{\infty}|f(t)|\cdot|\bR_{0,n,i,j}G_{n,i,j}-\bR_{0,i,j}G_{i,j}|(dt)\\
        &\leq  \frac{p_i}{p_j}\int_{0}^{\infty}|\bR_{0,n,i,j}G_{n,i,j}-\bR_{0,i,j}G_{i,j}|(dt),
}
where the supremum is taken over f ranging over the set of all measurable functions from $\IR_+$ to $[-1,1]$; see for example \cite[Theorem 2.4]{barbour1992stein}. Denote the densities of $ \bR_{0,n,i,j}G_{n,i,j} $ and $\bR_{0,i,j}G_{i,j}$ by $g_{n,i,j}(t)$ and $g_{i,j}$ at $t\in\IR_{+}$ respectively, and
\bes{
        g_{n,i,j}(t)&=n_j\cdot f_{X^{\mathrm{e}}_{n,i,j}}(t) e^{-\gamma_i  t}\\
        &=n_j\frac{p_{n,i,j}r_{1,n,i,j}r_{2,n,i,j}e^{-r_{1,n,i,j}t}+(1-p_{n,i,j})r_{1,n,i,j}r_{2,n,i,j}e^{-r_{2,n,i,j}t}}{p_{n,i,j}r_{2,n,i,j}+(1-p_{n,i,j})r_{1,n,i,j}}e^{-\gamma_i t}\\
        &=n_j\frac{p_{n,i,j}\beta_{n,i,j}\lambda_{n,i,j}e^{-r_{1,n,i,j}t}  
        +(1-p_{n,i,j})\beta_{n,i,j}\lambda_{n,i,j}e^{-r_{2,n,i,j}t}}{\lambda_{n,i,j}+\mu_{n,i,j}}e^{-\gamma_i t}.
}

Recall the vector of exponentially truncated point processes, $\xi_{n,Q_i,i,j}$, defined in \eqref{superposition1}, when $Q_{i}$ is specified as a given time instant $T$, $\xi_{n,Q_i,i,j}$ is a vector of deterministically truncated point processes which can be approximated by Poisson point processes, $\xi_{T,i,j}$, in total variation distance as discussed in Section~\ref{proof2.7}. The corresponding intensity measure of the Poisson process is 
\bes{
    \Lambda_{n,T,i,j}(A)&=\IE\bbclc{\sum_{l=1}^{n_j}\delta_{X^{\mathrm{e}}_{n,i,j,l}}(A\cap[0,T])}\\
    &=n_j\IP[X^{\mathrm{e}}_{n,i,j}\in (A\cap[0,T])],\ \forall A\in \F{B}(\IR_+).
}
Specifically, 
\bes{
    \mel\IP[X^{\mathrm{e}}_{n,i,j}\in ([0,t]\cap[0,T])]\\
    &=\begin{cases}
     \displaystyle\frac{p_{n,i,j}r_{2,n,i,j}(1-e^{-r_{1,n,i,j}T})+(1-p_{n,i,j})r_{1,n,i,j}(1-e^{-r_{2,n,i,j}T})}{p_{n,i,j}r_{2,n,i,j}+(1-p_{n,i,j})r_{1,n,i,j}}&\text{if $t>T$,}\\[3ex]
    \displaystyle\frac{p_{n,i,j}r_{2,n,i,j}(1-e^{-r_{1,n,i,j}t})+(1-p_{n,i,j})r_{1,n,i,j}(1-e^{-r_{2,n,i,j}t})}{p_{n,i,j}r_{2,n,i,j}+(1-p_{n,i,j})r_{1,n,i,j}}&\text{if $t\leq T$.}
    \end{cases}
}
Thus
\bes{
    \mel \Lambda_{n,T,i,j}(dt)\\
    &=
    \begin{cases}
     0 & \text{if $t> T$,}\\[3ex]
    \displaystyle n_j\frac{p_{n,i,j}r_{2,n,i,j}r_{1,n,i,j}e^{-r_{1,n,i,j}t} +(1-p_{n,i,j})r_{1,n,i,j}r_{2,n,i,j}e^{-r_{2,n,i,j}t}}{p_{n,i,j}r_{2,n,i,j}+(1-p_{n,i,j})r_{1,n,i,j}}dt
    &\text{if $t\leq T$.}
    \end{cases}
}
Note that $g_{n,i,j}(t)dt=e^{-\gamma_i t}\Lambda_{n,T,i,j}(dt)$ for $t\in[0,T]$. Let 
\be{
  \Lambda_{T,i,j}(dt)=\lim_{n\rightarrow\infty}\Lambda_{n,T,i,j}(dt),
}
and let $\xi_{T,i,j}$ be a Poisson point process with intensity measure $\Lambda_{T,i,j}$ and $N^{*}_{T,i}=\bclr{N^{*}_{T,i,1},\dots,N^{*}_{T,i,k}}$. Through replacing the constant $T$ by the exponential distributed random variable $Q_{i}$, the process $N^{*}_{n,Q_{i},i,j}$ is a Cox process directed by $\Lambda_{n,Q^{(i)},i,j}$. We have
\be{
    \law (N^{*}_{n,Q_{i},i,j})=\int_0^{\infty} \law (\xi_{n,T})(\cdot)d\law (Q_{i})(T),
}
and thus the variation distance between $\xi_{n,Q_i,i}$ and $N^{*}_{n,Q_{i},i}=\bclr{N^{*}_{n,Q_{i},i,1},\dots,N^{*}_{n,Q_{i},i,k}}$ is 
\bes{
  \mel\dtv(\law (\xi_{n,Q_i,i}),\law (N^{*}_{n,Q_{i},i})\\
  &=\sup_{A\in\F{B}(\IR_{+}^k)}| \law (\xi_{n,Q_i,i})(A)-\law (N^{*}_{n,Q_{i},i})(A) |\\
  &\leq \sup_{A\in\F{B}(\IR_+^k)} \int_0^{\infty} \big| \law (\xi_{n,Q_i,i}|_{Q_{i}=T})(A)-\law (N^{*}_{n,T,i})(A) \big|d\law (Q_{i})(T)\\
  &\leq \int_0^{\infty} \dtv\bbclr{\law (\xi_{n,Q_i,i}|_{Q_{i}=T}),\law (N^{*}_{n,T,i})}d\law (Q_{i})(T)\\
  & \leq \int_0^{\infty} \sum_{j\in[k]}\dtv\bbclr{\law (\xi_{n,Q_i,i,j}|_{Q_{i}=T}),\law (N^{*}_{n,T,ij})}d\law (Q_{i})(T)\\
  & \leq \int_0^{\infty} \sum_{j\in[k]} F_{X^{\mathrm{e}}_{n,i,j}}([0,T])d\law (Q_{i})(T),
}
where the last inequality follows from \cite[Theorem 1.4.2]{reiss2012course}. We have 
\bes{
    \mel\int_0^{\infty}\sum_{j\in[k]}F_{X^{\mathrm{e}}_{n,i,j}}([0,T])d\law (Q_{i})(T),\\
    &=\sum_{j\in[k]}\int_0^{\infty}\frac{p_{n,i,j}r_{2,n,i,j}(1-e^{-r_{1,n,i,j}T})+(1-p_{n,i,j})r_{1,n,i,j}(1-e^{-r_{2,n,i,j}T})}{p_{n,i,j}r_{2,n,i,j}+(1-p_{n,i,j})r_{1,n,i,j}}\gamma_i e^{-\gamma_i T}dT\\
     &=\sum_{j\in[k]} \bbbclr{1 - \frac{p_{n,i,j}r_{2,n,i,j}}{p_{n,i,j}r_{2,n,i,j}+(1-p_{n,i,j})r_{1,n,i,j}}\frac{\gamma_i }{r_{1,n,i,j}+\gamma_i } \\
     &\kern14em- \frac{(1-p_{n,i,j})r_{1,n,i,j}}{p_{n,i,j}r_{2,n,i,j}+(1-p_{n,i,j})r_{1,n,i,j}}\frac{\gamma_i }{r_{2,n,i,j}+\gamma_i }}\\
     &=\sum_{j\in[k]} \bbbclr{1 - \frac{p_{n,i,j}r_{2,n,i,j}}{\lambda_{n,i,j}+\mu_{n,i,j}}\frac{\gamma_i }{r_{1,n,i,j}+\gamma_i }  - \frac{(1-p_{n,i,j})r_{1,n,i,j}}{\lambda_{n,i,j}+\mu_{n,i,j}}\frac{\gamma_i }{r_{2,n,i,j}+\gamma_i }} .
}
From the above arguments, we obtain
\besn{\label{TV4xin}
        \mel\dtv(\law (\xi_{n,Q_i,i}),\law (N^{*}_{Q_{i},i}))\\
        &\leq \int_0^{\infty} \sum_{j\in[k]} F_{X^{\mathrm{e}}_{n,i,j}}([0,T])d\law (Q_{i})(T)+\dtv(\law (N^{*}_{n,Q_{i},i},\law (N^{*}_{Q_{i},i})\\
        & \leq \int_0^{\infty} \sum_{j\in[k]}\bbbclr{F_{X^{\mathrm{e}}_{n,i,j}}([0,T]) + \dtv(\Lambda_{n,T,i,j},\Lambda_{T,i,j})}  d\law (Q_{i})(T).
}
Note that
\be{
        \IE\bbbclc{\bbclr{N_{Q_{i},i,j}^{*}(\IR_+ )}^2}
        =\int_0^{\infty} \IE\bbbclc{\bbclr{N_{T,i,j}^{*}(\IR_+)}^2} d\law (Q_{i})(T);
}
here $N_{T,i,j}^{*}(\IR_+)$ is a random variable with distribution $\Po(\Lambda_{T,i,j}(\IR_+))$, thus
\ben{\label{limitN}
        \IE{\bbclr{N_{Q_{i},i,j}^{*}(\IR_+ )}^2}=\int_0^{\infty} \bbclr{\Lambda_{T,i,j}(\IR_+) + (\bR_{0,i,j})^2} \gamma_i  e^{-\gamma_i T}dT.
}

\subsection{Proof of Lemma~\ref{Gxi} and verification of Assumption~\ref{Lambda} and Assumption~\ref{empirical}}\label{proofandverification}
The relevant parameters of Model 3 are $\lambda_{n,i,j}$, $\mu_{n,i,j}$, $\beta_{n,i,j}$ and $\gamma_i $ for $i, j\in[k]$. Note that the infectious period obeys the exponential distribution $\Exp(\gamma_i )$, so the equations \eqref{MarkovSIR1} and \eqref{MarkovSIR2} hold in Model 3. The other three parameters related to $n$ have the forms
\bes{
    \lambda_{n,i,j}=\lambda_{i,j} n_j^{\kappa_{\lambda_{i,j}}},
    \qquad 
    \mu_{n,i,j}=\mu_{i,j} n_j^{\kappa_{\mu_{i,j}}},
    \qquad
    \beta_{n,i,j}=\beta_{i,j} n_j^{\kappa_{\beta_{i,j}}},
}
where $\lambda_{i,j},\mu_{i,j},\beta_{i,j}\geq 0$ and where $\kappa_{\lambda_{i,j}},\kappa_{\mu_{i,j}}\in\IR$ and $\kappa_{\beta_{i,j}}\leq 0$; the latter is so that  the contact intensity of an infected individual is not too strong. We discuss the following cases:
\begin{AutoMultiColEnumerate}[label=$(\arabic*)$]
  \item $\kappa_{\lambda_{i,j}}> 0$ and $\kappa_{\mu_{i,j}}> 0$.
  \item $\kappa_{\lambda_{i,j}}> 0$ and $\kappa_{\mu_{i,j}}< 0$.
  \item $\kappa_{\lambda_{i,j}}> 0$ and $\kappa_{\mu_{i,j}}= 0$.
  \item $\kappa_{\lambda_{i,j}}< 0$ and $\kappa_{\mu_{i,j}}> 0$.
  \item $\kappa_{\lambda_{i,j}}< 0$ and $\kappa_{\mu_{i,j}}< 0$.
  \item $\kappa_{\lambda_{i,j}}< 0$ and $\kappa_{\mu_{i,j}}= 0$.
  \item $\kappa_{\lambda_{i,j}}= 0$ and $\kappa_{\mu_{i,j}}> 0$.
  \item $\kappa_{\lambda_{i,j}}= 0$ and $\kappa_{\mu_{i,j}}< 0$.
  \item $\kappa_{\lambda_{i,j}}= 0$ and $\kappa_{\mu_{i,j}}= 0$.
\end{AutoMultiColEnumerate}
\noindent In what follows, we will show that the vectors of point processes $\xi_{Q_i,i}$ mentioned in Lemma~\ref{Gxi} is $N^{*}_{Q_{i},i}$, the doubly stochastic process introduced in the last subsection. We will derive the total variation distances  $\dtv\bclr{\law \clr{\xi_{n,Q_i,i}},\law \clr{\xi_{Q_i,i}}}$ and  $\|\bR_{0,n,i,j}G_{n,i,j},\bR_{0,i,j}G_{i,j}\|_{\mathrm{TV}}$ for $i$, $j\in[k]$. The calculations are straightforward, but tedious, hence we only give a sketch of the arguments.

\subsubsection{Cases $(1)$, $(2)$, $(3)$, $(4)$, and $(7)$}
In these cases, at least one of $\lambda_{n,i,j}$ and $\mu_{n,i,j}$ tends to infinity as $n\rightarrow\infty$, leading to $r_{1,n,i,j}\rightarrow\infty$. Recalling relation \eqref{expxi}, we know that
\bes{
        \bR_{0,i,j} &=\lim_{n\rightarrow\infty} \IE\clc{\xi_{n,Q_i,i,j}(\IR_+)}\\
    &=\lim_{n\rightarrow\infty} \frac{n_j\beta_{n,i,j}\lambda_{n,i,j}}{(\gamma_i )^2+\gamma_i (\beta_{n,i,j}+\lambda_{n,i,j}+\mu_{n,i,j})+\beta_{n,i,j}\lambda_{n,i,j}}\bbbclr{1+\frac{\gamma_i }{\lambda_{n,i,j}+\mu_{n,i,j}}}\\
    &=\lim_{n\rightarrow\infty} \frac{n_jr_{1,n,i,j}r_{2,n,i,j}}{(r_{1,n,i,j}+\gamma_i )(r_{2,n,i,j}+\gamma_i )}\bbbclr{1+\frac{\gamma_i }{\lambda_{n,i,j}+\mu_{n,i,j}}}
        =\frac{1}{\gamma_i }\lim_{n\rightarrow\infty}n_jr_{2,n,i,j}.
}
To ensure that $\bR_{0,i,j}$ is positive and finite, $r_{2,n,i,j}$ needs to be of order $1/n_j$, which can be ensured by selecting appropriate parameters $\lambda_{i,j}$, $\mu_{i,j}$, $\beta_{i,j}$, $\gamma_i$, and $\kappa_{\lambda_{i,j}}$, $\kappa_{\mu_{i,j}}$, $\kappa_{\beta_{i,j}}$, and it can be shown that $p_{n,i,j}\rightarrow0$. Moreover, the measure $\Lambda_{T,i,j}$ is homogeneous, thus by \eqref{limitN}, we know that $\IE \bclr{N^{(i,j)*}_{Q_{i}}(\IR_+)}^2$ is finite. We now analyse the superposition process $\xi_{n,i,j}$ and the corresponding relative intensity measure. For $0\leq s<t<\infty$,
\bes{
    \mel\lim_{n\rightarrow\infty}n_j\IP\cls{X^{\mathrm{e}}_{n,i,j}\in[s,t]}\\
    &=\lim_{n\rightarrow\infty}n_j\frac{p_{n,i,j}r_{2,n,i,j}(e^{-r_{1,n,i,j}s}-e^{-r_{1,n,i,j}t})+(1-p_{n,i,j})r_{1,n,i,j}(e^{-r_{2,n,i,j}s}-e^{-r_{2,n,i,j}t})}{p_{n,i,j}r_{2,n,i,j}+(1-p_{n,i,j})r_{1,n,i,j}}\\
    &=\lim_{n\rightarrow\infty}n_jr_{2,n,i,j}(t-s)\\
    &=\bR_{0,i,j}\gamma_i (t-s).
}
Under the conditions such that the limit $\lim_{n\rightarrow\infty}n_jr_{2,n,i,j}$ is positive and finite, we know that the corresponding limit relative intensity measure $G_{i,j}$ is
\be{
  G_{i,j}(dt)=\gamma_i  e^{-\gamma_i  t}dt.
}
Let $d_{i,j}=\|p_i\bR_{0,n,i,j}G_{n,i,j}/p_j,p_i\bR_{0,i,j}G_{i,j}/p_j\|_{\mathrm{TV}}$, which has the same order as $\|\bR_{0,n,i,j}G_{n,i,j},\bR_{0,i,j}G_{i,j}\|_{\mathrm{TV}}$.
We summarise the total variation distance between $\bR_{0,n,i,j}G_{n,i,j}$  and $\bR_{0,i,j}G_{i,j}$ under different parameters as follows:
\ba{
     (1)\enskip & \text{If $\kappa_{\lambda_{i,j}}> 0$ and $\kappa_{\mu_{i,j}}> 0$,}\\
       &
    \begin{cases}
      (a) & \text{and if $\kappa_{\lambda_{i,j}}<\kappa_{\mu_{i,j}}$ and $1+\kappa_{\lambda_{i,j}}+\kappa_{\beta_{i,j}}-\kappa_{\mu_{i,j}}=0$,} \\
      & \text{then $d_{i,j}=\max\clc{\bigo(n^{\kappa_{\lambda_{i,j}}-\kappa_{\mu_{i,j}}}), \bigo(n^{-1})}$};\\
    (b) &\text{and if $ \kappa_{\lambda_{i,j}}=\kappa_{\mu_{i,j}}$ and $\kappa_{\beta_{i,j}}=-1$, then $d_{i,j}= \bigo(n^{-1})$;}\\
    (c) & \text{and if $\kappa_{\lambda_{i,j}}>\kappa_{\mu_{i,j}}$ and $\kappa_{\beta_{i,j}}=-1$,}\\
     & \text{then $d_{i,j}=\max\clc{\bigo(n^{-1}),\bigo(n^{\kappa_{\mu_{i,j}}-\kappa_{\lambda_{i,j}}})}.$}
    \end{cases}\\
     (2)\enskip & \text{If $\kappa_{\lambda_{i,j}}> 0$ and $\kappa_{\mu_{i,j}}< 0$ and $\kappa_{\beta_{i,j}}=-1$},\\
     &\text{then $d_{i,j}=\max\clc{\bigo(n^{-1}),\bigo(n^{\kappa_{\mu_{i,j}}-\kappa_{\lambda_{i,j}}})}$.}\\
     (3)\enskip &\text{If $\kappa_{\lambda_{i,j}}> 0$ and $\kappa_{\mu_{i,j}}= 0$ and $\kappa_{\beta_{i,j}}=-1$,}\\
     &\text{then $d_{i,j}=\max\bclc{\bigo(n^{-1}),\bigo(n^{-\kappa_{\lambda_{i,j}}})}$.}\\
     (4)\enskip &\text{If $\kappa_{\lambda_{i,j}}< 0$ and $\kappa_{\mu_{i,j}}> 0$ and $1+\kappa_{\lambda_{i,j}}+\kappa_{\beta_{i,j}}-\kappa_{\mu_{i,j}}=0$,}\\
     &\text{then $d_{i,j}=\max\clc{\bigo(n^{\kappa_{\beta_{i,j}}-\kappa_{\mu_{i,j}}}), \bigo(n^{-1}),  \bigo(n^{\kappa_{\lambda_{i,j}}-\kappa_{\mu_{i,j}}})}$.}\\
     (7)\enskip &\text{If $\kappa_{\lambda_{i,j}}= 0$ and $\kappa_{\mu_{i,j}}> 0$ and $1+\kappa_{\beta_{i,j}}=\kappa_{\mu_{i,j}}$,}\\
     &\text{then $d_{i,j}=\max\clc{\bigo(n^{-1}),  \bigo(n^{-\kappa_{\mu_{i,j}}})}$.}
}

\subsubsection{Case $(5)$}
In this case, both $\lambda_{n,i,j}$ and $\mu_{n,i,j}$ tend to 0. When $\kappa_{\beta_{i,j}}=0$, the limit~$\bR_{0,i,j}$ is either 0 or infinity. When $\kappa_{\beta_{i,j}}<0$, $\bR_{0,i,j}$ could be positive and finite for appropriate selections of $\lambda_{i,j}$, $\mu_{i,j}$, $\beta_{i,j}$, $\kappa_{\lambda_{i,j}}$, and $\kappa_{\mu_{i,j}}$ that make the limit relative intensity measure $G_{i,j}(dt)=\gamma_i  e^{-\gamma_i  t}dt$.  Moreover, the measure~$\Lambda_{T,i,j}$ is homogeneous, thus by \eqref{limitN} we know that $\lim_{n\rightarrow\infty}\IE\clc{( N^{(i,j)*}_{Q_{i}}(\IR_+))^2}$ is finite. The total variation distance between $\bR_{0,n,i,j}G_{n,i,j}$  and $\bR_{0,i,j}G_{i,j}$ is as follows:
\bes{
     (5)\enskip &\text{If $\kappa_{\lambda_{i,j}}< 0$ and $\kappa_{\mu_{i,j}}< 0$,}\\
       &\begin{cases}
      (a)& \text{and if $\kappa_{\lambda_{i,j}}=\kappa_{\mu_{i,j}}$ and $\kappa_{\beta_{i,j}}=-1$, then $d_{i,j}=\max\bclc{\bigo(n^{-1}), \bigo(n^{\kappa_{\lambda_{i,j}}})}$};\\
    (b)&\text{and if $\kappa_{\lambda_{i,j}}>\kappa_{\mu_{i,j}}$ and $\kappa_{\beta_{i,j}}=-1$,}\\
     &\text{then $d_{i,j}=\max\bclc{\bigo(n^{-1}), \bigo(n^{\kappa_{\lambda_{i,j}}})$, $\bigo(n^{\kappa_{\mu_{i,j}}-\kappa_{\lambda_{i,j}}})}$;}\\
    (c)&\text{and if $\kappa_{\lambda_{i,j}}<\kappa_{\mu_{i,j}}$ and $1+\kappa_{\beta_{i,j}}+\kappa_{\lambda_{i,j}}-\kappa_{\mu_{i,j}}=0$},\\
     &\text{then $d_{i,j}=\max\bclc{\bigo(n^{\kappa_{\mu_{i,j}}}),\bigo(n^{\kappa_{\beta_{i,j}}}),\bigo(n^{\kappa_{\lambda_{i,j}}-\kappa_{\mu_{i,j}}})}$}.
    \end{cases}
}

\subsubsection{Cases $(8)$ and $(9)$}
In these two cases, $\lambda_{n,i,j}=\lambda_{i,j}$ is a positive constant independent of $n$, and the limit expected number of offspring is positive and finite if $\kappa_{\beta_{i,j}}=-1$ and
\bes{
        \bR_{0,i,j} &=\lim_{n\rightarrow\infty} \frac{n_j\beta_{n,i,j}\lambda_{n,i,j}}{(\gamma_i )^2+\gamma_i (\beta_{n,i,j}+\lambda_{n,i,j}+\mu_{n,i,j})+\beta_{n,i,j}\lambda_{n,i,j}}\bbbclr{1+\frac{\gamma_i }{\lambda_{n,i,j}+\mu_{n,i,j}}}\\
          &=\begin{cases}
       \displaystyle\frac{\beta_{i,j}}{\gamma_i }&\text{if $\kappa_{\mu_{i,j}}<0$,}\\[3ex]
        \displaystyle\frac{\lambda_{i,j}\beta_{i,j}}{\lambda_{i,j}+\mu_{i,j}}\frac{1}{\gamma_i}
        &\text{if $\kappa_{\mu_{i,j}}=0$.}
    \end{cases}
}
For $0\leq s<t<\infty$,
\bes{
    \lim_{n\rightarrow\infty}n_j\IP\cls{X_e\in[s,t]}
    &=\bR_{0,i,j}\gamma_i (t-s).
}
Thus the limit relative intensity measure is $G_{i,j}(dt)=\gamma_i  e^{-\gamma_i  t}dt$. The measure $\Lambda_{T,i,j}$ is homogeneous, and thus by \eqref{limitN} we know that $\IE\clc{( N^{(i,j)*}_{Q_{i}}(\IR_+))^2}$ is finite.
The total variation distance between $\bR_{0,n,i,j}G_{n,i,j}$  and $\bR_{0,i,j}G_{i,j}$ is as follows:
\ba{
     (8)\enskip&\text{If $\kappa_{\lambda_{i,j}}= 0$ and $\kappa_{\mu_{i,j}}< 0$ and $\kappa_{\mu_{i,j}}=-1$,}\\
     &\text{then $d_{i,j}= \max\clc{\bigo(n^{-1}),\bigo(n^{\kappa_{\mu_{i,j}}})}$.}\\
     (9)\enskip&\text{If $\kappa_{\lambda_{i,j}}= 0$ and $\kappa_{\mu_{i,j}}= 0$ and $\kappa_{\beta_{i,j}}=-1$, then $d_{i,j}= \bigo(n^{-1})$.}
}

\subsubsection{Case $(6)$}
In this case, $\lambda_{n,i,j}\rightarrow0$ and $\mu_{n,i,j}=\mu_{i,j}$ is a constant independent of $n$. Note that
\bes{
    \bR_{0,i,j} 
    &=\lim_{n\rightarrow\infty} \IE\clc{\xi_{n,Q_i,i,j}(\IR_+)}\\
    &=\lim_{n\rightarrow\infty} \frac{n_j\beta_{n,i,j}\lambda_{n,i,j}}{ \gamma_i^2+\gamma_i (\beta_{n,i,j}+\lambda_{n,i,j}+\mu_{n,i,j})+\beta_{n,i,j}\lambda_{n,i,j}}\bbbclr{1+\frac{\gamma_i }{\lambda_{n,i,j}+\mu_{n,i,j}}}.
}
Now, $\bR_{0,i,j}$ is positive and finite if and only if $1+\kappa_{\lambda_{i,j}}+\kappa_{\beta_{i,j}}=0$.  When $\kappa_{\beta_{i,j}}<0$, the limit is
\be{
        \bR_{0,i,j}  =
        \begin{cases}
       \displaystyle\frac{\lambda_{i,j}\beta_{i,j}}{\mu_{i,j}\gamma_i }
       &\text{if $\kappa_{\beta_{i,j}}\neq-1$,}\\[3ex]
       \displaystyle\frac{\lambda_{i,j}\beta_{i,j}}{\lambda_{i,j}+\mu_{i,j}}\frac{1}{\gamma_i }
       &\text{if $\kappa_{\beta_{i,j}}=-1$.}
    \end{cases}
}
For $0\leq s<t<\infty$,
\be{
    \lim_{n\rightarrow\infty}n_j\IP\cls{X^{\mathrm{e}}_{n,i,j}\in[s,t]}
    =\bR_{0,i,j}\gamma_i (t-s).
}
Thus, $G_{i,j}(dt)=\gamma_i  e^{-\gamma_i  t}dt$. Note that $r_{2,n,i,j}$ tends to 0 as $n\rightarrow\infty$ with order of $n^{\kappa_{\lambda_{i,j}}+\kappa_{\beta_{i,j}}}=n^{-1}$, and the measure $\Lambda_{T,i,j}$ is homogeneous, thus by \eqref{limitN} we know that $\IE\clc{( N^{(i,j)*}_{Q_{i}}(\IR_+))^2}$ is finite. When $\kappa_{\beta_{i,j}}=0$, thus $\kappa_{\lambda_{i,j}}=-1$,
\bes{
        \bR_{0,i,j} &=\lim_{n\rightarrow\infty} \frac{n_j\beta_{n,i,j}\lambda_{n,i,j}}{(\gamma_i )^2+\gamma_i (\beta_{n,i,j}+\lambda_{n,i,j}+\mu_{n,i,j})+\beta_{n,i,j}\lambda_{n,i,j}}\bbbclr{1+\frac{\gamma_i }{\lambda_{n,i,j}+\mu_{n,i,j}}}\\
        &=\frac{\lambda_{i,j}\beta_{i,j}(\mu_{i,j}+\gamma_i )}{(\beta_{i,j}+\mu_{i,j}+\gamma_i )\mu_{i,j}\gamma_i }.
}
Moreover,
\bes{
    \mel\lim_{n\rightarrow\infty}n_j\IP\cls{X^{\mathrm{e}}_{n,i,j}\in[s,t]} \\
    &=\lim_{n\rightarrow\infty}n_j\frac{p_{n,i,j}r_{2,n,i,j}(e^{-r_{1,n,i,j}s}-e^{-r_{1,n,i,j}t})+(1-p_{n,i,j})r_{1,n,i,j}(e^{-r_{2,n,i,j}s}-e^{-r_{2,n,i,j}t})}{p_{n,i,j}r_{2,n,i,j}+(1-p_{n,i,j})r_{1,n,i,j}}\\
    &=\frac{\lambda_{i,j}(\beta_{i,j})^2}{\mu_{i,j}(\mu_{i,j}+\beta_{i,j})^2}\clr{e^{-(\beta_{i,j}+\mu_{i,j})s}-e^{-(\beta_{i,j}+\mu_{i,j})t}}+\frac{\lambda_{i,j}\beta_{i,j}}{\mu_{i,j}+\beta_{i,j}}(t-s),
}
as well as
\bes{
   \mel\lim_{n\rightarrow\infty} G_{n,i,j}(dt)\\
   &=\
   \lim_{n\rightarrow\infty} \frac{[p_{n,i,j}e^{-r_{1,n,i,j}t}+(1-p_{n,i,j})e^{-r_{2,n,i,j}t}](r_{1,n,i,j}+\gamma_i )(r_{2,n,i,j}+\gamma_i )}{p_{n,i,j}(r_{2,n,i,j}+\gamma_i )+(1-p_{n,i,j})(r_{1,n,i,j}+\gamma_i )} e^{-\gamma_i  t}dt\\
   &=\
        \frac{(\beta_{i,j} e^{-(\beta_{i,j}+\mu_{i,j})t}+\mu_{i,j})(\mu_{i,j}+\beta_{i,j}+\gamma_i )\gamma_i }{\beta_{i,j}\gamma_i +\mu_{i,j}(\mu_{i,j}+\beta_{i,j}+\gamma_i )} e^{-\gamma_i  t}dt.
}
Thus the contact processes in backward branching process are non-homogeneous Poisson process. Note that
\be{
  \Lambda_{T,i,j}(\IR_+)=\frac{(\beta_{i,j})^2\lambda_{i,j}\frac{1- e^{-(\mu_{i,j}+\beta_{i,j})T}}{\mu_{i,j}+\beta_{i,j}} +\mu_{i,j}\beta_{i,j}\lambda_{i,j}T}{(\mu_{i,j}+\beta_{i,j}) \mu_{i,j}} ,
}
thus by \eqref{limitN} we know that $\IE\clc{( N^{(i,j)*}_{Q_{i}}(\IR_+))^2}$ is finite. The total variation distance between $\bR_{0,n,i,j}G_{n,i,j}$  and $\bR_{0,i,j}G_{i,j}$ is as follows:
\bes{
     (6)\enskip&\text{If $\kappa_{\lambda_{i,j}}< 0$ and $\kappa_{\mu_{i,j}}= 0$,}\\
       &\begin{cases}
     (a)& \text{and if $\kappa_{\beta_{i,j}}<0$ and $\kappa_{\lambda_{i,j}}+\kappa_{\beta_{i,j}}=-1$,}\\
       &\text{then $d_{i,j}=\max\clc{\bigo(n^{\kappa_{\beta_{i,j}}}), \bigo( n^{\kappa_{\lambda_{i,j}}} ), \bigo(n^{-1})};$}\\
   (b)&\text{and if $\kappa_{\beta_{i,j}}=0$ and $\kappa_{\lambda_{i,j}}=-1$,}\\
     &\text{then $d_{i,j}= \bigo(n^{-1}).$}
    \end{cases}
}

\subsection{Additional constraints on parameters}

Based on the results of Sections~\ref{preliminaryresult} and \ref{proofandverification} (equation \eqref{TV4xin}, and the total variation distances $d_{i,j}$), it can be shown that the upper bound of the total variation distance between the distributions of $\xi_{n,Q_i,i}$ and $\xi_{Q_i,i}$ matches the upper bound of the aggregated total variation distances between $\bR_{0,n,i,j}G_{n,i,j}$ and $\bR_{0,i,j}G_{i,j}$ across $[k]$. While the calculations for each of the nine cases are straightforward, they are also quite tedious, hence we omit them. In summary, to ensure Assumption~\ref{empirical}(ii) and (iii), we need that
\be{
d_{i,j}=
\begin{cases}
\lito\bclr{n^{-1/3}} &\text{if $k=1$,}\\
\lito\bclr{n^{-17/24}} &\text{if $k>1$.}\\
\end{cases}
}
This leads to the additional parameter constraints in Tables~\ref{singletable} and~\ref{table:constraints}.

\subsection{Weak and strong forms of varying edges on epidemic curves}
Differentiating both sides of \eqref{ratioS}, we obtain
\bes{
    \frac{d}{dt}\bss^{\iota}_{ \nu}(t)&=\bss^{\iota}_{\nu}(t)\sum_{i\in[k]}\bbbclr{\frac{p_i}{p_{\nu}}\bR_{0,i,\nu}\int_0^{\infty}\frac{d}{ds}\bss^{\iota}_{i}(s)|_{s=t-u}g_{i,j}(u)du}\\
    &=\bss^{\iota}_{\nu}(t)\sum_{i\in[k]}\frac{p_i}{p_{\nu}}\bbbclr{\bR_{0,i,j}g_{i,\nu}(0)\bss^{\iota}_{i}(t)\\
    &\quad+ \bR_{0,i,\nu}\int_{-\infty}^{t}\bss^{\iota}_{i}(v)g_{i,\nu}(t-v)dv},
}
where we have used the fact that $g_{i,\nu}(\infty)=0$ and $\bss^{(i,\nu)}(-\infty)=1$ for $i,\nu\in[k]$. Thus the epidemic curves of Model 1 can be summarised by the integro-differential equations
\besn{\label{IDEModel3}
    \frac{d}{dt}\bss^{\iota}_{ \nu}(t)&=\bss^{\iota}_{\nu}(t)\sum_{i\in[k]}\frac{p_i}{p_{\nu}}\bbbclr{\bR_{0,i,\nu}g_{i,\nu}(0)\bss^{\iota}_{i}(t)\\
    &\quad+ \bR_{0,i,\nu}\int_{-\infty}^{t}\bss^{\iota}_{i}(v)g_{i,\nu}(t-v)dv},\\
    \frac{d}{dt}\bsi^{\iota}_{ \nu}(t)&=-\bss^{\iota}_{\nu}(t)\sum_{i\in[k]}\frac{p_i}{p_{\nu}}\bbbclr{\bR_{0,i,\nu}g_{i,\nu}(0)\bss^{\iota}_{i}(t)\\
    &\quad+ \bR_{0,i,\nu}\int_{-\infty}^{t}\bss^{\iota}_{i}(v)g_{i,\nu}(t-v)dv}-\gamma_{\nu} \bsi^{\iota}_{\nu}(t),\\
    \frac{d}{dt}\bsr^{\iota}_{ \nu}(t)&=\gamma_{\nu} \bsi^{\iota}_{\nu}(t),
}
with initial conditions $\bss^{\iota}_{\nu}(-\infty)=1$, $\bsi^{\iota}_{\nu}(-\infty)=0$, $\bsr^{\iota}_{\nu}(-\infty)=0$ and the constraint \eqref{psiEW}. 

The integro-differential equations in \eqref{IDEModel3} can be expressed in various forms. We will divide our discussion into three parts. First, we consider the case where all of the $\clc{G_{n,i,j}}_{i,j\in[k]}$ converge to homogeneous measures. Second, we examine the case where all of the $\clc{G_{n,i,j}}_{i,j\in[k]}$ converge to non-homogeneous measures. Finally, we consider mixed cases that combine the behaviors of the first two cases.

\subsubsection{Weak form: All of the relative intensity measures are homogeneous}
For the homogeneous cases, we have $G_{i,j}(dt)=\gamma_i  e^{-\gamma_i  t}dt$ and this is the general case we have discussed before. In these cases, the effect of the dynamic structure of the dynamic graph is ``weak'', reflecting only on the expected number $\bR_{0,i,j}$.
The epidemic curves are described by the differential equations
\ba{
     \frac{d}{dt}\bss^{\iota}_{ \nu}(t)&=-\bss^{\iota}_{\nu}(t)\bbbclr{\sum_{i\in[k]}\frac{p_i}{p_{\nu}}\bR_{0,i,\nu}\gamma_i \bsi_{i}(t)},\\
     \frac{d}{dt}\bsi^{\iota}_{ \nu}(t)&=\bss^{\iota}_{\nu}(t)\bbbclr{\sum_{i\in[k]}\frac{p_i}{p_{\nu}}\bR_{0,i,j}\gamma_i \bsi_{i}(t)}-\gamma \bsi^{\iota}_{\nu}(t),\\
     \frac{d}{dt}\bsr^{\iota}_{ \nu}(t)&=\gamma_{\nu} \bsi^{\iota}_{\nu}(t),
}
with appropriate initial conditions.

\subsubsection{Strong form: All of the limit relative intensity measures are non-homogeneous}\label{constraintonsimplecase}
We will now focus on a special case where $\kappa_{\mu_{i,j}}=\kappa_{\beta_{i,j}}=0$; that is, $\beta_{n,i,j}$ and $\mu_{n,i,j}$ are constants for all $i$, $j\in[k]$. In this case, by introducing some new variables, we can convert the integro-differential equations \eqref{IDEModel3} to a system of ordinary differential equations. This conversion makes the dynamic structure of the random graph more interpretable.

Note that the relative intensity measure can be rewritten as
\bes{
   \mel\lim_{n\rightarrow\infty}  G_{n,i,j}(dt)\\
     &=\frac{(\beta_{i,j} e^{-(\beta_{i,j}+\mu_{i,j})t}+\mu_{i,j})(\mu_{i,j}+\beta_{i,j}+\gamma_i )\gamma_i }{\beta_{i,j}\gamma_{i,j}+\mu_{i,j}(\mu_{i,j}+\beta_{i,j}+\gamma_i )} e^{-\gamma_i  t}dt\\
      &=\frac{1}{\bR_{0,i,j}}\frac{\lambda_{i,j}}{\mu_{i,j}}\beta_{i,j} e^{-(\mu_{i,j} +\beta_{i,j}+\gamma_i )t}dt\\
      &\quad+\frac{1}{\bR_{0,i,j}}\frac{\lambda_{i,j}}{\mu_{i,j}+\beta_{i,j}}\beta_{i,j}(1-e^{-(\mu_{i,j}+\beta_{i,j})t})e^{-\gamma_i  t}dt.
}
Recall that
\bes{
    \frac{d}{dt}\bss^{\iota}_{ j}(t)
    &=\bss^{\iota}_{j}(t)\sum_{i\in[k]}\frac{p_i}{p_j}\bbbcls{\int_0^{\infty}\bbbclr{\frac{d}{ds}\bss^{\iota}_{i}(s)|_{s=t-u}\frac{\lambda_{i,j}}{\mu_{i,j}}\beta_{i,j} e^{-(\mu_{i,j}+\beta_{i,j}+\gamma_i )u}\\
    &\quad+\frac{d}{dt}\bss^{\iota}_{i}(s)|_{s=t-u}\frac{\lambda_{i,j}}{\mu_{i,j}+\beta_{i,j}}\beta_{i,j}(1-e^{-(\mu_{i,j}+\beta_{i,j})u})e^{-\gamma_i  u}}du}.
}
Define a new variable $\bsl^{\iota}_{\mathrm{c},i,j}$, for every $t\in\IR$ let
\bes{
    \bsl^{\iota}_{\mathrm{c},i,j}(t)&=- \int_0^{\infty} \frac{d}{dt}\bss^{\iota}_{i}(s)|_{s=t-u}\frac{\lambda_{i,j}}{\mu_{i,j}} e^{-(\mu_{i,j}+\beta_{i,j}+\gamma_i   )u}du\\
    &=-  \frac{\lambda_{i,j}}{\mu_{i,j}}  \bss^{\iota}_{i}(t)+ (\mu_{i,j}+\beta_{i,j}+\gamma_i )\int_{0}^{\infty} \bss^{\iota}_{i}(t-u)\frac{\lambda_{i,j}}{\mu_{i,j}} e^{-(\mu_{i,j}+\beta_{i,j}+\gamma_i )u}du,
}
which satisfies
\bes{
 \frac{d}{dt}\bsl^{\iota}_{\mathrm{c}, i,j}(t) 
    &=-  \frac{\lambda_{i,j}}{\mu_{i,j}} \frac{d}{dt}\bss^{\iota}_{i}(t)+ (\mu_{i,j}+\beta_{i,j}+\gamma_i )\int_{0}^{\infty} \frac{d}{ds}\bss^{\iota}_{i}(s)|_{s=t-u}\frac{\lambda_{i,j}}{\mu_{i,j}} e^{-(\mu_{i,j}+\beta_{i,j}+\gamma_i )u}du \\
    &=- \frac{\lambda_{i,j}}{\mu_{i,j}} \frac{d}{dt}\bss^{\iota}_{i}(t)- (\mu_{i,j}+\beta_{i,j}+\gamma_i )\bsl^{\iota}_{\mathrm{c},i,j}(t).
}
Recall that
\bes{
    \bsi^{\iota}_{j}(t)&=-\int_0^{\infty}\bss^{\iota\prime}_{ j}(t-u)e^{-\gamma_i  u}du.
}
Let 
\bes{
\bsl^{\iota}_{\mathrm{d},i,j}(t)&=- \int_0^{\infty}\frac{d}{ds}\bss^{\iota}_{i}(s)|_{s=t-u} \frac{\lambda_{i,j}}{\mu_{i,j}+\beta_{i,j}}(1-e^{-(\mu_{i,j}+\beta_{i,j})u})e^{-\gamma_i  u}du\\
&=\frac{\lambda_{i,j}}{\mu_{i,j}+\beta_{i,j}}\bsi^{\iota}_{i}(t) - \frac{\mu_{i,j}}{\mu_{i,j}+\beta_{i,j}}\bsl^{\iota}_{\mathrm{c},i,j}(t),
}
which satisfies
\bes{
\bsl^{\iota\prime}_{\mathrm{d}, i,j}(t)&= \mu_{i,j} \bsl^{\iota}_{\mathrm{c},i,j}(t)- \gamma_i  \bsl^{\iota}_{\mathrm{d},i,j}(t).
}
So we know that the epidemic curves satisfy the following differential equations,
\besn{\label{SLLR}
        \frac{d}{dt}\bss^{\iota}_{ j}(t)&=-\bss^{\iota}_{j}(t)\sum_{i\in[k]}\frac{p_i}{p_j}\beta_{i,j}\clr{\bsl^{\iota}_{\mathrm{c},i,j}(t)+ \bsl^{\iota}_{\mathrm{d},i,j}(t)},\\
        \frac{d}{dt}\bsl^{\iota}_{\mathrm{c}, i,j}(t)&=\bss^{\iota}_{i}(t)\sum_{l\in[k]}\frac{p_l}{p_j}\frac{\lambda_{li}}{\mu_{li}}\beta_{li}\clr{\bsl^{\iota}_{\mathrm{c}, li}(t)+ \bsl^{\iota}_{\mathrm{d}, li}(t)}
        -(\mu_{i,j}+\beta_{i,j}+\gamma_i ) \bsl^{\iota}_{\mathrm{c},i,j}(t),\\
        \frac{d}{dt}\bsl^{\iota}_{\mathrm{d}, i,j}(t)&=\mu_{i,j} \bsl^{\iota}_{\mathrm{c},i,j}(t)-\gamma_i  \bsl^{\iota}_{\mathrm{d},i,j}(t),\\
        \frac{d}{dt}\bsi^{\iota}_{ j}(t)&=\bss^{\iota}_{j}\sum_{i\in[k]}\frac{p_i}{p_j}\beta_{i,j}\clr{\bsl^{\iota}_{\mathrm{c},i,j}(t)+ \bsl^{\iota}_{\mathrm{d},i,j}(t)} -\gamma_j  \bsi^{\iota}_{j}(t)\\
        \frac{d}{dt}\bsr^{\iota}_{ j}(t)&=\gamma_j  \bsi^{\iota}_{j}(t).
}
Recall that the expected number of $\xi_{Q_i,i,j}$ ($i$, $j\in[k]$) can be derived from the varying random graph with $n$ vertices through
\bes{
\bR_{0,i,j} &=\frac{\lambda_{i,j}\beta_{i,j}(\mu_{i,j}+\gamma_i )}{(\beta_{i,j}+\mu_{i,j}+\gamma_i )\mu_{i,j}\gamma_i }\\
&=\lim_{n\rightarrow\infty}\bbbclr{n_j\frac{\frac{\lambda_{i,j}}{n_j}}{\frac{\lambda_{i,j}}{n_j}+\mu_{i,j}}\frac{\beta_{i,j}}{\mu_{i,j}+\beta_{i,j}+\gamma_i }
+n_j\frac{\mu_{i,j}}{\frac{\lambda_{i,j}}{n_j}+\mu_{i,j}}\frac{\frac{\lambda_{i,j}}{n_j}}{\frac{\lambda_{i,j}}{n_j}+\gamma_i }\frac{\beta_{i,j}}{\mu_{i,j}+\beta_{i,j}+\gamma_i }}.
}
The second equality provides an interpretation of how to calculate the expected number. In the equilibrium state, a vertex $x$ of type~$i$ is expected to have, on average, $n_j\frac{\lambda_{i,j}/n_j}{\lambda_{i,j}/n_j+\mu_{i,j}}$ connected neighbours of type~$j$ and $n_j\frac{\mu_{i,j}}{\lambda_{i,j}/n_j+\mu_{i,j}}$ disconnected ``neighbours'' of type~$j$. Among the connected neighbours, a proportion of $\frac{\beta_{i,j}}{\mu_{i,j}+\beta_{i,j}+\gamma_j }$ is expected to be infected by vertex $x$. Among the disconnected neighbours, a proportion of $\frac{\lambda_{i,j}/n_j}{\lambda_{i,j}/n_j+\gamma_i }$ is expected to have an opportunity to be infected by $x$, with a probability of $\frac{\beta_{i,j}}{\mu_{i,j}+\beta_{i,j}+\gamma_i }$.

\subsubsection{Mixed case}
Now we consider the case in which some of $\clc{G_{i,j}}_{i,j\in[k]}$ are homogeneous and some of them are non-homogeneous.

For every $\nu\in[k]$, suppose that all selected parameters satisfy the assumptions of Theorem~\ref{maintheorem}. Let
\bes{
         K^{\nu}_{\mathrm{h}}&=\clc{i \in [k]:\text{$H_{i,\nu}$ is homogeneous}},\\
         K^{\nu}_{\mathrm{nh}}&=\clc{u\in [k]:\text{$H_{u,\nu}$ is non-homogeneous}}.
}
Then for every $j\in[k]$,
\bes{
    \frac{d}{dt}\bss^{\iota}_{ j}(t)&=\bss^{\iota}_{j}(t)\sum_{i\in[k]}\bbbclr{\frac{p_i}{p_j}\bR_{0,i,j}\int_0^{\infty}\frac{d}{ds}\bss^{\iota}_{i}(s)|_{s=t-u}g_{i,j}(u)du}\\
    &=-\bss^{\iota}_{j}(t)\sum_{i\in K^j_{\mathrm{h}}}\bbbclr{\frac{p_i}{p_j}\bR_{0,i,j}\gamma_i \bsi^{\iota}_{i}} \\
    &\quad-\bss^{\iota}_{j}(t)\sum_{u\in K^j_{\mathrm{nh}}}\frac{p_i}{p_j}\frac{\lambda_{li}}{\mu_{li}}\beta_{li}\clr{\bsl^{\iota}_{\mathrm{c}, uj}(t)+ \bsl^{\iota}_{\mathrm{d}, uj}(t)},
}
where for $u\in K^j_{\mathrm{nh}}$,
\bes{
        \frac{d}{dt}\bsl^{\iota}_{\mathrm{c}, uj}(t)
        &=\bss^{\iota}_{u}(t)\sum_{i\in K^u_\mathrm{h}}\bbbclr{\frac{p_i}{p_j}\bR_{0,i,u}\gamma_i \bsi^{\iota}_{i}}\\
        &\quad+\bss^{\iota}_{u}(t)\sum_{h\in K^u_\mathrm{nh}}\frac{p_h}{p_j}\frac{\lambda_{li}}{\mu_{li}}\beta_{li}\clr{\bsl^{\iota}_{\mathrm{c}, hj}(t)+ \bsl^{\iota}_{\mathrm{d}, hj}(t)}\\
        &\quad- (\mu_{uj}+\beta_{uj}+\gamma_u ) \bsl^{\iota}_{\mathrm{c}, uj}(t),\\
        \frac{d}{dt}\bsl^{\iota}_{\mathrm{d}, uj}(t)&=\mu_{uj} \bsl^{\iota}_{\mathrm{c}, uj}(t)-\gamma_u  \bsl^{\iota}_{\mathrm{d}, uj}(t),\\
}
Thus the epidemic curves of Model 1 satisfy
\bes{
          \frac{d}{dt}\bss^{\iota}_{ j}(t) &=-\bss^{\iota}_{j}(t)\sum_{i\in K^j_{\mathrm{h}}}\bbbclr{\frac{p_i}{p_j}\bR_{0,i,j}\gamma_i \bsi^{\iota}_{i}} \\
          &-\bss^{\iota}_{j}(t)\sum_{u\in K^j_{\mathrm{nh}}}\frac{p_u}{p_j}\beta_{uj}\clr{\bsl^{\iota}_{\mathrm{c}, uj}(t)+ \bsl^{\iota}_{\mathrm{d}, uj}(t)},\\
          \frac{d}{dt}\bsi^{\iota}_{ j}(t)&=\bss^{\iota}_{j}(t)\sum_{i\in K^j_{\mathrm{h}}}\bbbclr{\frac{p_i}{p_j}\bR_{0,i,j}\gamma_i \bsi^{\iota}_{i}} \\
          &\quad+\bss^{\iota}_{j}(t)\sum_{u\in K^j_{\mathrm{nh}}}\frac{p_u}{p_j}\beta^{\iota}_{uj}\clr{\bsl^{\iota}_{\mathrm{c}, uj}(t)+ \bsl^{\iota}_{\mathrm{d}, uj}(t)}-\gamma_j \bsi^{\iota}_{j},\\
          \frac{d}{dt}\bsl^{\iota}_{\mathrm{c}, uj}(t)&=\bss^{\iota}_{u}(t)\sum_{i\in K^u_\mathrm{h}}\bbbclr{\frac{p_i}{p_j}\bR_{0,i,u}\gamma_i \bsi^{\iota}_{i}} \\
          &\quad+\bss^{\iota}_{u}(t)\sum_{h\in K^u_\mathrm{nh}}\frac{p_h}{p_j}\frac{\lambda_{hu}}{\mu_{hu}}\beta_{hj}\clr{\bsl^{\iota}_{\mathrm{c}, hj}(t)+ \bsl^{\iota}_{\mathrm{d}, hj}(t)} \\
          &\quad- (\mu_{uj}+\beta_{uj}+\gamma_u ) \bsl^{\iota}_{\mathrm{c}, uj}(t),\\
        \frac{d}{dt}\bsl^{\iota}_{\mathrm{d}, uj}(t)&=\mu_{uj} \bsl^{(uj)}_{\mathrm{c}}(t)-\gamma_u  \bsl^{(uj)}_{\mathrm{d}}(t),\\
        \frac{d}{dt}\bsr^{\iota}_{ j}&=\gamma_j \bsi^{\iota}_{j}.
}

\section*{Acknowledgements}

This project was supported by the Singapore Ministry of Education Academic Research Fund Tier 2 grant MOE2018-T2-2-076.

\setlength{\bibsep}{0.5ex}
\def\bibfont{\small}

\bibliographystyle{natbibstyle}

\begin{sidewaystable}
\centering
\begin{tabular}{r@{\hspace{0.3em}}lr@{\hspace{0.3em}}lll} 
\toprule
\multicolumn{2}{c}{Parameter ranges} & \multicolumn{2}{c}{Parameter constraints} & $\gamma\bR_0$\\
\midrule
\multirow{3}{*}{(1)} & \multirow{3}{*}{$\kappa_{\lambda}> 0$, $\kappa_{\mu}> 0$} & (a) & $\kappa_{\lambda}-\kappa_{\mu}<-1/3$, $1+\kappa_{\lambda}+\kappa_{\beta}-\kappa_{\mu}=0$ & $\lambda\beta/\mu$\\
& & (b) & $\kappa_{\lambda}=\kappa_{\mu}$, $\kappa_{\beta}=-1$ & $\lambda\beta/(\lambda+\mu)$\\
& & (c) & $-1/3+\kappa_{\lambda}>\kappa_{\mu}$, $\kappa_{\beta}=-1$ & $\beta$\\
\midrule
(2) & $\kappa_{\lambda}> 0$, $\kappa_{\mu}< 0$ & & $\kappa_{\beta}=-1$, $\kappa_{\mu}-\kappa_{\lambda}<-1/3$ & $\beta$\\
\midrule
(3) & $\kappa_{\lambda}> 0$, $\kappa_{\mu}= 0$ & & $\kappa_{\beta}=-1$ & $\beta$\\
\midrule
(4) & $\kappa_{\lambda}< 0$, $\kappa_{\mu}> 0$ & & $\kappa_{\lambda}<-4/3$, $1+\kappa_{\lambda}+\kappa_{\beta}-\kappa_{\mu}=0$, $\kappa_{\beta}<-4/3$ & $\lambda\beta/\mu$\\
\midrule
\multirow{4}{*}{(5)} & \multirow{4}{*}{$\kappa_{\lambda}< 0$, $\kappa_{\mu}< 0$} & (a) & $ \kappa_{\lambda}=\kappa_{\mu}<-1/3$, $\kappa_{\beta}=-1$ & $\beta\lambda/(\lambda+\mu)$ \\
& & (b) & $-1/3>\kappa_{\lambda}>\kappa_{\mu}$, $\kappa_{\beta}=-1$, $-1/3>\kappa_{\mu}-\kappa_{\lambda}$ & $\beta\lambda/\lambda$\\
& & \multirow{2}{*}{(c)} & $0<\kappa_{\mu}-\kappa_{\lambda}<-1/3$, $1+\kappa_{\beta}+\kappa_{\lambda}-\kappa_{\mu}=0$ & \multirow{2}{*}{$\beta\lambda/\mu$} \\
& & & $\kappa_{\lambda}<-1/3$, $\kappa_{\beta}<-1/3$ & \\
\midrule
\multirow{2}{*}{(6)} & \multirow{2}{*}{$\kappa_{\lambda}< 0$, $\kappa_{\mu}= 0$} & (a) & $\kappa_{\beta}<-1/3$, $\kappa_{\lambda}+\kappa_{\beta}=-1$, $\kappa_{\lambda}<-1/3$ & $ \lambda\beta/\mu $\\
& & (b) & $\kappa_{\beta}=0$, $\kappa_{\lambda}=-1$ & $\lambda\beta(\mu+\gamma )/(\mu(\beta+\mu+\gamma ))$\\
\midrule
(7) & $\kappa_{\lambda}= 0$, $\kappa_{\mu}> 0$ &  &  $\kappa_{\lambda}= 0$, $\kappa_{\mu}> 1/3$, $1+\kappa_{\beta}=\kappa_{\mu}$  & $\lambda\beta/\mu$ \\
\midrule
(8) & $\kappa_{\lambda}=0$, $\kappa_{\mu}< 0$  & & $\kappa_{\lambda}= 0$, $\kappa_{\mu}<-1/3$, $\kappa_{\mu}=-1$  & $\beta$ \\
\midrule
(9) & $\kappa_{\lambda}=0$, $\kappa_{\mu}= 0$  & & $\kappa_{\lambda}= 0$, $\kappa_{\mu}= 0$, $\kappa_{\beta}=-1$  & $\beta\lambda/(\lambda+\mu)$ \\
\bottomrule
\end{tabular}
\caption{\label{singletable}The different limiting cases, their parameter ranges and constraints, as well as the respective reproduction numbers.}
\end{sidewaystable}

\begin{sidewaystable}
\centering
\begin{tabular}{r@{\hspace{0.3em}}lr@{\hspace{0.3em}}lll} 
\toprule
\multicolumn{2}{c}{Parameter ranges} & \multicolumn{2}{c}{Parameter constraints} & $\gamma_i\bR_{0,i,j}$\\
\midrule
\multirow{3}{*}{(1)} & \multirow{3}{*}{$\kappa_{\lambda_{i,j}}> 0$, $\kappa_{\mu_{i,j}}> 0$} & (a) & $\kappa_{\lambda_{i,j}}-\kappa_{\mu_{i,j}}<-17/24$, $1+\kappa_{\lambda_{i,j}}+\kappa_{\beta_{i,j}}-\kappa_{\mu_{i,j}}=0$ & $\lambda_{i,j}\beta_{i,j}/\mu_{i,j}$\\
& & (b) & $\kappa_{\lambda_{i,j}}=\kappa_{\mu_{i,j}}$, $\kappa_{\beta_{i,j}}=-1$ & $\lambda_{i,j}\beta_{i,j}/(\lambda_{i,j}+\mu_{i,j})$\\
& & (c) & $-17/24+\kappa_{\lambda_{i,j}}>\kappa_{\mu_{i,j}}$, $\kappa_{\beta_{i,j}}=-1$ & $\beta_{i,j}$\\
\midrule
(2) & $\kappa_{\lambda_{i,j}}> 0$, $\kappa_{\mu_{i,j}}< 0$ & & $\kappa_{\beta_{i,j}}=-1$, $\kappa_{\mu_{i,j}}-\kappa_{\lambda_{i,j}}<-17/24$ & $\beta_{i,j}$\\
\midrule
(3) & $\kappa_{\lambda_{i,j}}> 0$, $\kappa_{\mu_{i,j}}= 0$ & & $\kappa_{\beta_{i,j}}=-1$ & $\beta_{i,j}$\\
\midrule
(4) & $\kappa_{\lambda_{i,j}}< 0$, $\kappa_{\mu_{i,j}}> 0$ & & $\kappa_{\lambda_{i,j}}<-41/24$, $1+\kappa_{\lambda_{i,j}}+\kappa_{\beta_{i,j}}-\kappa_{\mu_{i,j}}=0$, $\kappa_{\beta_{i,j}}<-41/24$ & $\lambda_{i,j}\beta_{i,j}/\mu_{i,j}$\\
\midrule
\multirow{4}{*}{(5)} & \multirow{4}{*}{$\kappa_{\lambda_{i,j}}< 0$, $\kappa_{\mu_{i,j}}< 0$} & (a) & $ \kappa_{\lambda_{i,j}}=\kappa_{\mu_{i,j}}<-17/24$, $\kappa_{\beta_{i,j}}=-1$ & $\beta_{i,j}\lambda_{i,j}/(\lambda_{i,j}+\mu_{i,j})$ \\
& & (b) & $-17/24>\kappa_{\lambda_{i,j}}>\kappa_{\mu_{i,j}}$, $\kappa_{\beta_{i,j}}=-1$ & $\beta_{i,j}\lambda_{i,j}/\lambda_{i,j}$\\
& & \multirow{2}{*}{(c)} & $0<\kappa_{\mu_{i,j}}-\kappa_{\lambda_{i,j}}<-17/24$, $1+\kappa_{\beta_{i,j}}+\kappa_{\lambda_{i,j}}-\kappa_{\mu_{i,j}}=0$ & \multirow{2}{*}{$\beta_{i,j}\lambda_{i,j}/\mu_{i,j}$}\\
& & & $\kappa_{\lambda_{i,j}}<-17/24$, $\kappa_{\beta_{i,j}}<-17/24$ & \\
\midrule
\multirow{2}{*}{(6)} & \multirow{2}{*}{$\kappa_{\lambda_{i,j}}< 0$, $\kappa_{\mu_{i,j}}= 0$} & (a) & $\kappa_{\beta_{i,j}}<-17/24$, $\kappa_{\lambda_{i,j}}+\kappa_{\beta_{i,j}}=-1$, $\kappa_{\lambda_{i,j}}<-17/24$ & $ \lambda_{i,j}\beta_{i,j}/\mu_{i,j} $\\
& & (b) & $\kappa_{\beta_{i,j}}=0$, $\kappa_{\lambda_{i,j}}=-1$ & $\frac{\lambda_{i,j}\beta_{i,j}(\mu_{i,j}+\gamma_i )}{\mu_{i,j}(\beta_{i,j}+\mu_{i,j}+\gamma_i ) }$\\
\midrule
(7) & $\kappa_{\lambda_{i,j}}= 0$, $\kappa_{\mu_{i,j}}> 0$ & & $\kappa_{\lambda_{i,j}}= 0$, $\kappa_{\mu_{i,j}}> \frac{17}{24}$, $1+\kappa_{\beta_{i,j}}=\kappa_{\mu_{i,j}}$ & $\lambda_{i,j}\beta_{i,j}/\mu_{i,j}$\\
\midrule
(8) & $\kappa_{\lambda_{i,j}}= 0$, $\kappa_{\mu_{i,j}}< 0$ & & $\kappa_{\lambda_{i,j}}= 0$, $\kappa_{\mu_{i,j}}< -\frac{17}{24}$, $\kappa_{\mu_{i,j}}=-1$ & $\beta_{i,j}$\\
\midrule
(9) & $\kappa_{\lambda_{i,j}}= 0$, $\kappa_{\mu_{i,j}}= 0$ & & $\kappa_{\lambda_{i,j}}= 0$, $\kappa_{\mu_{i,j}}= 0$, $\kappa_{\beta_{i,j}}=-1$  & $\frac{\beta_{i,j}\lambda_{i,j}}{\lambda_{i,j}+\mu_{i,j}}$\\
\bottomrule
\end{tabular}
\caption{\label{table:constraints}The different limiting cases, their parameter ranges and constraints, as well as the respective by-type reproduction numbers.}
\end{sidewaystable}

\end{document}